\documentclass[10pt]{amsart}
\usepackage{texmaX,semmaX,semtkzX}
 \usepackage{makecell}
\usepackage{fancyhdr}
\advance\textheight 0.1cm 

\fancypagestyle{newstyle}{
\fancyhf{} 
\fancyhead[CE]{\Small{B\MakeLowercase{est}, B\MakeLowercase{etts}, B\MakeLowercase{isatt}, B\MakeLowercase{ommel}, D\MakeLowercase{okchitser}, F\MakeLowercase{araggi}, K\MakeLowercase{unzweiler}, M\MakeLowercase{aistret}, M\MakeLowercase{organ}, M\MakeLowercase{uselli}, N\MakeLowercase{owell}}} 
\fancyhead[CO]{\fontsize{7}{10} \selectfont \MakeUppercase{A user's guide to the local arithmetic of hyperelliptic curves}}
\fancyfoot[C]{\thepage} 

}
\begin{document}

\title{A user's guide to the local arithmetic of hyperelliptic curves}
\author[B\MakeLowercase{est}, B\MakeLowercase{etts}, B\MakeLowercase{isatt}, B\MakeLowercase{ommel}, D\MakeLowercase{okchitser}, F\MakeLowercase{araggi}, K\MakeLowercase{unzweiler}, M\MakeLowercase{aistret}, M\MakeLowercase{organ}, M\MakeLowercase{uselli}, N\MakeLowercase{owell}]{Alex J. Best, L.~Alexander Betts, Matthew Bisatt, Raymond van Bommel, Vladimir Dokchitser, Omri Faraggi, Sabrina Kunzweiler, C\'eline Maistret, Adam Morgan, Simone Muselli, Sarah Nowell}

\address{Mathematics and Statistics Department, Boston University, Boston, MA 02215, USA}
\email{alex.j.best@gmail.com} 

\address{Mathematics Department, Harvard University, Cambridge, MA 02138, USA}
\email{abetts@math.harvard.edu}

\address{Department of Mathematics, MIT, Cambridge, MA 02139, USA} 
\email{bommel@mit.edu} 

\address{School of Mathematics, University of Bristol, Bristol, BS8 1UG, UK} 
\email{matthew.bisatt@bristol.ac.uk}
\email{celine.maistret@bristol.ac.uk} 
\email{simone.muselli@bristol.ac.uk}

\address{Department of Mathematics, University College London, London WC1H 0AY, UK}
\email{v.dokchitser@ucl.ac.uk}
\email{omri.faraggi.17@ucl.ac.uk}
\email{sarah.nowell.17@ucl.ac.uk}

\address{Max-Planck-Institut f\"ur Mathematik, 53111 Bonn, Germany}
 \email{am516@mpim-bonn.mpg.de}


\address{Institut f\"ur Algebra und Zahlentheorie, Universit\"at Ulm, 89081 Ulm, Germany}
\email{sabrina.kunzweiler@uni-ulm.de}


 


%

\begin{abstract}
A new approach has been recently developed to study the arithmetic of hyperelliptic curves $y^2=f(x)$ over local fields of odd residue characteristic 
via combinatorial data associated to the roots of~$f$. Since its introduction, numerous papers have used this machinery of ``cluster pictures'' to compute a plethora of 
arithmetic invariants associated to these curves. The purpose of this user's guide is to summarise and centralise all of these 
results in a self-contained fashion, complemented by an abundance of examples.
\end{abstract}

\maketitle

\setcounter{tocdepth}{1}

\tableofcontents
\newpage

\section{Introduction}\label{se:intro}
In this paper, we provide a summary of a recently developed approach to understanding the local arithmetic of hyperelliptic curves. This approach revolves around the theory of ``clusters'', and enables one to read off many local arithmetic invariants of hyperelliptic curves from explicit equations $y^2=f(x)$. The paper is meant to serve as a \emph{user's} guide: our aim has been to make it accessible to mathematicians interested in applications outside of local arithmetic geometry, or who may wish to compute local invariants without having to decipher the theoretical background.

Throughout this article, $K$ will be a local field of odd residue characteristic $p$ and $C/K$ a hyperelliptic curve given by 
$$
y^2 = f(x) = c \prod_{r \in \cR}(x-r),
$$
where $f \in K[x]$ is separable, $\deg(f) = 2g+1$ or $2g+2$ and $g\ge2$.
\subsection{How to use this guide.}

The article is structured as follows. We begin in Section \ref{se:notation} 
by declaring some general notation which will be used throughout, and proceed to give some background theory on cluster pictures and BY trees in Sections \ref{se:clusters} and \ref{se:BYtree} respectively. Cluster pictures will be critical background for all sections of the article; BY trees will be used in Sections \ref{se:tama}, \ref{sec:minimaldiscriminant}, \ref{se:iso} and the Appendix. 

From there on, each section will be self-contained and independent of the other sections.
This will allow a reader who is concerned with just one topic (Galois representations, say) to be able to learn everything they need by reading just the background theory in Sections \ref{se:clusters} and \ref{se:BYtree} and the relevant section (in our example, Section~\ref{se:GalRep}). 

From Section \ref{se:redtype} 
onwards, each section will consist of two pages: the first stating the relevant theorems, and the second providing examples illustrating the theorems. None of the theorems are original (apart from Theorem \ref{thm:discriminant_by}, 
 whose proof is given in the Appendix) and we give no proofs; each section has references at the end where the interested reader can find proofs and more general statements of the theorems.

\subsection{Related work}

The key references for the present work are \cite{m2d2, Betts, BisattRN, hyble, FN, Kunzweiler, Muselli}. We have made a blanket assumption that $K$ is a local field; this is often unnecessarily restrictive, and many results hold for complete discretely valued fields. The reference \cite{m2d2} also discusses a number of topics that we have omitted, in particular how to use clusters to check whether a curve is deficient, how one may perturb $f(x)$ without changing the standard local invariants, and how to classify semistable hyperelliptic curves in a given genus.

As many of our examples will illustrate, the method of cluster pictures is very convenient for computations. However, it can also be used for more theoretical purposes: for instance, one can work explicitly with families of hyperelliptic curves for which the genus becomes arbitrarily large (see e.g. \cite{AD, Co}), or prove general results for curves of a given genus by a complete case-by-case analysis of cluster pictures (see e.g. \cite{asparity}). 

We would like to mention some alternative techniques that have been recently developed for investigating similar topics. 
In \cite{Rut, BW, OW, OS, KW}, the authors determine different kinds of models, the conductor exponent, the local $L$-factor, compare the Artin conductor to the minimal discriminant and compute a basis of the integral differentials.  
In arbitrary residue characteristic (including $2$), but under some technical assumptions, \cite{D, Muselli, DD} determine the minimal regular model with normal crossings, a basis of integral differentials, reduction types, conductor and action of the inertia group on the $\ell$-adic representation. 

\subsection{Implementation}
We have implemented many of the methods described in this guide as a package using the SageMath computer algebra system \cite{sage}. The package is available online at \cite{package}.
This package includes implementations of cluster pictures and BY trees as abstract objects, which it can also plot. Given a hyperelliptic curve, the implementation determines its associated cluster picture and BY tree. It also determines the Tamagawa number, root number, reduction type, minimal discriminant and dual graph of the minimal regular model, as described in this article.

We have also computed cluster pictures for all elliptic curves over $\Q$ and number fields, and all genus 2 curves present in the L-Functions and Modular Forms Database \cite{lmfdb}. The latter is incorporated in the LMFDB homepages of curves.

\subsection{Acknowledgements}
We wish to thank Edgar Costa for technical assistance and ICERM for hosting the virtual workshop `Arithmetic Geometry, Number Theory, and Computation' where this work was done. We also wish to warmly thank the anonymous referee for very carefully reading the manuscript and for their numerous constructive comments.   

Alex J. Best, C\'eline Maistret and Raymond van Bommel were supported by the Simons Collaboration on Arithmetic Geometry, Number Theory, and Computation (Simons Foundation grant numbers 550023, 550023 and 550033, respectively).
L. Alexander Betts and Adam Morgan were supported by the Max Planck Institute in Bonn.
Matthew Bisatt was supported by an EPSRC Doctoral Prize fellowship.
Vladimir Dokchitser was supported by a Royal Society University Research Fellowship.
Omri Faraggi and Sarah Nowell were supported by the Engineering and Physical Sciences Research Council [EP/L015234/1], the EPSRC Centre for Doctoral Training in Geometry and Number Theory (The London School of Geometry and Number Theory), University College London.

\newpage
\section{Notation}\label{se:notation}
Here we set out the notation that will be used throughout the paper. 

Formally by a hyperelliptic curve $C$ we mean the smooth projective curve associated to $y^2 = f(x)$, equivalently the glueing of the pair of affine patches 
$$
y^2 = f(x) \quad \text{ and } \quad v^2 = t^{2g+2}f\Big(\frac{1}{t}\Big)
$$ 
along the maps $x = \frac{1}{t}$ and $y = \frac{v}{t^{g+1}}$, where $f\in K[x]$ is separable, and $\deg(f) \ge 5$. 
We will not consider double covers of general conics. 

We fix the following notation associated to fields and hyperelliptic curves.
\begin{table}[H]
\begin{tabular}{ll}
$K$      & local field of odd residue characteristic $p$\cr
$\cO_K$ & ring of integers of $K$\cr
$k$      &  residue field of $K$\cr
$\pi$ & uniformiser of $K$ \cr
$v$     & normalised valuation with respect to $K$ so that $v(\pi)=1$ \cr
$\Kbar$ & algebraic closure of $K$ \cr
$\Ks$ &  separable closure of $K$ inside $\Kbar$\cr
$\Knr$      & maximal unramified extension of $K$ inside $\Ks$ \cr
$\bar{k}$ & algebraic closure of $k$ and residue field of $\Knr$ \cr
$G_K$ & the absolute Galois group $\Gal(\Ks/K)$ \cr
$I_K$ & inertia subgroup of $G_K$ \cr
$\Frob$ &a choice of (arithmetic) Frobenius element in $G_K$\cr
$\bar{x}$ or $x \mmod \mathfrak{m}$ & image in the residue field $\bar{k}$ for $x \in \bar{K}$ with $v(x) \ge 0$\cr
&\cr
%
$C$ & hyperelliptic curve given by $y^2 = f(x)$ \cr
$c$ & leading coefficient of $f(x)$ \cr
$\cR$ & set of roots of $f(x)$ in $\Ks$ \cr
$g$ & genus of $C$ \cr
$\Cmin$ & minimal regular model of $C/\cO_{K^{\text{nr}}}$ \cr
$\Cmink$ & special fibre of $\Cmin$ \cr
Jac $C$ & Jacobian of $C$ \cr
\phantom{$\bar{x}$ or $x \mmod \mathfrak{m}$ }& \phantom{image in the residue field for $x \in \bar{K}$ with $v(x) \ge 0$}\cr
\end{tabular}
\end{table}
We will say $C$ is semistable if $C$ has semistable reduction. Similarly $C$ is tame if $C$ acquires semistable reduction over a tame extension of $K$. If $p>2g+1$, $C$ is always tame, see Remark \ref{tame remark}.

\newpage
\section{Clusters}\label{se:clusters}

\begin{definition}[Clusters and cluster pictures]
\label{def:cluster}
A {\em cluster} is a non-empty subset $\c\subseteq\cR$ of the form $\c = D \cap \cR$ for some disc $D=\{x\!\in\! \Kbar \mid v(x-z)\!\geq\! d\}$ for some $z\in \Kbar$ and $d\in \Q$. 

For a cluster $\s$ with $|\s|>1$, its {\em depth} $d_\s$ is the maximal $d$ for which $\s$ is cut out by such a disc, that is $d_\s\! =\! \min_{r,r' \in \mathfrak{s}} v(r\!-\!r')$. If moreover $\s\neq \cR$, then its {\em relative depth} is $\delta_\s\! =\! d_\s\! -\!d_{P(\s)}$, where $P(\s)$ is the smallest cluster with $\s\subsetneq P(\s)$ (the \emph{parent} cluster).

We refer to this data as the {\em cluster picture} of $C$.
\end{definition}

\begin{remark}\label{CLrem1}
The Galois group acts on clusters via its action on the roots. It preserves depths and containments of clusters. 
\end{remark}
\begin{notation}\label{CLnotation}
We draw cluster pictures by drawing roots $r \in \cR$ as \smash{\raise4pt\hbox{\clusterpicture\Root[A]{1}{first}{r1};\endclusterpicture}}, and draw ovals around roots to represent clusters (of size $>1$), such as:
$$
\scalebox{1}{\clusterpicture            
  \Root[A] {1} {first} {r1};
  \Root[A] {} {r1} {r2};
  \Root[A] {} {r2} {r3};
  \Root[A] {4} {r3} {r4};
  \Root[A] {1} {r4} {r5};
  \Root[A] {} {r5} {r6};
  \ClusterLD c1[][{2}] = (r1)(r2)(r3);
  \ClusterLD c2[][{2}] = (r5)(r6);
  \ClusterLD c3[][{1}] = (r4)(c2);
  \ClusterLD c4[][{0}] = (c1)(c2)(c3);
\endclusterpicture}
$$
The subscript on the largest cluster $\cR$ is its depth, while the subscripts on the other clusters are their relative depths.
\end{notation}

\begin{notation}
For a cluster $\s$ we use the following terminology.
\renewcommand{\arraystretch}{1}%
\begin{table}[H]  \label[table]{clusternotation}
\begin{tabular}{ll}
size of $\s$ & $|\s|$\\
$\s'$ a child of $\s$, $\s'\!<\!\s$ & $\s'$ is a maximal subcluster of $\s$
 \\
parent of $\s$, $P(\s)$ & $P(\s)$ is the smallest cluster with $\s \subsetneq P(\s)$
\\
singleton & cluster of size 1\\
proper cluster & cluster of size $>$ 1 
\\
even cluster &cluster of even size
\\
odd cluster  &cluster of odd size
\\
\"ubereven cluster  &even cluster all of whose children are even
\\
twin  &cluster of size 2
\\
cotwin & non-\ub\ cluster with a child of size $2g$
\\
principal cluster $\s$& if $|\s| \ne 2g+2$: $\s$ is proper, not a twin or a cotwin;\\
& if $|\c| = 2g+2$: $\s$ has $\ge 3$ children\\
 $\neck{\c}$& if $\s$ is not a cotwin: \\
 &smallest $\neck{\c} \supseteq \c$ that  
does not have an \ub\ parent;\\
& if $\s$ is a cotwin: the child of $\s$ of size $2g$\\
 $\c \wedge \c'$ & smallest cluster containing $\c$ and $\c'$
\\ 
$\so$ & set of odd children of  $\s$\\
centre $z_\s$ & a choice of $z_\s\in \Ks$ with $\min_{r\in\s}v(z_\s-r)=d_\s$\\
$\theta_\c$& a choice of $ \sqrt{c\prod\nolimits_{r \notin \c} (z_\c-r)}$\\
$\epsilon_\s$  & $\epsilon_\s\colon G_K\!\to\! \{\pm1\}$, $ \epsilon_\s(\sigma)\!=\!\frac{\sigma(\theta_{\c^*})}{\theta_{\neck{(\sigma\s)}}} \mod \m$ 
if $\s$ even or a cotwin,\\
& $\epsilon_\s =0$ otherwise \\
$\nu_\s$ & $=v(c)+|\s|d_\s+\sum_{r\notin \s}d_{\{r\}\wedge \s}$, for a proper cluster $\c$\\
$\tilde\lambda_{\s}$ & $=\frac{1}{2}(v(c)+|\tilde{\s}|d_{\s}+ \sum_{r \not\in \s}d_{ \{r\}\wedge\s})$, for a proper cluster $\c$\\
\end{tabular}
\label{clustertable}
\end{table}
\renewcommand{\arraystretch}{1}%
\end{notation}

\begin{remark}\label{CLrem2}
For even clusters and cotwins, $\epsilon_\s$ does not depend on the choice of centre of $\c$. When restricted to the stabiliser of $\s$, it is a homomorphism and does not depend on the choice of square root of $\theta_\s^2$.  
\end{remark}

\newpage
\begin{example}\label{CLex1}
Consider $C : y^2=(x^2+7^2)(x^2-7^{15})(x-7^6)(x-7^6-7^9)$ over $\Q_7$. Its cluster picture is
\vskip -16pt
$$
 \clusterpicture            
  \Root[A] {} {first} {r1};
  \Root[A] {} {r1} {r2};
  \Root[A] {2} {r2} {r3};
  \Root[A] {} {r3} {r4};
  \Root[A] {4} {r4} {r5};
  \Root[A] {} {r5}{r6};
  \ClusterLDName c1[][\frac 32][\t_1] = (r3)(r4);
  \ClusterLDName c2[][3][\t_2] = (r5)(r6);
  \ClusterLDName c3[][5][\mathfrak{a}] = (c1)(c2);
  \ClusterLDName c4[][1][\cR] = (r1)(r2)(c3);
\endclusterpicture, 
\qquad \text{with }
\cR = \{7i, -7i, 7^{\frac{15}{2}}, -7^{\frac{15}{2}}, 7^6, 7^6+7^9\}, \text{ where } i^2=-1.
$$
$\bullet$ \underline{Depths and relative depths}: For each pair of roots $r,r'$ in the picture, 
$
v(r\!-\!r')\ge 1
$, and $v(7i-7^6)=1$ so that $d_{\cR} =1$. Similarly $\mathfrak{a}=\{7^{\frac{15}{2}}, -7^{\frac{15}{2}}, 7^6, 7^6+7^9\}$ is a cluster of depth $d_{\mathfrak{a}}=6$ and therefore relative depth $\delta_{\mathfrak{a}}=5$. Finally, $\t_1 = \{7^{\frac{15}{2}}, -7^{\frac{15}{2}}\}$ has depth $d_{\t_1} = \frac{15}{2}$ and $\t_2 =\{7^6, 7^6+7^9\}$ has depth $d_{\t_2} = 9$. The only other clusters are singletons hence are not assigned any depth. 

 \noindent$\bullet$ \underline{Children}: 
The children of $\cR$ are $\{7i\},\{-7i\}$ and $\mathfrak{a}$, so $\tilde{\cR} = \{\{7i\},\{-7i\}\}$. The children of $\mathfrak{a}$ are $\t_1$ and $\t_2$, so $\tilde{\mathfrak{a}}$ is empty.

\noindent $\bullet$ \underline{Types}:
$\cR, \mathfrak{a}, \t_1, \t_2$ are proper and even. The only odd clusters are singletons. Both $\t_1$ and $\t_2$ are twins, $\mathfrak{a}$ is \ub\ and $\cR$ is a cotwin. The only principal cluster is~$\mathfrak{a}$. 

\noindent $\bullet$  \underline{$\s^*$ and $\s_1\wedge \s_2$}: 
$\t_1^*=\t_2^* = \mathfrak{a}^* = \cR^* = \mathfrak{a}$, $\t_1 \wedge \t_2 = \mathfrak{a}$, $\t_1 \wedge \mathfrak{a} = \mathfrak{a}$ and $\t_1 \wedge \{7i\} = \cR$. 

\noindent $\bullet$ \underline{$z_\s$ and $\epsilon_\s$}: 
Pick $z_{\cR} = z_{\mathfrak{a}} =z_{\t_1} = 0$ and $z_{\t_2} = 7^6$. As $\t_1^*=\t_2^* = \mathfrak{a}^* = \cR^* = \mathfrak{a}$, we get $\epsilon_{\t_1} =\epsilon_{\t_2} = \epsilon_{\cR}=\epsilon_{\mathfrak{a}}$. With our choice of $z_{\mathfrak{a}}$ we obtain $\theta_{\mathfrak{a}} = \sqrt{(0-7i)(0+7i)} = \pm 7$. Say we choose $\theta_{\mathfrak{a}}=7$, then for any $\sigma \in G_K$ we have $\epsilon_{\mathfrak{a}}(\sigma)= \frac{\sigma(\theta_{\mathfrak{a}})}{\theta_{\sigma\mathfrak{a}}} = +1$. 
\end{example}
\medskip
\begin{example}\label{CLex2}
Suppose $C/\Q_p: y^2 = f(x)$ with $f(x) \in \Z_p[x]$ monic. Suppose also that $f(x) \mmod p$ has at least two distinct roots, equivalently $d_\cR =0$. Consider the reduction $\bar{C}/\F_p: y^2=\bar{f}(x)$.

(i) A child of $\cR$ consists of roots that have the same image in the residue field. For example if $p=5$ and $\cR =\{0,1,2,3,5,8,13\}$ we have the cluster picture
 \clusterpicture            
  \Root[A] {} {first} {r1};
  \Root[A] {} {r1} {r2};
  \Root[A] {2} {r2} {r3};
  \Root[A] {} {r3} {r4};
  \Root[A] {3} {r4} {r5};
  \Root[A] {} {r5}{r6};
  \Root[A] {} {r6}{r7};
  \ClusterLDName c1[][][] = (r3)(r4);
   \ClusterLDName c2[][][] = (r5)(r6)(r7);
  \ClusterLDName c4[][0][] = (r1)(r2)(c1)(c2);
\endclusterpicture and $\bar{C} : y^2 = x^2(x-1)(x-2)(x-3)^3$.

(ii) If $f(x) \mmod p$ has a double root and no other repeated roots then the cluster picture has a twin $\t$ and $\bar{C}$ has a node. Generally, for semistable curves, twins contribute nodes to the special fibre of the stable model. \hfill \clusterpicture            
  \Root[A] {} {first} {r1};
  \Root[A] {} {r1} {r2};
  \Root[A] {2} {r2} {r3};
  \Root[A] {} {r3} {r4};
  \Root[A] {} {r4} {r5};
  \Root[Dot] {} {r5}{r6};
  \Root[Dot] {} {r6}{r7};
  \Root[Dot] {} {r7}{r8};
   \Root[A] {} {r8}{r9};
  \ClusterLDName c1[][][\t] = (r1)(r2);
  \ClusterLDName c4[][0][] = (c1)(r3)(r4)(r5)(r6)(r7)(r8)(r9);
\endclusterpicture 

(iii) The normalisation of $\bar{C}$ is obtained by removing the maximal square factor in $\bar{f}(x)$, so the new roots are in 1:1 correspondence with the odd clusters. Explicitly, it is the hyperelliptic curve given by 
$
y^2=\prod_{\s \in \tilde{\cR}} (x-\bar{z}_{\s})
$.
For example, for the curve in (i), the normalisation is given by $y^2 = (x-1)(x-2)(x-3)$.

(iv) When $\cR$ is \ub\ the normalisation of $\bar{C}$ is $y^2 =1$, which is a union of two lines. Generally, for semistable curves, \ub\ clusters contribute pairs of $\mathbb{P}^1$s to the special fibre of both semistable and regular models of $C/\Q_p^{\text{nr}}$. 

(v) Suppose that $\cR =\{1,2,p,2p,3p,4p\}$ so the cluster picture is \clusterpicture            
  \Root[A] {} {first} {r1};
  \Root[A] {} {r1} {r2};
  \Root[A] {2} {r2} {r3};
  \Root[A] {} {r3} {r4};
  \Root[A] {} {r4} {r5};
  \Root[A] {} {r5}{r6};
  \ClusterLDName c1[][1][] = (r3)(r4)(r5)(r6);
  \ClusterLDName c4[][0][] = (r1)(r2)(c1);
\endclusterpicture\ for $p>3$. Applying the change of variable $x'=\frac{1}{x}$ gives a curve whose cluster picture is
\clusterpicture            
  \Root[A] {} {first} {r1};
  \Root[A] {} {r1} {r2};
  \Root[A] {2} {r2} {r3};
  \Root[A] {} {r3} {r4};
  \Root[A] {} {r4} {r5};
  \Root[A] {} {r5}{r6};
  \ClusterLDName c1[][1][] = (r1)(r2);
  \ClusterLDName c4[][-1][] = (c1)(r3)(r4)(r5)(r6);
\endclusterpicture. Generally, changing the model can convert twins to cotwins and vice versa, and the number of twins plus cotwins is model independent. 

(vi) For a curve as in (ii), the node on $\bar{C}$ is split if and only if $\prod_{r\notin \t} (\bar{z}_t-\bar{r})$ is a square in $\F_p$. Equivalently if and only if $\epsilon_{\t}(\Frob) =+1$. Generally, $\epsilon$ keeps track of whether the nodes are split or non split and similar data. 
\end{example}
\newpage
\begin{example}\label{ex:nuR}
Let $C/\Q_p\!:\! y^2 \!=\!f(x)$ with $f(x)\!\in \!\Z_p[x]$, $\deg(f)\!=\!8$ and ~$v(c) \!\ge \!0$. 

(i) Suppose that $f(x) \mod p$ has distinct roots, equivalently that the cluster picture of $C$ is 
\clusterpicture            
  \Root[A] {} {first} {r1};
  \Root[A] {} {r1} {r2};
  \Root[A] {} {r2} {r3};
  \Root[A] {} {r3} {r4};
  \Root[A] {} {r4} {r5};
  \Root[A] {} {r5}{r6};
  \Root[A] {} {r6}{r7};
  \Root[A] {} {r7}{r8};
  \ClusterLDName c1[][0][] = (r1)(r2)(r3)(r4)(r5)(r6)(r7)(r8);
 \endclusterpicture. In this case $\nu_\cR = v(c)$. Here when $\nu_{\cR}$ is even $C$ has good reduction and when $\nu_\cR$ is odd it is a quadratic twist of such a curve.
 
 (ii) Suppose that $f(x) \mod p$ has a repeated root of multiplicity 5 and the corresponding roots in $\bar{\Q}_p$ are equidistant with distance $v(r_i-r_j)=n$, equivalently the cluster picture of $C$ is 
 \clusterpicture            
  \Root[A] {} {first} {r1};
  \Root[A] {} {r1} {r2};
  \Root[A] {} {r2} {r3};
  \Root[A] {2} {r3} {r4};
  \Root[A] {} {r4} {r5};
  \Root[A] {} {r5}{r6};
  \Root[A] {} {r6}{r7};
  \Root[A] {} {r7}{r8};
  \ClusterLDName c1[][n][\s] = (r4)(r5)(r6)(r7)(r8);
  \ClusterLDName c2[][0][] = (r1)(r2)(r3)(c1);
 \endclusterpicture. 
 
 The substitution $x' = \frac{x-z_\s}{p^n}$ gives 
 $
 f(x) = c\prod_{r\in \cR}(p^nx'\!+\!z_\s\!-\!r).
 $
 Observe that $v(z_\s\!-\!r)=d_\s =n$ for $r\in \s$ and $v(z_\s\!-\!r)= d_{\{r\} \wedge \s} =0$ otherwise. The equation for $C$ becomes 
 $$
 y^2 = cp^{5n}\prod_{r\notin \s}(p^nx'\!+\!z_\s\!-\!r)\prod_{r\in \s}(x'\!+\!\frac{z_\s\!-\!r}{p^n}).
 $$ 
 Note that $v(cp^{5n})$ is precisely $v(c) + |\s|d_\s + \sum_{r \notin \s} d_{\{r\}\wedge \s} = \nu_\s$, and by construction each factor of the above polynomial has integral coefficients. 
 
 In general for any proper cluster $\s$ the above change of variable will give an integral equation for $C$ of the form $cp^mh(x)$ where $v(c)\!+\!m\!=\!\nu_\s$ and $h(x)$ is integral. When $n \in \Z$, $z_\s \in \Q_p$ and $\nu_\s \in 2\Z$ the substitution $y = y'p^{\frac{\nu_\s}{2}}$ gives an equation for $C/\Q_p$ whose reduction is of the form
 $$
 y^2=(\text{constant}) \prod_{r\in \s}(x-r').
 $$
 When $\s$ is principal, this is a curve over $\F_p$ of genus at least 1.
 \end{example}

\medskip
 \begin{example}\label{ex:lambdatilde}
Consider the two curves $C_1 : y^2 = x^6-p$ and $C_2 : y^2=x(x^5-p)$. These have cluster picture \clusterpicture            
  \Root[A] {} {first} {r1};
  \Root[A] {} {r1} {r2};
  \Root[A] {} {r2} {r3};
  \Root[A] {} {r3} {r4};
  \Root[A] {} {r4} {r5};
  \Root[A] {} {r5}{r6};
  \ClusterLDName c1[][n][] = (r1)(r2)(r3)(r4)(r5)(r6);
 \endclusterpicture, where $n\!=\!\frac16$ for $C_1$ and $n\!=\!\frac15$ for $C_2$. These curves have $2\tilde{\lambda}_\cR =v(c)+|\tilde{\cR}|d_\cR+\sum_{r \notin \cR} d_{\{r\}\wedge \cR} =6n$. The denominator of $2\tilde{\lambda}_\cR$ is either $1$ or $5$. This reflects the different inertia action on the roots: it has no fixed points for $C_1$ and one fixed point for $C_2$.
 
The general case is more subtle. Roughly, for a proper cluster $\s$, the denominator of $\tilde{\lambda}_\s$ is related to the inertia action on $\tilde{\s}$ and to the inertia action by geometric automorphisms on the reduced curve associated to $\s$ \`a la Example \ref{ex:nuR}.
\end{example}

\newpage
\newpage
\section{BY trees}\label{se:BYtree}

\begin{definition}[BY tree]
\label{defbytree}
A \emph{BY tree} is a finite tree $T$ with a \emph{genus} function $g\colon V(T)\to\Z_{\ge 0}$ on vertices, a length function $\delta\colon E(T) \to \R_{>0}$  
on edges, and a 2-colouring blue/yellow on vertices and edges such that
\begin{enumerate}
\item 
yellow vertices have genus 0, degree $\ge 3$, and only yellow incident edges; 
\item
blue vertices of genus 0 have at least one yellow incident edge;
\item
at every vertex, $2g(v)+2\ge$ \# blue incident edges at $v$.
\end{enumerate}
Note that all leaves (vertices of degree 1) are necessarily blue.
\end{definition}

\begin{notation}
In diagrams, yellow edges are drawn squiggly (\begin{tikzpicture}[scale=\GraphScale]
  \BlueVertices
  \InfVertices
  \Vertex[x=1.50,y=0.000,L=\InfLabel]{1};
  \InfVertices
  \Vertex[x=0.000,y=0.000,L=\InfLabel]{2}
  \YellowEdges
  \Edge(1)(2)
\end{tikzpicture}) and yellow vertices 
hollow (\smash{\raise-0.5pt\hbox{\begin{tikzpicture}[scale=\GraphScale]
  \YellowVertices
  \VertexLN[x=3.00,y=0.000,L=\relax]{2}{};
 \end{tikzpicture}}}) for the benefit of viewing them in black and white. We write the genus of a blue vertex inside the vertex (\smash{\raise-1pt\hbox{\begin{tikzpicture}[scale=\GraphScale]
  \BlueVertices
  \VertexLN[x=3.00,y=0.000,L=2]{2}{};
 \end{tikzpicture}}}); we omit it for blue vertices with genus 0. We write the length of edges next to them.
\end{notation}

\begin{definition}
\label{StoTC}
 The BY tree $T_C$ associated to $C$ is given by: 
\begin{itemize}
\item 
one vertex $v_\s$ for every proper cluster $\s$, coloured yellow
if $\s$ is \ub\ and blue otherwise,
\item 
for every pair $\c'<\s$ with $\c'$ proper, link $v_{\c'}$ and $v_\s$ with an edge, yellow of length $2\delta_{\c'}$ if $\c'$ is even and blue of length $\delta_{\c'}$ if $\c'$ is odd,
\item
if $\cR$ has size $2g+2$ and is a union of two proper children, remove $v_{\cR}$ and merge the two remaining edges (adding their lenghts), 
\item
if $\cR$ has size $2g+2$ and has a child $\c$ of size $2g+1$, remove $v_{\cR}$ and the edge between $v_{\cR}$ and $v_\c$,
\item the genus $g(v_\c)$ of a blue vertex $v_\c$ is defined so that $|\tilde{\s}| = 2g(v_\s)+2$ or $2g(v_\s)+1$. 
\end{itemize}
\end{definition}

\begin{definition}
\label{autb}
An \emph{isomorphism of BY trees} $T\to T'$ is a pair $(\alpha, \epsilon)$ where 
\begin{itemize}
\item 
$\alpha$ is a graph isomorphism $T\to T'$ 
that preserves edge lengths, genera of vertices and colours, and
\item
for every connected component $Y$ of 
the yellow part $T_y\subset T$, $\epsilon(Y)\in \{\pm 1\}$. 
\end{itemize}
Equivalently, 
$\epsilon$ is a collection of signs $\epsilon(v)\in \{\pm 1\}$ and $\epsilon(e)\in \{\pm 1\}$ for every 
yellow vertex and yellow edge, such that $\epsilon(v)=\epsilon(e)$ whenever $e$ 
ends at $v$.
Isomorphisms are composed by the cocycle rule
$$
  (\alpha, \epsilon_\alpha)\circ (\beta, \epsilon_\beta) = 
    \bigl(\alpha\circ\beta, \bullet\mapsto\epsilon_\beta(\bullet)\epsilon_\alpha(\beta(\bullet))\bigr).
$$
An \emph{automorphism} of $T$ is an isomorphism from $T$ to itself. 
\end{definition}

\begin{definition}\label{actionb}
The induced action of $G_K$ is given by 
$
\sigma \mapsto (\alpha_{\sigma}, \epsilon_{\sigma}) \in \Aut T_C
$
with
$
\alpha_{\sigma}(v_\c) = v_{\sigma (\c)}
$
for all vertices $v_{\s}$, and $ \epsilon_\sigma(Y) =\epsilon_{\s_Y}(\sigma)$ for yellow components~$Y$. Here the cluster $\s_Y$ is taken so that $v_{\s_Y}$ is any vertex in the closure of $Y$, 
other than the maximal one among these clusters. Note that $\epsilon_{\s_Y}(\sigma)$ depends on the choices of square roots of~$\theta^2$.
\end{definition}
\begin{notation}
We draw arrows between edges and signs above yellow components to represent automorphisms. 
\end{notation}
\begin{remark}\label{inertiatrivial}
For semistable curves, inertia maps to the identity in $\Aut(T_C)$, that is $\alpha_\sigma = \id$ and $\epsilon_{\sigma}(Y) = +1$ for all $\sigma \in I_K$ and all yellow components $Y$. 
\end{remark}
\begin{lemma}\label{genuscurve}
The genus of the curve satisfies
$$
g= \#(\text{connected components of the blue part of } T_C) -1 + \sum_{v \in V(T_C)} g(v). 
$$
\end{lemma}
\newpage
\begin{example}\label{BYex1}
 Consider the cluster picture from Example \ref{CLex1}. 
There are 4 proper clusters $\cR, \mathfrak{a}, \t_1$ and $\t_2$ so the BY tree has 4 vertices $v_\cR, v_{\mathfrak{a}}, v_{\t_1}, v_{\t_2}$, where only $v_{\mathfrak{a}}$ is yellow since $\mathfrak{a}$ is \ub. There are 3 yellow edges corresponding to the three even children $\mathfrak{a}\!< \!\cR, \t_1 \!<\!\mathfrak{a}, \t_2 \!<\! \mathfrak{a}$, of length $2\!\times \!5, 2 \!\times \!\frac{3}{2}, 2 \!\times\! 3$ respectively.
\vskip 6pt
$$\smash{\raise0pt\hbox{\clusterpicture            
  \Root[A] {} {first} {r1};
  \Root[A] {} {r1} {r2};
  \Root[A] {2} {r2} {r3};
  \Root[A] {} {r3} {r4};
  \Root[A] {4} {r4} {r5};
  \Root[A] {} {r5}{r6};
  \ClusterLDName c1[][\frac 32][\t_1] = (r3)(r4);
  \ClusterLDName c2[][3][\t_2] = (r5)(r6);
  \ClusterLDName c3[][5][\mathfrak{a}] = (c1)(c2);
  \ClusterLDName c4[][1][\cR] = (r1)(r2)(c3);
\endclusterpicture }}
\qquad\Longrightarrow\qquad
\smash{\raise-11pt\hbox{\begin{tikzpicture}[scale=\GraphScale]
  \BlueVertices
  \VertexLN[x=3.00,y=0.000,L=\relax]{2}{};
  \VertexLE[x=0.000,y=0.000,L=\relax]{3}{};
  \VertexLW[x=1.50,y=1.00,L=\relax]{4}{};
  \YellowVertices
  \Vertex[x=1.50,y=0.000,L=\relax]{1}
  \YellowEdges
  \Edge(1)(3)
  \Edge(1)(2)
  \Edge(1)(4)
\TreeEdgeSignS(1)(2){0.5}{6}
\TreeEdgeSignS(1)(3){0.5}{3}
\TreeEdgeSignW(1)(4){0.5}{10}
\end{tikzpicture}
}}
$$
\end{example}
\vskip 12pt
\begin{remark}\label{BYrem1}
(i) The depth $d_{\cR}$ is not relevant for the BY tree.\\
(ii) The yellow part forms an open subset (since yellow vertices correspond to \ub\ clusters, which are even and only have even children). \\
(iii) One can reconstruct the cluster picture from the BY tree and $d_{\cR}$, provided that there is a vertex $v_{\cR}$ and it is identified. 
\end{remark}

\begin{example}\label{BYex2}
Consider the curve $C/\Q_{11}$ given by $y^2 = f(x)$ with $f(x)$ monic with set of roots 
$$
\cR=\{0,1,2,\>\>\zeta_7\!-\!11, \zeta_7\!+\!11,\>\> \zeta_7^2\!-\!11, \zeta_7^2\!+\!11, \>\>\zeta_7^4\!-\!11, \zeta_7^4\!+\!11\},
$$
where $\zeta_7^7=1$ and $\zeta_7^3+5\zeta_7^2+4\zeta_7 + 10\equiv 0 \mmod 11$. Its cluster picture and BY tree are
$$\smash{\raise 0pt\hbox{\clusterpicture            
  \Root[A] {} {first} {r1};
  \Root[A] {} {r1} {r2};
  \Root[A] {} {r2} {r3};
  \Root[A] {4} {r3} {r4};
  \Root[A] {} {r4} {r5};
  \Root[A] {4} {r5}{r6};
  \Root[A] {} {r6}{r7};
   \Root[A] {4} {r7}{r8};
   \Root[A] {} {r8}{r9};
  \ClusterLDName c1[][1][\t_1] = (r4)(r5);
  \ClusterLDName c2[][1][\t_2] = (r6)(r7);
  \ClusterLDName c3[][1][\t_4] = (r8)(r9);
  \ClusterLDName c4[][0][\cR] = (r1)(r2)(r3)(c1)(c2)(c3);
\endclusterpicture }}
\qquad \text{and} \qquad
\smash{\raise -16pt\hbox{\begin{tikzpicture}[scale=\GraphScale]
  \BlueVertices
  \VertexLE[x=3.00,y=0.000,L=\relax]{2}{$v_{\t_4}$};
  \VertexLW[x=0.000,y=0.000,L=\relax]{3}{$v_{\t_1}$};
  \VertexLE[x=1.50,y=1.00,L=\relax]{4}{$v_{\t_2}$};
  \VertexLS[x=1.50,y=0.000,L=1]{1}{$v_{\cR}$};
  \YellowEdges
  \Edge(1)(3)
  \Edge(1)(2)
  \Edge(1)(4)
\TreeEdgeSignS(1)(2){0.5}{2}
\TreeEdgeSignS(1)(3){0.5}{2}
\TreeEdgeSignW(1)(4){0.5}{2}
\end{tikzpicture}}},
$$
\vskip 10pt
\noindent with centres for the twins $z_{\t_1} = \zeta_7$, $z_{\t_2} = \zeta_7^2$ and $z_{\t_4} = \zeta_7^4$. Note that $\Frob(\t_1)= \t_4$, $\Frob(\t_4)= \t_2$ and $\Frob(\t_2) = \t_1$. We find that 
$$
\theta_{\t_1}^2 = (\zeta_7-\zeta_7^2)^2(\zeta_7-\zeta_7^4)^2\zeta_7(\zeta_7-1)(\zeta_7-2) \equiv \zeta_7^2+3\zeta_7+7 \mmod 11,
$$
and similarly for $\theta_{\t_2}^2$ and $\theta_{\t_4}^2$.
We can pick $\theta_{\t_1}, \theta_{\t_2},\theta_{\t_4}$ so that $\Frob(\theta_{\t_1}) = \theta_{\t_4}$, $\Frob(\theta_{\t_4}) = \theta_{\t_2}$ and therefore $\epsilon_{\t_1}(\Frob)= \epsilon_{\t_4}(\Frob) =+1$. Then $\epsilon_{\t_2}(\Frob)\equiv \frac{\Frob^3(\theta_{\t_1})}{\theta_{\t_1}} \mmod 11$. One checks that $\theta^2_{\t_1}$ is not a square in $\F_{11}(\zeta_7)$, so $\epsilon_{\t_2}(\Frob) =-1$. 
\begin{minipage}[t]{.7\textwidth}
\quad In terms of the BY tree, $\Frob$ permutes the three edges cyclicly. Here the yellow components are the three edges $v_{\cR}v_{\t_1}$, $v_{\cR}v_{\t_2}$, $v_{\cR}v_{\t_4}$, and $\epsilon_{\Frob}(v_{\cR}v_{\t_1})\!=\!\epsilon_{\Frob}(v_{\cR}v_{\t_4})\!=\!+1$, while $\epsilon_{\Frob}(v_{\cR}v_{\t_2})\!=\!-1$.
\end{minipage}
\begin{minipage}{0.25\textwidth}
$\hfill{\smash{\raise -40pt\hbox{\begin{tikzpicture}[scale=\GraphScale]
  \BlueVertices
  \VertexLE[x=3.00,y=0.000,L=\relax]{2}{$v_{\t_4}$};
  \VertexLW[x=0.000,y=0.000,L=\relax]{3}{$v_{\t_1}$};
  \VertexLN[x=1.50,y=1.00,L=\relax]{4}{$v_{\t_2}$};
  \VertexLS[x=1.50,y=0.000,L=1]{1}{};
  \YellowEdges
  \Edge(1)(3)
  \Edge(1)(2)
  \Edge(1)(4)
\TreeEdgeSignS(1)(2){0.5}{}
\TreeEdgeSignS(1)(3){0.5}{}
\TreeEdgeSignW(1)(4){0.5}{}
\TreeSignAt(1)(-1.2,0.2){$+$}
\TreeSignAt(2)(-0.3,-0.2){$+$}
\TreeSignAt(4)(0.3,-0.2){-}
\EArr1312{in=-120,out=-60}
\EArr1214{in=0,out=60}
\EArr1413{in=90,out=180}
\end{tikzpicture}}}}
$
\end{minipage}
\end{example}

\begin{example}\label{BYex3}
Let 
$C/\Q_p:y^2=(x-1)(x-2)(x-3)(x-p^2)(x-p^{n+2})(x+p^{n+2})$ for $p \ge 5$ and $n \ge 1$. The substitutions $(x',y')=(\frac{1}{x-1},\frac{y}{x-1})$ and $(x'',y'') = ( \frac{1}{x}, \frac{y}{x})$ yield other models. Their cluster pictures are respectively: 
$$\smash{\raise0pt\hbox{\clusterpicture            
  \Root[A] {} {first} {r1};
  \Root[A] {} {r1} {r2};
  \Root[A] {} {r2} {r3};
  \Root[A] {2} {r3} {r4};
  \Root[A] {2} {r4} {r5};
  \Root[A] {} {r5}{r6};
\ClusterLDName c1[][n][] = (r5)(r6);
\ClusterLDName c2[][2][] = (r4)(c1);
\ClusterLDName c3[][0][] = (r1)(r2)(r3)(c1)(c2);
\endclusterpicture, \qquad
\clusterpicture            
  \Root[A] {} {first} {r1};
  \Root[A] {} {r1} {r2};
  \Root[A] {2} {r2} {r3};
  \Root[A] {2} {r3} {r4};
  \Root[A] {} {r4} {r5};
\ClusterLDName c1[][n][] = (r4)(r5);
\ClusterLDName c2[][2][] = (r3)(c1);
\ClusterLDName c3[][0][] = (r1)(r2)(c1)(c2);
\endclusterpicture \qquad \text{ and }\qquad
\clusterpicture            
  \Root[A] {} {first} {r1};
  \Root[A] {} {r1} {r2};
  \Root[A] {} {r2} {r3};
  \Root[A] {3} {r3} {r4};
  \Root[A] {2} {r4} {r5};
  \Root[A] {} {r5}{r6};
\ClusterLDName c1[][2][] = (r1)(r2)(r3);
\ClusterLDName c2[][n][] = (c1)(r4);
\ClusterLDName c3[][\!-\!n\!-\!2][] = (c1)(c2)(r5)(r6);
\endclusterpicture}}.
$$
\vskip 5pt
\noindent Note that these all have the same BY tree: 
$
\smash{\raise -7pt\hbox{\begin{tikzpicture}[scale=\GraphScale]
  \BlueVertices
  \VertexLE[x=.00,y=0.000,L=1]{1}{};
  \VertexLW[x=1.000,y=0.000,L=\relax]{2}{};
  \VertexLN[x=2,y=0.00,L=\relax]{3}{};
  \YellowEdges
  \Edge(2)(3)
  \BlueEdges
  \Edge(1)(2)
  \TreeEdgeSignS(1)(2){0.5}{2}
  \TreeEdgeSignS(2)(3){0.5}{2n}
\end{tikzpicture}}}.
$
Generally the BY tree is model independent. 
\end{example}

\begin{remark}
For semistable curves the special fibre of the minimal regular model is the double cover of the BY tree ramified over the blue part (with all edge lengths halved). In Example \ref{BYex3} the dual graph is 
$$
\smash{\raise -10pt\hbox{\begin{tikzpicture}[scale=\GraphScale]
  \BlackVertices
  \VertexLE[x=1.00,y=0.000,L=1]{1}{};
  \VertexLN[x=2,y=0.00,L=\relax]{3}{};
  \VertexLN[x=3,y=0.5,L=\relax]{4}{};
   \VertexLN[x=4,y=0.5,L=\relax]{5}{};
    \VertexLN[x=5,y=0.5,L=\relax]{6}{}; 
      \VertexLN[x=6,y=0,L=\relax]{7}{}; 
       \VertexLN[x=3,y=-0.5,L=\relax]{8}{};
   \VertexLN[x=4,y=-0.5,L=\relax]{9}{};
    \VertexLN[x=5,y=-0.5,L=\relax]{10}{}; 
  \BlackEdges
  \Edge(1)(3)
  \Edge(3)(4)
  \Edge(4)(5)
  \Edge[style={dashed}](5)(6)
    \Edge(6)(7)
     \Edge(7)(10)
      \Edge[style={dashed}](10)(9)
       \Edge(9)(8)
        \Edge(3)(8)
\end{tikzpicture}}},
$$
where the loop has $2n$ vertices. 
\end{remark}

\newpage

\section{Reduction type}\label{se:redtype}

In this section we explain how to read off information about the reduction of both $C$ and its Jacobian from the cluster picture of $C$. 

\begin{theorem}[Semistability criterion] \label{semistability criterion subsection}

The  curve $C$, or equivalently  $\textup{Jac }C$, is semistable if and only if the following three conditions are satisfied:
\begin{itemize}
\item[(1)] the field extension $K(\mathcal{R})/K$ given by adjoining the roots of $f(x)$ has ramification degree at most $2$,
\item[(2)] every proper cluster is invariant under the action of the inertia group $I_K$, 
\item[(3)] every principal cluster $\s$ has $d_\s\in \mathbb{Z}$ and $\nu_\mathfrak{s}\in 2\mathbb{Z}$.
\end{itemize}
\end{theorem}

\begin{remark} \label{field of semistability remark}
It follows from Theorem \ref{semistability criterion subsection} that $C$ is semistable over any ramified quadratic extension of $K(\mathcal{R})$. 
\end{remark}

\begin{theorem}[Good reduction of the curve] \label{good red curve}
The curve $C$ has good reduction if and only if the following three conditions are all satisfied:
\begin{itemize}
\item[(1)] the field extension $K(\mathcal{R})/K$ is unramified,
\item[(2)] every proper cluster has size at least $2g+1$,
\item[(3)] the (necessarily unique) principal cluster has $\nu_\s\in 2\mathbb{Z}$.
\end{itemize}
\end{theorem}

\begin{theorem}[Good reduction of the Jacobian] \label{good red jac}
The Jacobian of $C$  has good reduction if and only if the following three conditions are all satisfied:
\begin{itemize}
\item[(1)] the field extension $K(\mathcal{R})/K$ is unramified,
\item[(2)] every cluster $\s\neq \mathcal{R}$ is odd,
\item[(3)] every principal cluster $\s$ has $\nu_\s\in 2\mathbb{Z}$.
\end{itemize}
\end{theorem}

A consequence of Theorems \ref{good red curve} and \ref{good red jac}  is the following criterion for potentially good reduction.
\vspace{3pt}
\begin{theorem}[Potentially good reduction of the curve or the Jacobian]
~
\begin{itemize}
\item The curve $C$ has potentially good reduction if and only if every proper cluster has size at least $2g+1$.
\item The Jacobian, $\textup{Jac }C$, has potentially good reduction if and only if  every cluster $\s\neq \mathcal{R}$ is odd.
\end{itemize}  
\end{theorem}

\begin{theorem}[Potential toric rank of the Jacobian] \label{potential toric rank}\label{th:ptr}
~
\begin{itemize}
\item The potential toric rank of $\textup{Jac }C$ is equal to the number of even non-\ub ~  clusters $\s\neq \mathcal{R}$, less $1$ if $\mathcal{R}$ is \ub. 
\item The Jacobian, $\Jac C$, has potentially totally toric reduction if and only if every cluster has at most $2$ odd children.
\end{itemize}
\end{theorem}

\begin{remark}[Tame reduction] \label{tame remark}
The curve $C$, or equivalently $\textup{Jac }C$, has tame reduction (semistable after  tamely ramified  extension) if and only if $K(\mathcal{R})/K$ is tamely ramified. In particular, this is always the case if $p>2g+1$  since then the wild inertia group acts trivially on the roots of the (degree $\leq 2g+2$) polynomial $f(x)$.  
 \end{remark}
 
\newpage

\begin{example} \label{RTex1}
As in Example \ref{CLex1} we consider the genus $2$ hyperelliptic curve \[C : y^2=(x^2+7^2)(x^2-7^{15})(x-7^6)(x-7^6-7^9)\]  over $\Q_7$ with cluster picture
\vskip -8pt
$$
 \clusterpicture            
  \Root[A] {} {first} {r1};
  \Root[A] {} {r1} {r2};
  \Root[A] {2} {r2} {r3};
  \Root[A] {} {r3} {r4};
  \Root[A] {4} {r4} {r5};
  \Root[A] {} {r5}{r6};
  \ClusterLDName c1[][\frac 32][\t_1] = (r3)(r4);
  \ClusterLDName c2[][3][\t_2] = (r5)(r6);
  \ClusterLDName c3[][5][\mathfrak{s}] = (c1)(c2);
  \ClusterLDName c4[][1][\cR] = (r1)(r2)(c3);
\endclusterpicture 
\qquad \text{and }
\cR = \{7i, -7i, 7^{\frac{15}{2}}, -7^{\frac{15}{2}}, 7^6, 7^6+7^9\}.
$$

We have $d_\mathcal{R}=1$. The single principal cluster $\mathfrak{s}$ has $d_\mathfrak{s}=6$ and $|\mathfrak{s}|=4$. We find:
\begin{itemize}
\item $C$ is semistable. 
Indeed, $\mathbb{Q}_7(\mathcal{R})=\mathbb{Q}_7(i,\sqrt{7})$  has ramification degree $2$ over $\mathbb{Q}_7$. The inertia group swaps the roots  $7^{\frac{15}{2}}$ and $-7^{\frac{15}{2}}$ which lie in a twin, and fixes all others, so that every proper cluster is fixed by inertia. Finally,  $d_\mathfrak{s}\in \mathbb{Z}$ and 
$\nu_\mathfrak{s}=4\cdot d_\mathfrak{s}+2d_\mathcal{R}=26\in 2\mathbb{Z}.$
\item $C$ does not have potentially good reduction since the cluster $\mathfrak{s}$ has size $4<2g+1=5$. In fact, $\Jac C$ has totally toric reduction. Indeed, $C$ is already semistable over $\mathbb{Q}_7$, and every cluster has at most $2$ odd children ($\mathcal{R}$ and the twins $\mathfrak{t}_1$ and $\mathfrak{t}_2$ each have two odd children, whilst   $\mathfrak{a}$ has no odd children). 
\end{itemize}
\end{example}

\begin{remark}
Any hyperelliptic curve $C:y^2=f(x)$ with the same cluster picture as the one in Example \ref{RTex1} (same depths, all proper clusters inertia invariant) and such that $f(x)$ has unit leading coefficient,  is necessarily also semistable with totally toric reduction, by the same argument.
\end{remark}

\begin{example}
Consider the genus $2$ hyperelliptic curve $C:y^2=x^6-27$
over $\mathbb{Q}_3$. Its cluster picture is 
\vskip -8pt
$$
 \clusterpicture            
  \Root[A] {2} {first} {r1};
  \Root[A] {2} {r1} {r2};
  \Root[A] {2} {r2} {r3};
  \Root[A] {4} {r3} {r4};
  \Root[A] {2} {r4} {r5};
  \Root[A] {2} {r5}{r6};
  \ClusterLDName c1[][\frac{1}{2}][\s_1] = (r1)(r2)(r3);
  \ClusterLDName c2[][\frac{1}{2}][\s_2] = (r4)(r5)(r6);
  \ClusterLDName c3[][\frac{1}{2}][\cR] = (c1)(c2);
\endclusterpicture 
\qquad \text{with }~~
\cR = \{\sqrt{3},\zeta_3\sqrt{3},\zeta_3^2\sqrt{3},-\sqrt{3},-\zeta_3\sqrt{3},-\zeta_3^2\sqrt{3}\},
$$
for a fixed  primitive $3$rd root of unity $\zeta_3$. The non-principal cluster $\mathcal{R}$ has depth $\frac{1}{2}$, whilst the principal clusters $\s_1$ and $\s_2$ each have depth $1$.   We find:
\begin{itemize}
\item $C$ is not semistable since the action of inertia swaps $\s_1$ and $\s_2$. 
\item $C$ does not have potentially good reduction, since $\s_1$ and $\s_2$ are both proper clusters of size $<2g+1$. On the other hand, $\textup{Jac }C$ \textit{does} have potentially good reduction since $\s_1$ and $\s_2$  are odd.
\item $C$ has tame reduction since $\mathbb{Q}_3(\mathcal{R})=\mathbb{Q}_3(\sqrt{3},\zeta_3)$ has ramification degree $2$ over $\mathbb{Q}_3$. In fact, the minimal degree extension over which $C$ is semistable is $4$, realised by any totally ramified  extension of this degree. To see this, note that the inertia group acts on the proper clusters through it's unique order $2$ quotient, whilst for $i=1,2$ we have $d_{\s_i}\in \mathbb{Z}$ and $\nu_{\s_i}=3\cdot 1+3\cdot d_\mathcal{R}=9/2$, so that $C$ satisfies  the semistability criterion (Theorem \ref{semistability criterion subsection}) over some $F/\mathbb{Q}_3$ if and only if the ramification degree of this extension is divisible by $4$. 
\end{itemize}

\end{example}

\newpage
\section{Special fibre (semistable case)}
\label{se:specialfibre}
\newcommand{\Sectionsixcomment}[1]{}

In this section, assuming that $C/K$ is \underline{semistable} and that \underline{$\mathcal{R}$ is principal}, we describe the special fibre of the minimal regular model  of $C$ over $\mathcal{O}_{K^{\textup{nr}}}$. The case where $\cR$ is not principal is dealt with in \cite[Section 8]{m2d2}. 
\begin{definition}[Leading terms and reduction maps]
For a principal cluster $\s$, define $c_\s\in \bar{k}^{\times}$ and $\operatorname{red}_{\mathfrak{s}}\colon z_\s+\pi^{d_\s}\mathcal{O}_{\bar{K}}\rightarrow \bar{k}$ by
\begin{equation*}
    c_{\mathfrak{s}}= \frac{c}{\pi^{v(c)}} \prod_{r \notin \mathfrak{s}}\frac{z_{\mathfrak{s}}-r}{\pi^{v(z_\mathfrak{s}-r)}} \bmod \mathfrak{m} \,\qquad \text{and} \qquad \operatorname{red}_{\mathfrak{s}}(t)=\frac{t-z_{\mathfrak{s}}}{\pi^{d_{\mathfrak{s}}}} \quad \bmod \mathfrak{m}\text.
\end{equation*}
For any cluster $\mathfrak s' < \mathfrak s$  we  define $\operatorname{red}_{\mathfrak s} (\mathfrak  s')$ to be $\operatorname{red}_{\mathfrak s}(r)$ for any choice of $r \in \mathfrak s'$.
\end{definition}

\begin{theorem}[Components] \label{section 6 big theorem}
    The special fibre $\Cmink$ contains connected components $\Gamma _{\mathfrak s}$ corresponding to principal clusters $\mathfrak s$, given by the equations \Sectionsixcomment{ If $\delta_{\mathfrak s} \ne \frac 12$ it can be written as}
    \begin{equation*}
        \Gamma_{\mathfrak{s}}: Y^{2}=c_{\mathfrak{s}} \prod_{\text {odd } \mathfrak{o}<\mathfrak{s}}\left(X-\operatorname{red}_{\mathfrak{s}}(\mathfrak{o})\right) \prod_{\substack{\text {twin } \mathfrak{t}<\mathfrak{s}\\ \delta_{\mathfrak t} = \frac12}}\left(X-\operatorname{red}_{\mathfrak{s}}(\mathfrak{t})\right)^{2}.	
    \end{equation*}
    This component is irreducible when $\mathfrak s$ is non-\"ubereven but splits into a pair of irreducible components $\Gamma _{\mathfrak s}^+,\Gamma _{\mathfrak s}^-$ otherwise (we write $\Gamma_{\mathfrak s}^+ = \Gamma_{\mathfrak s}^- = \Gamma_{\mathfrak s}$ in the non-\"ubereven case). These components are linked by chains of $\P^1$s  as described in Theorem \ref{thm:links}.
\end{theorem}

\Sectionsixcomment{
\begin{remark}
In case $\delta_{\mathfrak s} = \frac 12$, the point at infinity needs to be a ramification point of the map to $\P^1$ and a slightly different model is needed, see \cite[Section 8]{m2d2}.
\end{remark}
}

\begin{theorem}[Links]\label{thm:links}
	The chains of $\mathbb P^1$s linking the irreducible components of Theorem \ref{section 6 big theorem}  arise   in exactly one of the following four ways:

    If $\mathfrak s' < \mathfrak s$ with both clusters principal and $\mathfrak s'$ is odd, we have a chain containing $\frac12 \delta _{\mathfrak s'}-1$ components, linking $\Gamma_{\mathfrak s}$ to $\Gamma_{\mathfrak s'}$.
    If $\mathfrak s' < \mathfrak s$ with both clusters principal and $\mathfrak s'$ even, we have two chains containing $\delta _{\mathfrak s'}-1$ components each, one linking $\Gamma_{\mathfrak s}^+$ to $\Gamma_{\mathfrak s'}^+$ and the other $\Gamma_{\mathfrak s}^-$ to $\Gamma_{\mathfrak s'}^-$.
    If $\mathfrak t < \mathfrak s$ with $\mathfrak s$ principal and $\mathfrak t$ a twin, we have a chain containing $2\delta _{\mathfrak t}-1$ components, linking $\Gamma_{\mathfrak s}^+$ to $\Gamma_{\mathfrak s}^-$.
    \Sectionsixcomment{
    {\color{red} If $\mathfrak s < \mathfrak t$, with $\mathfrak s$ principal and $\mathfrak t$ a cotwin, we have a chain containing $2\delta _{\mathfrak s}$ components, linking $\Gamma_{\mathfrak s}^+$ to $\Gamma_{\mathfrak s}^-$. }
    
   {\color{red} If $\mathcal{R} = \mathfrak{s}_1 \sqcup \mathfrak{s}_2$ with $\mathfrak{s}_1, \mathfrak{s}_2$ principal odd, we have a chain of length $\tfrac12(\delta_{\mathfrak{s}_1} + \delta_{\mathfrak{s}_2})$ linking $\Gamma_{\mathfrak{s}_1}$ and $\Gamma_{\mathfrak{s}_2}$.
   If $\mathcal{R} = \mathfrak{s}_1 \sqcup \mathfrak{s}_2$ with $\mathfrak{s}_1, \mathfrak{s}_2$ principal even, we have two chain of length $\delta_{\mathfrak{s}_1} + \delta_{\mathfrak{s}_2}$, one linking $\Gamma_{\mathfrak{s}_1}^+$ and $\Gamma_{\mathfrak{s}_2}^+$, the other linking $\Gamma_{\mathfrak{s}_1}^-$ and $\Gamma_{\mathfrak{s}_2}^-$.
   If $\mathcal{R} = \mathfrak{s}_1 \sqcup \mathfrak{s}_2$ with $\mathfrak{s}_1$ principal even and $\mathfrak{s}_2$ a twin, we have a chain of length $2(\delta_{\mathfrak{s}_1} + \delta_{\mathfrak{s}_2})$,  linking $\Gamma_{\mathfrak{s}_1}^+$ and $\Gamma_{\mathfrak{s}_1}^-$.}}
\end{theorem}


\begin{theorem}[Frobenius action]
    The Frobenius element $\Frob$ acts by permutation on the components of  $\Cmink$ by sending $\Gamma_{\mathfrak{s}}^\pm$ to $\Gamma_{\Frob(\mathfrak{s})}^{\pm \epsilon_{\mathfrak s}(\Frob)}$.
\end{theorem}

\begin{remark}
There are also formulae describing the Frobenius action on the linking chains. See   Theorem  \ref{thm:dual_graph_semistable}, and for full details \cite[Theorem 8.5]{m2d2}.
\end{remark}

\begin{theorem}[Reduction maps]
   For a principal cluster $\s\neq \mathcal{R}$, the reduction of a point $(x,y) \in C(K^{\mathrm{nr}})$ lies on $\Gamma _{\mathfrak s}$ if and only if 
   \begin{equation} \label{reduction map condition}
   v(x-z_{\mathfrak s}) \ge d_{\mathfrak s}\mathrm{~and~} \operatorname{red}_{\mathfrak s}(x)\ne\operatorname{red}_{\mathfrak s}(\mathfrak s')\mathrm{~for~every~proper~}\mathfrak s' < \mathfrak s.
   \end{equation}
    When these conditions are satisfied the reduction is given by 
    \begin{equation} \label{formula for the reduction}
        (x, y) \mapsto\left(\operatorname{red}_{\mathfrak{s}}(x), \pi^{-\frac{\nu_{\mathfrak{s}}}{2}} y \cdot \prod_{\substack{\mathfrak{s}^{\prime}< \mathfrak{s} \\ \delta_{\mathfrak s'}>\frac{1}{2}}}\left(\operatorname{red}_{\mathfrak{s}}(x)-\operatorname{red}_{\mathfrak{s}}(\mathfrak{s}^{\prime})\right)^{-\left\lfloor\frac{\left|\mathfrak{s}^{\prime}\right|}{2}\right\rfloor}\right).
    \end{equation}
  If $\s=\mathcal{R}$ then the reduction of $(x,y)\in C(K^{\mathrm{nr}})$ lies on $\Gamma _\mathcal{R}$ if and only if either (\ref{reduction map condition}) holds, or $v(x-z_\mathcal{R})<d_\s$. In the former case the reduction is given by (\ref{formula for the reduction}), whilst in the latter case  $(x,y)$ reduces to one of the points at infinity on $\Gamma_\mathcal{R}$.\footnote{When there are two points at infinity on $\Gamma_\mathcal{R}$ the reduction can be pinned down precisely by \cite[Proposition 5.23 (i)]{m2d2}.}
\end{theorem}



\begin{example}
 Consider the genus $2$ curve $C \colon y^2 = x((x+1)^2-5)(x+4)(x-6)$ over $\Q_5$ with associated cluster picture
$$\clusterpicture\Root[A] {2} {first} {mc1};
\Root[A] {2} {mc1} {mc2c1};
\Root[A] {2} {mc2c1} {mc2c2};
\ClusterLDName mc2[][\frac{1}{2}][\mathfrak{t}_1] = (mc2c1)(mc2c2);
\Root[A] {2} {mc2} {mc3c1};
\Root[A] {2} {mc3c1} {mc3c2};
\ClusterLDName mc3[][1][\mathfrak{t}_2] = (mc3c1)(mc3c2);
\ClusterLDName m[][0][\mathcal{R}] = (mc1)(mc2)(mc3);
\endclusterpicture.$$
Picking $z_{\mathcal{R}} = 0$  we have
$\mathrm{red}_{\mathcal{R}}(t) = t \mod \mathfrak m$ and $c_{\mathcal{R}} = 1\in \bar{\mathbb{F}}_5^\times$.
The special fibre of the minimal regular model has a component coming from the unique principal cluster $\mathcal{R}$ given by the equation
\begin{align*}
&\Gamma_{\mathcal{R}} \colon Y^2 = c_{\mathcal{R}} \cdot (X - \mathrm{red}_{\mathcal{R}}(0))(X - \mathrm{red}_{\mathcal{R}}(-4))^2 =  X(X + 1)^2, &\qquad  
\end{align*}
 a genus $0$ curve with a single node at $(X,Y)=(-1,0)$. For the twin $\mathfrak{t}_1$ we have  $2 \delta_{\mathfrak{t}_1}-1 = 0$ so that $\mathfrak{t}_1$ contributes no components (rather, it corresponds to the node on $\Gamma_\mathcal{R}$). On the other hand, the twin $\mathfrak{t}_2$ gives rise to a chain of $2\delta_{\mathfrak{t}_2}-1 = 1$  projective lines from $\Gamma_{\mathcal{R}}$ to itself, as pictured below.
\begin{figure}[h]
\vspace{-26pt}
	\begin{tikzpicture}
	\draw [line width=1.5pt] (0.5,0)-- node[left, xshift=-2cm]{$\Gamma_{\cR}$} ++(3,0);
	\draw [line width=1.5pt] (0.5,0)-- ++(0,-0.8);
	\draw (1.2,0.5) .. controls (1.45, -1.3)and (2.45, -1.3) .. (2.7,0.5);
	\draw [line width=1.5pt] (0.5,0) .. controls (-1, 0)and (0.5, 1.5) .. (0.5,0);
	\end{tikzpicture}
\end{figure}
\vspace{-20pt}

A point $(x,y) \in C(\Q^{\mathrm{nr}}_5)$ reduces to a point on $\Gamma_{\mathcal{R}}$ if and only if either $x \notin \Z^{\mathrm{nr}}_5$, in which case it reduces to the unique point at infinity on $\Gamma_\mathcal{R}$, or  $x \in \Z^{\mathrm{nr}}_5$ and  $x \not \equiv \pm 1~~\textup{mod }5$. Since $\nu_\mathcal{R}=0$, for points satisfying the second condition the reduction map is given by
$(x,y)\mapsto \left(\bar{x},\bar{y}(\bar{x}-1)^{-1}\right).$ 
\Sectionsixcomment{
and $(-19, 48\sqrt{33649})$ reduces to $\Gamma_{\mathcal{R}}$ and its reduction is
$$\left(-19, 3^0 \cdot 48\sqrt{33649} \cdot (-19 - \mathrm{red}_{\mathcal{R}}(\mathfrak{s}_1))^{-1} (-19 - \mathrm{red}_{\mathcal{R}}(\mathfrak{s}_2))^{-1} \right)  \equiv (0, 0) \mod \mathfrak{m}.$$

and $(i + 2, 10\sqrt{7i - 16})$, where $i$ is a root of $-1$ in $\Q_3^{\mathrm{nr}}$, reduces to $\Gamma_{\mathcal{R}}$ and its reduction is
$$\left(i+2, 10\sqrt{7i-16} \cdot (i+2-4)^{-1} (i+2 - 0)^{-1} \right) = \left(i+2, -2\sqrt{i+2}\right) \in \Gamma_{\mathcal{R}}(\F_{3^4})$$ 
The point $(i, 2\sqrt{102 - 289i})$, where $i$ is a root of $-1$ in $\Q_3^{\mathrm{nr}}$, reduces to $\Gamma_{\mathcal{R}}$ and its reduction is
$$\left(i, 2\sqrt{102 - 289i} \cdot (i-4)^{-1} (i - 0)^{-1} \right) = (i, 1)
\in \Gamma_{\mathcal{R}}(\F_{3}(i))\text.$$ }
\end{example}

\Sectionsixcomment{
\begin{example}
Let $C_2 \colon y^2 = x(x-1)(x-p)(x-p^2)(x-2p^2)(x-p^4)$ over $\Qp$ with associated cluster picture depicted below. {\color{red} Not good: $\mathcal{R}$ is not principal.}
\begin{figure}[h] 
\clusterpicture
\Root[A] {1} {first} {mc1c1c1c1};
\Root[A] {1} {mc1c1c1c1} {mc1c1c1c2};
\ClusterLDName mc1c1c1[][2][\mathfrak{t}] = (mc1c1c1c1)(mc1c1c1c2);
\Root[A] {1} {mc1c1c1} {mc1c1c2};
\Root[A] {1} {mc1c1c2} {mc1c1c3};
\ClusterLDName mc1c1[][1][\mathfrak{p}] = (mc1c1c1)(mc1c1c2)(mc1c1c3);
\Root[A] {1} {mc1c1} {mc1c2};
\ClusterLDName mc1[][1][\mathfrak{c}] = (mc1c1)(mc1c2);
\Root[A] {1} {mc1} {mc2};
\ClusterLDName m[][0][\mathcal{R}] = (mc1)(mc2);
\endclusterpicture
\end{figure}

Then $\mathfrak{p}$ is the only principal cluster, $\mathfrak{c}$ is a cotwin, and $\mathfrak{t}$ is a twin. When we pick $z_{\mathfrak{p}} = 0$, we get that
$$\Gamma_{\mathfrak{p}} \colon y^2 = ... \cdot (x-1)(x-2) \cdot x^2$$
is of genus 1. The cotwin $\mathfrak{c}$ gives rise to a chain of $\P^1$'s of length $2 \cdot \delta_{\mathfrak{p}} = 2$ from $\Gamma_{\mathfrak{p}}$ to itself, and the twin $\mathfrak{t}$ gives rise to a chain of $\P^1$'s of length $2 \cdot \delta_{\mathfrak{t}} = 4$ from $\Gamma_{\mathfrak{p}}$ to itself.
\end{example}
}

\begin{example}
   Consider $C \colon y^2 = (x^4-p^8)((x+1)^2-p^2)((x-1)^2-p)$ over $\Qp$, with associated cluster picture
    \vspace{-5pt}
$$\clusterpicture
\Root[A] {1} {first} {mc1c1};
\Root[A] {1} {mc1c1} {mc1c2};
\Root[A] {1} {mc1c2} {mc1c3};
\Root[A] {1} {mc1c3} {mc1c4};
\ClusterLDName mc1[][2][\mathfrak{s}] = (mc1c1)(mc1c2)(mc1c3)(mc1c4);
\Root[A] {1} {mc1} {mc2c1};
\Root[A] {1} {mc2c1} {mc2c2};
\ClusterLDName mc2[][1][\mathfrak{t}_1] = (mc2c1)(mc2c2);
\Root[A] {1} {mc2} {mc3c1};
\Root[A] {1} {mc3c1} {mc3c2};
\ClusterLDName mc3[][\frac12][\mathfrak{t}_2] = (mc3c1)(mc3c2);
\ClusterLDName m[][0][\mathcal{R}] = (mc1)(mc2)(mc3);
\endclusterpicture.$$
Then $\mathcal{R}$ and $\mathfrak{s}$ are the only principal clusters. Moreover, $\mathcal{R}$ is \"ubereven. Taking $z_{\mathcal{R}} = z_{\mathfrak{s}} = 0$, we get associated components of $\mathcal{C}^{\min}_{\bar{\mathbb{F}}_p}$:
$$ \Gamma_{\mathcal{R}}^+ \colon Y =X-1,\quad\Gamma_{\mathcal{R}}^- \colon Y = 1-X ,\quad \textup{and}\quad \Gamma_\s:Y^2=X^4-1. $$

The parent-child relation $\mathfrak{s} < \mathcal{R}$ gives rise to two chains of length $\delta_{\mathfrak{s}} = 1$, one linking $\Gamma_{\mathcal{R}}^+$ with $\Gamma_{\mathfrak{s}}$, and the other linking $\Gamma_{\mathcal{R}}^-$ with $\Gamma_{\mathfrak{s}}$. The twin  $\mathfrak{t}_1$ gives rise to a chain of length $2\delta_{\t_1}-1=1$   linking $\Gamma_{\mathcal{R}}^-$ to $\Gamma_{\mathcal{R}}^+$. The twin $\t_2$ has $2\delta_{\t_2}-1=0$ so contributes a chain of length $0$ from $\Gamma_{\mathcal{R}}^-$ to $\Gamma_{\mathcal{R}}^+$, which is to be interpreted as a point of intersection between these two curves. The configuration of the components of the special fibre is shown below. Finally, since both $\mathcal{R}$ and $\s$ are $G_K$-stable, and $\epsilon_{\mathcal{R}}(\sigma) = 1$ for all $\sigma \in G_K$,  the Frobenius element fixes $\Gamma_{\mathcal{R}}^+$, $\Gamma_{\mathcal{R}}^-$, and $\Gamma_{\mathfrak{s}}$.

\begin{figure}[h]
\begin{tikzpicture}
\draw [line width=1.5pt] (0,0)-- ++(1.8,1.3);
\draw [line width=1.5pt] (3,0)-- ++(-1.8,1.3);
\draw (0.5,0.7)-- ++(2.1,0);
\draw (0.3,0.5)-- ++(0,-1.5);
\draw (2.7,0.5)-- ++(0,-1.5);
\draw [line width=1.5pt] (0,-0.7)-- node[above]{$\Gamma_{\c}$} node[below, font=\small]{genus $1$} ++(3,0);
\draw (-0.3,0) node{$\Gamma_{\mathcal{R}}^+$};
\draw (3.3,0) node{$\Gamma_{\mathcal{R}}^-$};
\end{tikzpicture}
\end{figure}

\end{example}

\newpage

\section{Minimal regular model (semistable case)}\label{se:MinRegMod}

Throughout this section, we assume that $C$ is \underline{semistable}. We also assume for simplicity that \underline{all proper clusters have integral depth}, and that \underline{there is no cluster} \underline{$\s\neq\cR$ of size $2g+1$}.

\begin{definition}\label{def:valid_discs}
An \emph{integral disc} in $\Kbar$ is a subset $D\subseteq\Kbar$ of the form $D=D(z_D,d_D)=\{x\in \Kbar\::\:v(x-z_D)\geq d_D\}$ with $d_D\in\ZZ$. The point $z_D$ is called a \emph{centre} of $D$, and $d_D$ is called its \emph{depth}. The \emph{parent disc} $P(D)$ of $D$ is the disc with the same centre and depth $d_D-1$. We also write $\nu_D(f)=v(c)+\sum_{r\in\cR}\min\{d_D,v(r-z_D)\}$, and $\omega_D(f)\in\{0,1\}$ for the parity of $\nu_D(f)$.

We write $D(\cR)$ for the smallest disc containing $\cR$. An integral disc $D$ is called \emph{valid} when $D\subseteq D(\cR)$ and $\#(\cR\cap D)\geq2$.
\end{definition}

\subsection*{Construction of a regular model $\cC^{\mathrm{disc}}$ over $\OKnr$}

Firstly, for each valid disc~$D$ we let $f_D(x_D)\in\OKnr[x_D]$ denote the polynomial $f_D(x_D)\!=\!\pi^{-\nu_D(f)}f(\pi^{d_D}x_D\!+\!z_D)$. We set $\mathcal{U}_D$ to be the subscheme of $\mathbb{A}^2_{\OKnr}$ cut out by $y_D^2=\pi^{\omega_D(f)}f_D(x_D)$. We let $\mathcal{U}_D^\circ$ denote the open subscheme of $\mathcal{U}_D$ formed by removing all the points in the special fibre corresponding to repeated roots of the reduction of $f_D$ (viewed as points on $\mathcal{U}_D$ with $y_D=0$).

\smallskip

Next, for the maximal valid disc $D=D(\cR)$ we let $g_D(t_D)\in\OKnr[t_D]$ denote the polynomial $g_D(t_D)=t_D^{\deg(f)}f_D(1/t_D)$. We set $\mathcal W_D$ to be the subscheme of $\mathbb{A}^2_{\OKnr}$ cut out by $w_D^2=\pi^{\omega_D(f)}g_D(t_D)$ if $\deg(f)$ is even, and $w_D^2=\pi^{\omega_D(f)}t_Dg_D(t_D)$ if $\deg(f)$ is odd. Again, we let $\mathcal{W}_D^\circ$ denote the open subscheme of $\mathcal{W}_D$ formed by removing all the points in the special fibre corresponding to repeated roots of the reduction of $g_D$ (viewed as points on $\mathcal{W}_D$ with $w_D=0$).

\smallskip

Finally, for each valid disc $D\!\neq\!D(\cR)$, we let $g_D(s_D,t_D)\!\in\!\OKnr[s_D,t_D]/(s_Dt_D\!-\!\pi)$ be the polynomial satisfying $ g_D(\pi/t_D,t_D)=t_D^{\nu_D(f)-\nu_{P(D)}(f)}f_D(1/t_D)$ in $\Knr(t_D)$. We set $\mathcal{W}_D$ to be the subscheme of $\mathbb{A}^3_{\OKnr}$ cut out by the equations $s_Dt_D=\pi$ and $w_D^2=s_D^{\omega_D(f)}t_D^{\omega_{P(D)}(f)}g_D(s_D,t_D)$. Again, we let $\mathcal{W}_D^\circ$ denote the open subscheme of $\mathcal{W}_D$ formed by removing all the points in the special fibre corresponding to repeated roots of the reduction of $g_D$ (viewed as points on $\mathcal{W}_D$ with $w_D=0$).

\begin{remark}
An explicit formula for $g_D$ is given in \cite[Definition~3.15]{m2d2}.
\end{remark}

\begin{theorem}\label{thm:integral_model}
A regular model $\cC^{\mathrm{disc}}$ of $C$ over $\OKnr$ is given by gluing each $\mathcal{W}_D^\circ$ to $\mathcal{U}_D^\circ$ for each valid $D$, and to $\mathcal{U}_{P(D)}^\circ$ for all valid $D\neq D(\cR)$ via the identifications
\begin{align*}
t_D = 1/x_D &= \pi/(x_{P(D)}-\pi^{1-d_D}(z_D-z_{P(D)})), \\
s_D = \pi x_D &= x_{P(D)}-\pi^{1-d_D}(z_D-z_{P(D)}), \\
w_D = t_D^{\lfloor\nu_D(f)/2\rfloor-\lfloor\nu_{P(D)}(f))/2\rfloor}y_D &= s_D^{\lfloor\nu_{P(D)}(f))/2\rfloor-\lfloor\nu_D(f)/2\rfloor}y_{P(D)}.
\end{align*}
\end{theorem}

\begin{remark}\label{rem:C_disc_not_minimal}
The regular model $\cC^{\mathrm{disc}}$ above is not minimal in general: discs with $\omega_D(f)=1$ produce $\mathbb{P}^1$s in the special fibre with multiplicity $2$ and self-intersection $-1$. Blowing down these components yields the minimal regular model.
\end{remark}

\begin{remark}\label{rem:twins_regularmodel}
In the construction of $\cC^{\mathrm{disc}}$ in \cite[Proposition~5.5]{m2d2} for general semistable $C$, the scheme $\mathcal{U}_D^\circ$ is defined by removing from the special fibre of $\mathcal{U}_D$ all points corresponding to the maximal valid subdiscs of~$D$. Under our extra assumptions, this is equivalent to the reduction of $f_D$ having a repeated root at this point. This is untrue when $C$ has a twin of half-integral depth; see Example~\ref{ex:minregmod2}.
\end{remark}

\newpage

\begin{example}\label{ex:minregmod1}
	Consider $C: y^2 = (x^4-p^4)(x^4-1)$ over $\QQ_p$. Its cluster picture is 
	\vspace{-0.1cm}
	\[
		\scalebox{1}{\clusterpicture            
		\Root[A] {1} {first} {r1};
		\Root[A] {} {r1} {r2};
		\Root[A] {} {r2} {r3};
		\Root[A] {} {r3} {r4};
		\Root[A] {3} {r4} {r5};
		\Root[A] {} {r5} {r6};
		\Root[A] {} {r6} {r7};
		\Root[A] {} {r7} {r8};
		\ClusterLD c1[][{1}] = (r1)(r2)(r3)(r4);;
		\ClusterLD c2[][{0}] = (c1)(r5)(r6)(r7)(r8);
		\endclusterpicture} \,.
	\]
	Here, there are two valid discs $D = D(0,0)$ and $D' = D(0,1)$. These correspond to the two proper clusters in the cluster picture. Using $\nu_D(f) = 0$ and $\nu_{D'}(f) = 4$, we find
	\[
	\mathcal{U}_D = \Spec\left(\frac{\ZZ_p^{\nr}[x,y]}{(y^2-(x^4-p^4)(x^4-1))}\right),
	\; \mathcal{W}_D = 
	\Spec\left(\frac{\ZZ_p^{\nr}[t,w]}{(w^2-(1-p^4t^4)(1-t^4))}\right)
	\]
	and 
	\[
		\resizebox{\hsize}{!}{$
	\mathcal{U}_{D'} = \Spec\left(\frac{\ZZ_p^{\nr}[x',y']}{(y'^2-(x'^4-1)(p^4x'^4-1))}\right),
	\; \mathcal{W}_{D'} = 
	\Spec\left(\frac{\ZZ_p^{\nr}[s',t',w']}{(s't'-p, w'^2-(1-t'^4)(s'^4-1))}\right).$}
	\]

	We have $\mathcal{U}_D^{\circ}  = \mathcal{U}_D \setminus \{(x,y,p)\}$, whereas $\mathcal{U}_{D'}^{\circ}  = \mathcal{U}_{D'}$, $\mathcal{W}_{D}^{\circ}  = \mathcal{W}_{D}$ and $\mathcal{W}_{D'}^{\circ}  = \mathcal{W}_{D'}$.
	Using the identifications $t' = 1/x' = p/x$, $s' = px' = x, y' = y/p^2$ and $w' = t'^2 y'$, we see that the special fibre of $\cC^{\mathrm{disc}}$ consists of two genus $1$ curves which intersect in two points.
\end{example}

\begin{example}\label{ex:minregmod2}	
	Consider $C: y^2 = p(x^2-p^5)(x^3-p^3)((x-1)^3-p^9)$ over $\QQ_p$ for $p \geq 5$. Its cluster picture is 
\vspace{-.3cm}	$$
	\clusterpicture           
	\Root[A] {} {first} {r1};
	\Root[A] {} {r1} {r2};
	\Root[A] {4} {r2} {r3};
	\Root[A] {} {r3} {r4};
	\Root[A] {} {r4} {r5};
	\Root[A] {5} {r5} {r6};
	\Root[A] {} {r6} {r7};
	\Root[A] {} {r7} {r8};
	\ClusterLD c1[][{\frac 32}] = (r1)(r2);
	\ClusterLD c2[][{1}] = (c1)(r3)(r4)(r5);
	\ClusterLD c3[][{3}] = (r6)(r7)(r8);
	\ClusterLD c4[][{0}] = (c2)(c3);
	\endclusterpicture.
	$$
	\vspace{-0.1cm}
	There are six valid discs: $D(0,0), \; D(0,1), \; D(0,2), \;  D(1,1), \;  D(1,2), \;  D(1,3).$
	Not all of these discs are minimal defining discs for clusters. For example, the cluster of relative depth $3$ is cut out by three different valid discs.
	
Note that $C$ has a proper cluster of non-integral depth, so Theorem~\ref{thm:integral_model} does not apply verbatim; we need the more general version from Remark~\ref{rem:twins_regularmodel}. We give a few illustrative charts of the model $\cC^{\mathrm{disc}}$.

	For $D = D(0,0)$, we find
	\[
	\mathcal{U}_D = \Spec \left(\frac{\ZZ_p^{\nr}[x,y]}{(y^2- p(x^2-p^5)(x^3-p^3)((x-1)^3-p^9))}\right)
	\]
	and $\mathcal{U}_D^{\circ} = \mathcal{U}_D \setminus \{(x,y,p),(x-1,y,p)\}$. Note that the special fibre of $\mathcal{U}_D$ is non-reduced. More precisely its closure is a projective line of multiplicity $2$ with self-intersection $-1$ and is blown down when constructing the minimal regular model (see Remark \ref{rem:C_disc_not_minimal}). The same applies to the component corresponding to $D(1,2)$. 
	
	For the disc $D_1 = D(1,1)$, we get
	\[
	\mathcal{W}_{D_1} =
	\Spec\left(\frac{\ZZ_p^{\nr}[s_1,t_1,w_1]}{(s_1t_1-p, w_1^2-t_1(1-p^6t_1^3)((s_1+1)^2-p^5)((s_1+1)^3-p^3))}\right)
	\]
	and $\mathcal{W}_{D_1}^{\circ} = \mathcal{W}_{D_1} \setminus \{(s_1+1,w_1,p)\}$. Similarly for $D_1' = D(0,1)$, we have $\mathcal{W}_{D_1'}^{\circ} = \mathcal{W}_{D_1'} \setminus \{(s'_1-1,w'_1,p)\}$.
	
	Let $D_2 = D(0,2)$, then  
	\[
	\mathcal{U}_{D_2} = \Spec \left(\frac{\ZZ_p^{\nr}[x_2,y_2]}{(y_2^2- (x_2^2-p)(p^3x_2^3-1)((p^2x_2-1)^3-p^9))}\right).
	\]
	Note that although the reduction of $f_{D_2}$ has a double root at $x_2 = 0$, this double root does not correspond to a valid subdisc of~$D_2$. Hence we do not remove this point in forming $\mathcal{U}_{D_2}^\circ$, so $\mathcal{U}_{D_2}^\circ=\mathcal{U}_{D_2}$ in this case.
\end{example}

\newpage
\section{Dual graph of special fibre and its homology (semistable case)}\label{se:dual}

In this section $C$ is \underline{semistable}. Let $\Cmin$ be its minimal regular model over $\OKnr$. The \emph{dual graph} $\Upsilon_C$ consists of a vertex $v_\Gamma$ for every irreducible component $\Gamma$ of the geometric special fibre $\Cmink$, with an edge connecting $v_\Gamma$ and $v_{\Gamma'}$ for each intersection point of $\Gamma$ and $\Gamma'$ (self-intersections of $\Gamma$ correspond to loops based at $v_\Gamma$). The action of $\Frob$ on $\Cmink$ induces a corresponding action on~$\Upsilon_C$.

\begin{theorem}\label{thm:dual_graph_semistable}
$\Upsilon_C$ consists of one vertex $v_\s$ for every non-\"ubereven principal cluster $\s$ and two vertices $v_\s^+$, $v_\s^-$ for each \"ubereven principal cluster $\s$, connected by chains of edges as follows:

\bgroup
\def\arraystretch{1.15}
\begin{tabular}{|c|c|c|c|c|}
\hline
Name & \multicolumn{2}{c|}{Endpoints} & Length & Conditions \\\hline
$L_{\s'}$ & $v_{\s'}$ & $v_\s$ & $\frac12\delta_{\s'}$ & $\s'<\s$ both principal, $\s'$ odd \\\hline
$L_{\s'}^{\pm}$ & $v_{\s'}^{\pm}$ & $v_\s^{\pm}$ & $\delta_{\s'}$ & $\s'<\s$ both principal, $\s'$ even \\\hline
$L_\t$ & $v_\s^-$ & $v_\s^+$ & $2\delta_\t$ & $\s$ principal, $\t<\s$ twin \\\hline
$L_\t$ & $v_\s^-$ & $v_\s^+$ & $2\delta_\t$ & $\s$ principal, $\s<\t$ cotwin \\\hline
\multicolumn{5}{|c|}{and, if $\cR$ is non-principal} \\\hline
$L_{\s_1,\s_2}$ & $v_{\s_1}$ & $v_{\s_2}$ & $\frac12(\delta_{\s_1}+\delta_{\s_2})$ & $\cR=\s_1\sqcup\s_2$ with $\s_1,\s_2$ principal odd \\\hline
$L_{\s_1,\s_2}^{\pm}$ & $v_{\s_1}^{\pm}$ & $v_{\s_2}^{\pm}$ & $\delta_{\s_1}+\delta_{\s_2}$ & $\cR=\s_1\sqcup\s_2$ with $\s_1,\s_2$ principal even \\\hline
$L_\t$ & $v_\s^-$ & $v_\s^+$ & $2(\delta_\s+\delta_\t)$ & $\cR=\s\sqcup\t$ with $\s$ principal even, $\t$ twin \\\hline
\end{tabular}
\egroup\\
Here, we adopt the convention that $v_\s^+=v_\s^-=v_\s$ if $\s$ is not \"ubereven, so for example if $\s'<\s$ are even non-\"ubereven principal clusters, then there are two chains of edges $L_{\s'}^+$, $L_{\s'}^-$ connecting $v_{\s'}$ and $v_\s$.

Frobenius acts on $\Upsilon_C$ by $\Frob(v_\s^{\pm})=v_{\Frob(\s)}^{\pm\epsilon_\s(\Frob)}$, $\Frob(L_{\s'}^{\pm})=L_{\Frob(\s')}^{\pm\epsilon_\s(\Frob)}$ and $\Frob(L_\t)=\epsilon_\t(\Frob)L_{\Frob(\t)}$, where $-L$ denotes $L$ with the opposite orientation. 
\end{theorem}

The homology $H_1(\Upsilon_C,\ZZ)$ is a finite-rank free $\ZZ$-module, carrying an induced Frobenius action and a \emph{length pairing} $\langle\cdot,\cdot\rangle\colon H_1(\Upsilon_C,\ZZ)\otimes H_1(\Upsilon_C,\ZZ)\rightarrow\ZZ$ where $\langle\gamma_1,\gamma_2\rangle$ is the length of the intersection of cycles $\gamma_1$ and $\gamma_2$, interpreted in a suitably signed manner. The rank of $H_1(\Upsilon_C,\ZZ)$ is the potential toric rank of $\Jac C$, and the cokernel of the map $H_1(\Upsilon_C,\ZZ)\rightarrow H^1(\Upsilon_C,\ZZ)$ induced by the length pairing is Frobenius-equivariantly isomorphic to the group of geometric components of the special fibre of the N\'eron model of $\Jac C$.

\begin{theorem}\label{thm:homology_of_dual_graph_semistable}
Let $A$ be the set of even non-\"ubereven clusters except for $\cR$.
\begin{enumerate}
	\item If $\cR$ is not \"ubereven, then $H_1(\Upsilon_C,\ZZ) = \ZZ[A]$ is the free $\ZZ$-module generated by symbols $\ell_\s$ for $\s\in A$.
	\item If $\cR$ is \"ubereven, let $B$ be the set of those clusters $\s\in A$ such that $\s^*=\cR$. Then $H_1(\Upsilon_C,\ZZ)\leq\ZZ[A]$ is the corank $1$ submodule of $\ZZ[A]$ consisting of those elements $\sum_{\s\in A}a_\s\ell_\s$ such that $\sum_{\s\in B}a_\s=0$.
\end{enumerate}
In both cases, Frobenius acts on $H_1(\Upsilon_C,\ZZ)$ is by $\Frob(\ell_\s)=\epsilon_\s(\Frob)\ell_{\Frob(\s)}$, and the length pairing by
\vspace{-.3cm}
\[
\langle\ell_{\s_1},\ell_{\s_2}\rangle = 
\begin{cases}
	0 & \text{if $\s_1^*\neq\s_2^*$,} \\
	2(d_{\s_1\wedge\s_2}-d_{P(\s_1^*)}) & \text{if $\s_1^*=\s_2^*\neq\cR$,} \\
	2(d_{\s_1\wedge\s_2}-d_{\cR}) & \text{if $\s_1^*=\s_2^*=\cR$.}
\end{cases}
\]
\end{theorem}
\begin{theorem}\label{th:BYTreeDualGraph}$\Upsilon_C$ is a double cover of $T_C$ ramified over the blue part, the quotient map being induced by the hyperelliptic involution $\iota$.
Giving edges on $\Upsilon_C$ length 2 makes the identification $\Upsilon_C/\langle\iota\rangle=T_C$ distance preserving.
The preimage of a vertex $v$ in $T_C$ of genus $g(v)>0$ is a vertex in $\Upsilon_C$ corresponding to a component of genus $g(v)$ in the special fibre.
\end{theorem}

\begin{example}\label{ex:dualgraph1}
Consider $C$ over $\QQ_p$ given by the equation
\[
y^2 = x(x-p)(x-2p)(x-3p)(x-1)(x-2)(x-3)(x-4)
\]
for $p\geq5$. Its cluster picture is
\hbox{\clusterpicture            
	\Root[A] {1} {first} {r1};
	\Root[A] {} {r1} {r2};
	\Root[A] {} {r2} {r3};
	\Root[A] {} {r3} {r4};
	\Root[A] {3} {r4} {r5};
	\Root[A] {} {r5} {r6};
	\Root[A] {} {r6} {r7};
	\Root[A] {} {r7} {r8};
	\ClusterLDName c1[][{1}][] = (r1)(r2)(r3)(r4);;
	\ClusterLDName c2[][{0}][] = (c1)(r5)(r6)(r7)(r8);
	\endclusterpicture}. Write $\s$ for the cluster of size 4. 
According to Theorem~\ref{thm:dual_graph_semistable}, the dual graph $\Upsilon_C$ consists of two vertices $v_\s$ and $v_\cR$, connected by two edges $L_\s^{\pm}$. The action of Frobenius on $\Upsilon_C$ fixes the two vertices, and acts on edges via $\Frob(L_\s^{\pm})=L_\s^{\pm\left(\!\frac6p\!\right)}$  where $\left(\!\frac6p\!\right)$ is the Legendre symbol. In other words, the action on $\Upsilon_C$ is trivial if $p\equiv\pm1$ or $\pm5$ mod $24$, and interchanges the two edges $L_\s^+$ and $L_\s^-$ if $p\equiv\pm7$ or $\pm11$ mod $24$. In particular, the Frobenius action on $\Upsilon_C$ can be non-trivial even when the action on $\cR$ is trivial. Pictorially, $\Upsilon_C$ is
\vspace{-0.3cm}
\[
\hbox{\begin{tikzpicture}[scale=\GraphScale]
	\BlackVertices
	\VertexLW[x=0.00,y=0.00,L=\relax]{1}{$v_{\cR}$};
	\VertexLE[x=2.00,y=0.00,L=\relax]{2}{$v_{\s}$};
	\BendedBlackEdges
	\Edge(1)(2)
	\Edge(2)(1)
\end{tikzpicture}}
\hspace{0.8cm}
\raisebox{0.2cm}{\text{or}}
\hspace{0.8cm}
\hbox{\begin{tikzpicture}[scale=\GraphScale]
	\BlackVertices
	\VertexLW[x=0.00,y=0.00,L=\relax]{1}{$v_{\cR}$};
	\VertexLE[x=2.00,y=0.00,L=\relax]{2}{$v_{\s}$};
	\BendedBlackEdges
	\Edge(1)(2)
	\Edge(2)(1)
	\EArrOfs1212{in=270,out=90}{0,-0.2}{0,0.3}{0.5}
	\EArrOfs2121{in=90,out=270}{0,0.2}{0,-0.3}{0.5}
\end{tikzpicture}} \,.
\]

From this, we see that $H_1(\Upsilon_C,\ZZ)=\ZZ$, the induced action of Frobenius is multiplication by $\left(\!\frac6p\!\right)$, and the length pairing is $\langle m,n\rangle = 2mn$. This agrees with Theorem~\ref{thm:homology_of_dual_graph_semistable}.
\end{example}

\begin{example}
Consider $C\colon y^2\!=\!(x-1)(x-2)(x-3)(x-p^2)(x-p^{n+2})(x+p^{n+2})$ over $\QQ_p$ for $p \ge 5$ (cf.\ Example~\ref{BYex3}). Its cluster picture is
\vspace{-0.1cm}
\[
\hbox{\clusterpicture            
	\Root[A] {} {first} {r1};
	\Root[A] {} {r1} {r2};
	\Root[A] {} {r2} {r3};
	\Root[A] {2} {r3} {r4};
	\Root[A] {2} {r4} {r5};
	\Root[A] {} {r5}{r6};
	\ClusterLDName c1[][n][\t] = (r5)(r6);
	\ClusterLDName c2[][2][\s] = (r4)(c1);
	\ClusterLDName c3[][0][\cR] = (r1)(r2)(r3)(c1)(c2);
\endclusterpicture} \,.
\]
According to Theorem~\ref{thm:dual_graph_semistable}, the dual graph $\Upsilon_C$ consists of two vertices $v_{\cR}$ and $v_{\s}$, connected by a single edge $L_\s$ and with a loop $L_\t$ of $2n$ edges connecting $v_{\s}$ to itself. Pictorially, $\Upsilon_C$ is
\vspace{-0.2cm}
\[
\hbox{\begin{tikzpicture}[scale=\GraphScale]
	\draw
	(4,0) node {$2n$};
	\BlackVertices
	\VertexLN[x=1,y=0.000,L=\relax]{1}{$v_{\cR}$};
	\VertexLN[x=2,y=0.00,L=\relax]{3}{$v_{\s}$};
	\VertexLN[x=3,y=0.5,L=\relax]{4}{};
	\VertexLN[x=4,y=0.5,L=\relax]{5}{};
	\VertexLN[x=5,y=0.5,L=\relax]{6}{}; 
	\VertexLN[x=6,y=0,L=\relax]{7}{}; 
	\VertexLN[x=3,y=-0.5,L=\relax]{8}{};
	\VertexLN[x=4,y=-0.5,L=\relax]{9}{};
	\VertexLN[x=5,y=-0.5,L=\relax]{10}{}; 
	\BlackEdges
	\Edge(1)(3)
	\Edge(3)(4)
	\Edge(4)(5)
	\Edge[style={dashed}](5)(6)
	\Edge(6)(7)
	\Edge(7)(10)
	\Edge[style={dashed}](10)(9)
	\Edge(9)(8)
	\Edge(3)(8)
\end{tikzpicture}} \,.
\]
\end{example}

\begin{example}\label{BYTreeEx}
Consider $C\colon y^2=(x^2-p^a)((x-1)^2-p^b)(p^cx^2-1)$ over $\QQ_p$, for some positive integers $a$, $b$, $c$. Its cluster picture is
\[
\hbox{\clusterpicture            
	\Root[A] {} {first} {r1};
	\Root[A] {} {r1} {r2};
	\Root[A] {7} {r2} {r3};
	\Root[A] {} {r3} {r4};
	\Root[A] {7} {r4} {r5};
	\Root[A] {} {r5}{r6};
\ClusterLDName c1[][a/2][\t_1] = (r1)(r2);
\ClusterLDName c2[][b/2][\t_2] = (r3)(r4);
\ClusterLDName c3[][c/2][\s] = (c1)(c2);
\ClusterLDName c4[][-c/2][\cR] = (c3)(r5)(r6);
\endclusterpicture} \,.
\]
We compute the homology $H_1(\Upsilon_C,\ZZ)$ using Theorem~\ref{thm:homology_of_dual_graph_semistable}, without first computing~$\Upsilon_C$. Except for $\cR$ the even non-\"ubereven clusters are the two twins $\t_1$ and $\t_2$, so $H_1(\Upsilon_C,\ZZ)$ is free of rank $2$, generated by $\ell_{\t_1}$ and $\ell_{\t_2}$. Frobenius acts on $H_1(\Upsilon_C,\ZZ)$ by multiplication by $\left(\!\frac{-1}{p}\!\right)$, and the length pairing has matrix $M=\smallmatrix{a+c}{c}{c}{b+c}$.

From this, we see that the potential toric rank of $\Jac C$ is $2$ (potentially totally toric reduction), and that the group of geometric components of the special fibre of the N\'eron model of $\Jac C$ has size $\det(M)=ab+bc+ca$. By computing the Smith normal form of $M$, we find that the group structure is $\ZZ/A\ZZ\oplus\ZZ/B\ZZ$, with $A=\gcd(a,b,c)$ and $B=(ab+bc+ca)/\gcd(a,b,c)$.
\end{example}

\newpage
\section{Special fibre of the minimal regular SNC model (tame case)}\label{se:SFSNC}

Assume $C$ has \underline{tame reduction}. We give a qualitative description of the special fibre of the minimal regular model of $C$ with strict normal crossings (SNC),  over $\mathcal{O}_{\Knr}$. Denote this model $\mathcal{C}^{snc}$,  special fibre $\mathcal{C}^{snc}_{\bar{k}}$. We assume \underline{$\mathcal{R}$ is principal}.\footnote{This only serves to simplify the statements, see the references given for the general case.} 

\begin{notation} \label{spfibdef1}
Let $X$ be an $I_K$-orbit of clusters with $\s\in X$. We say that $X$ is proper/principal/odd/even/\ub/twin/singleton if $\s$ is. If $X'$ is another orbit, write $X'<X$ if $\s'<\s$ for some $\s'\in X'$,  and call $X'$ \textit{stable} if $|X'|=|X|$.  Write $\s_\textup{sing}$ for the set of size $1$ children of $\s$. Define $g_{\textup{ss}}(X)=0$ if $X$ is \ub, and so that $|\tilde{\s}|\in \{2g_{\textup{ss}}(X)+2, 2g_{\textup{ss}}(X)+1\}$ otherwise. For $X$  (henceforth) proper write $d_X=d_\s$, $\delta_X=\delta_\s$ (for $\s\neq \mathcal{R}$), $\lambda_X=\tilde{\lambda}_\s$, and for $X$ even $\epsilon_X=(-1)^{\frac{|X|}{2}(\nu_{\s^*}-|\s^*|d_{\s^*})}\in \{\pm 1\}$.\footnote{Let $I_\s$ be the stabiliser of $\s$ inside the inertia group $I_K$. Then the restriction of  $\epsilon_\s$ to $I_\s$ is a character $I_\s\rightarrow\{\pm 1\}$, and $\epsilon_X=-1$ if and only if this character is non-trivial.}  Let $e_X\in \mathbb{Z}_{\geq 1}$ be minimal with $e_X|X|d_\s\in \mathbb{Z}$ and $e_X|X|\nu_\s\in 2\mathbb{Z}$.  Write $d_X=\frac{a_X}{b_X}$ in lowest terms, and set $b_X'=b_X/\textup{gcd}(|X|,b_X)$. Finally, define $g(X)$ as $\lfloor g_\textup{ss}(X)/b_X'\rfloor$ if $|X|\lambda_X\in \mathbb{Z}$,  $\lfloor g_\textup{ss}(X)/b_X'+1/2\rfloor$ if $|X|\lambda_X\notin \mathbb{Z}$ and $b_X'$ is even, and  $0$ otherwise. 
\end{notation}

\begin{definition} \label{sloped chain defi}
Let $t_1,t_2\in \mathbb{Q}$ and  $\mu\in \mathbb{N}$. Let $n$ be minimal such that there exist~coprime pairs $m_i,d_i\!\in\!\mathbb{Z}$   with
$\mu t_1=\frac{m_0}{d_0}\!>\!\frac{m_1}{d_1}\!>\!...\!>\!\frac{m_{n+1}}{d_{n+1}}\!=\!\mu t_2$ and  with $m_id_{i+1}\!-\!m_{i+1}d_i\!=\!1$  for each $0\!\leq \!i\!\leq \!n.$
A \textit{sloped chain of rational curves with parameters} $(t_2,t_1,\mu)$ is a chain of $\P^1$s $E_1,...,E_n$ with multiplicities $\mu d_i$, intersecting transversally. A \textit{crossed tail} is a sloped chain with $\mu\in 2\mathbb{N}$ and two additional (disjoint) $\P^1$s of multiplicity $\mu/2$ intersecting $E_n$  transversally.  
\end{definition}

\begin{theorem} \label{tame reduction fibre theorem}
 Each principal $I_K$-orbit $X$ of clusters gives rise to two `central' components $\Gamma_X^{\pm}$ of $\mathcal{C}^{snc}_{\bar{k}}$ if $X$ is \ub~ and $\epsilon_X=1$, and one central component $\Gamma_X$ ($=\Gamma_X^+=\Gamma_X^-$) otherwise. These have genus $g(X)$, and multiplicity $|X|e_X$ unless $X$ is \ub~with $\epsilon_X=-1$ when they have multiplicity $2|X|e_X$.  These components are linked by (one or two) sloped chains of rational curves with parameters $(t_2,t_1,\mu)$ indexed by pairs  $X'<X$ with $X$ principal as follows:  \vspace{-12pt}
\begin{center}
\scalebox{0.81}{\begin{tabular}{ |p{1.1cm}|p{0.9cm}|p{0.7cm}|p{1.3cm}p{2.3cm}p{1cm}| p{4.5cm}| }
 \multicolumn{7}{c}{} \\  
 \hline
Name & From  & To & $t_1$ & $t_2$ & $\mu$ & Condition \\ 
 \hline
$L_{X,X'}$  & $ \Gamma_X $  & $\Gamma_{X'}$ &$-\lambda_X$ &  $-\lambda_X-\frac{1}{2}\delta_{X'}$ & $|X'|$ & $X'$ odd principal\\
\hline
$L_{X,X'}^+$&   $\Gamma_X^+$&$\Gamma_{X'}^+$  &  $-d_X$ &$-d_{X'}$ & $|X'|$ & $X'$ even principal, $\epsilon_{X'}=1$\\
 \hline
 $L_{X,X'}^-$&   $\Gamma_X^-$&$\Gamma_{X'}^-$  & $-d_X$ & $-d_{X'}$ & $|X'|$ & $X'$  even principal, $\epsilon_{X'}=1$\\
 \hline
$L_{X,X'}$  & $ \Gamma_X $  &$\Gamma_{X'}$& $-d_X$   &  $-d_{X'}$& $2|X'|$ & $X'$ even principal, $\epsilon_{X'}=-1$\\
\hline
$L_{X'}$&   $\Gamma_X^-$&$\Gamma_{X}^+$  & $-d_{X}$  & $-d_{X}-2\delta_{X'}$ & $|X'|$ & $X'$ twin, $\epsilon_{X'}=1$\\
 \hline
\end{tabular}}
\end{center}
The central components   $\Gamma_X$ with $e_X>1$ are  intersected by  (the first curve of one or more)  sloped chains  with parameters $(\frac{1}{\mu}\lfloor \mu t_1-1\rfloor,t_1,\mu)$ as follows: 
\begin{center}
\addtolength{\leftskip} {-2cm}
 \addtolength{\rightskip}{-2cm}
\scalebox{0.81}{\begin{tabular}{|l|l|ll|l|l|l|}
\hline 
From  & No.  & $t_1$&$\mu$ & \multicolumn{2}{l|}{Condition}\tabularnewline
\hline 
  $ \Gamma_\mathcal{R} $  &$1$&   $(g+1)d_\mathcal{R}-\lambda_\mathcal{R}$ &$1$& \multicolumn{2}{l|}{$|\mathcal{R}|=2g+1$}\tabularnewline
\hline 
  $\Gamma_\mathcal{R}^{\pm}$&$2$  &$-d_\mathcal{R}$ & $1$& \multicolumn{2}{l|}{$|\mathcal{R}|=2g+2, \epsilon_{\mathcal{R}}=1$}\tabularnewline
\hline 
 $\Gamma_\mathcal{R}$&$1$  &$-d_\mathcal{R}$ & $2$&\multicolumn{2}{l|}{$|\mathcal{R}|=2g+2, e_\mathcal{R}>2,~\epsilon_{\mathcal{R}}=-1$}\tabularnewline
\hline 
 $ \Gamma_X $  &$\frac{|X||\s_\textup{sing}|}{b_X}$&   $-\lambda_X$ & $b_X$ & \multicolumn{2}{l|}{$e_X>b_X/|X|$,~$|\s_{\textup{sing}}|\geq 2~~~\forall~ \s\in X$}\tabularnewline
\hline 
  $\Gamma_X$&$1$  &$-d_X$ & $2|X|$& No $X'<X$ is stable,  and either & $\lambda_X\notin \mathbb{Z},~e_X>2$ \tabularnewline
\cline{1-4}    \cline{6-6}
   $\Gamma_X^{\pm}$&$2$  &$-d_X$ & $|X|$&   $X$ \ub~or $g_{\textup{ss}}(X)>0$  &$ \lambda_X\in \mathbb{Z}$\\
\hline 
 $\Gamma_X $&$1$  &$-\lambda_X$ & $|X|$& \multicolumn{2}{l|}{$X$ is not \ub, no odd proper $X'<X$ is stable,}\tabularnewline
 &&&&\multicolumn{2}{l|}{and $g_{\textup{ss}}(X)=0$ or some singleton $X'<X$ is stable}   \tabularnewline
\hline 
\end{tabular}}

%
\end{center}
Finally (regardless of whether $e_X>1$ or not), each $\Gamma_X$ is intersected by the (first curve of) a crossed tail $T_{X'}$ with parameters $(-d_{X'},-d_X+\frac{1}{2|X|},2|X|)$ for each $I_K$-orbit of twins $X'<X$ with $\epsilon_{X'}=-1$ . 
\end{theorem}

\begin{remark} \label{galois and equation remark}
There is also a description of the action of $\textup{Gal}(\bar{k}/k)$ on the special fibre in terms of clusters. Moreover, one can in principle  find equations for the components of the special fibre. We refer to the references below.  
\end{remark}

\begin{example}[A type $\textup{II}^*$ elliptic curve] \label{spfibex1} 
Take 
$E:y^2=x^3-p^5$ over $\mathbb{Q}_p$ for $p\geq 5$, and $\zeta_3$  a primitive $3$rd root of unity.\footnote{The material in this section applies verbatim to elliptic curves of the form $y^2=\textup{cubic}$.} The cluster picture is  
$$
 \clusterpicture            
  \Root[A] {} {first} {r1};
  \Root[A] {2} {r1} {r2};
  \Root[A] {2} {r2} {r3};;
  \ClusterLDName c1[][\frac 53][\cR] = (r1)(r2)(r3);
\endclusterpicture 
\qquad \text{with }\phantom{spa}
\mathcal{R}=\{p^{\frac{5}{3}},\zeta_3p^{\frac{5}{3}},\zeta_3^2p^{\frac{5}{3}}\}, 
$$ 
 $d_\mathcal{R}=5/3$, $\nu_\mathcal{R}=5$, $e_\mathcal{R}=6$ and $\lambda_\mathcal{R}=5/2$.  As $\mathbb{Q}_p(\mathcal{R})/\mathbb{Q}_p$ is tamely ramified, $E$ has tame reduction. 
The  cluster $\mathcal{R}$ is principal and fixed by $I_K$, but the roots  lie in a 
 
  \noindent
\begin{minipage}[t]{.63\textwidth}
single $I_K$-orbit.  The special fibre of the minimal regular SNC model (displayed right) has a
 single central component $\Gamma_\mathcal{R}$ of multiplicity $6$ and genus $0$, intersected by  sloped chains  with parameters $(-1,5/6,1)$, $(-3,-5/2,3)$, and $(-5/2,-5/3,2)$ coming from the $1$st, $4$th, and $5$th rows of the ($2$nd) table in Theorem \ref{tame reduction fibre theorem} respectively. By considering the sequences
\end{minipage}%
\smash{$~$ \raise -32pt\hbox{
\begin{minipage}{0.29\textwidth}
\vskip 3pt
  \includegraphics[height=2.27cm]{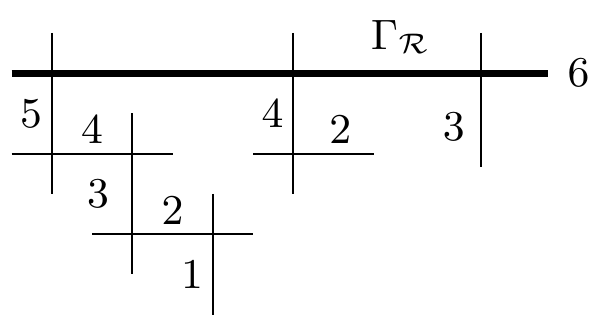}
\end{minipage}}
}
\vskip 2pt
\noindent  

\[\frac{5}{6}>\frac{4}{5}>\frac{3}{4}>\frac{2}{3}>\frac{1}{2}>0>-1,~~\phantom{hi}-\frac{15}{2}>-8>-9,~~~\textup{ and }~~~-\frac{10}{3}>-\frac{7}{2}>-4>-5,\]
which are each minimal\footnote{See \cite[Remark 3.15]{D}  for criteria guaranteeing minimality.} satisfying the determinant condition  of Definition \ref{sloped chain defi}. We find that the special fibre has the pictured form, so the Kodaira type of $E$   is $\textup{II}^*$.
\end{example}
 
 \begin{remark} \label{Kodaira type tentacle remark}
The other Kodaira types arise similarly to the above example, with one central component met by several sloped chains.  
\end{remark}

\begin{example}
Take $C/\mathbb{Q}_p:y^2=((x^2-p)^2-p^4)(x^2+1)(x-1)$, cluster picture
\vspace{-5pt}
$$
 \clusterpicture            
  \Root[A] {} {first} {r1};
  \Root[A] {} {r1} {r2};
  \Root[A] {4} {r2} {r3};
  \Root[A] {2} {r3} {r4};
  \Root[A] {6} {r4} {r5};
  \Root[A] {} {r5}{r6};
  \Root[A] {} {r6}{r7};
  \ClusterLDName c1[][1][\t_1] = (r1)(r2);
  \ClusterLDName c2[][1][\t_2] = (r3)(r4);
  \ClusterLDName c3[][\frac 12][\mathfrak{s}] = (c1)(c2);
  \ClusterLDName c4[][0][\cR] = (c3)(r5)(r6)(r7);
\endclusterpicture 
\qquad \text{with }
\phantom{hi}\cR = \{(p\pm p^2)^{\frac{1}{2}},-(p\pm p^2)^\frac{1}{2},i,-i,1\}.
$$
The special fibre of the minimal regular SNC model (displayed right) has $3$ central 
  \noindent
\begin{minipage}[t]{.65\textwidth}
 components  $\Gamma_\mathcal{R}$ and $\Gamma_\s^{\pm}$ ($\s$ is \ub~and $\epsilon_{\{\s\}}\!=\!1$). The component  $\Gamma_\s^{+}$ (resp. $\Gamma_\s^-$)   intersects  $\Gamma_\mathcal{R}$ as they are linked by a chain with parameters $(0,\frac{1}{2},1)$ which is empty. The $\Gamma_{\s}^{\pm}$ are linked by a chain with parameters $(-1/2,3/2,2)$, consisting of $3$ curves of multiplicity~$2$, coming from the inertia orbit $X\!=\!\{\t_1,\t_2\}$ with $\epsilon_X\!=\!1$.
\end{minipage}%
\smash{$\quad\quad$\raise -28pt\hbox{
\begin{minipage}{0.28\textwidth}
\vskip 3pt
\hskip -.5cm
  \includegraphics[height=2.4cm]{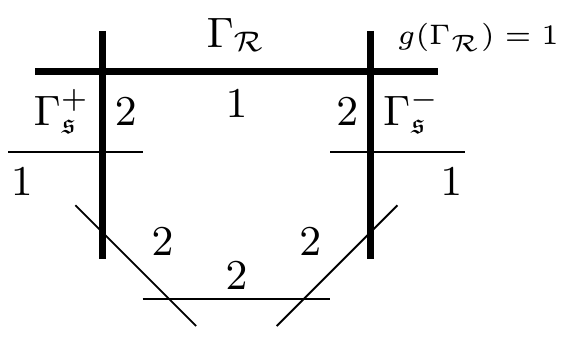}
\end{minipage}}
}
\vskip 2pt
\noindent  
The  $\Gamma^{\pm}_\s$ are each intersected by one further chain with parameters $(-2,-1/2,1)$  arising from the $6$th row of the $2$nd table in Theorem \ref{tame reduction fibre theorem}. 
\end{example}

\noindent\textbf{Erratum.}
%
%
 In \cite{FN} Theorems 1.12 and 7.18, there is a typo : the column ``$t_2$'' in the
 second table of each should  be ``$t_2 + \delta$''. Similarly in Theorem 6.3 in the first table. 
 \begin{minipage}[t]{.65\textwidth}
For example the curve $C:y^2=(x^3-p^2)(x^4-p^{11})$ has 
cluster picture  \scalebox{0.8}{\clusterpicture
		\Root[A] {} {first} {r1};
		\Root[A] {} {r1} {r2};
		\Root[A] {} {r2} {r3};
		\Root[A] {} {r3} {r4}
		\ClusterLDName c1[][\frac{11}{4}][] = (r1)(r2)(r3)(r4);
		\Root[A] {} {c1} {r5};
		\Root[A] {} {r5} {r6};
		\Root[A] {} {r6} {r7};
		\ClusterLDName c2[][\frac23][] = (c1)(r5)(r6)(r7);
		\endclusterpicture}~ and special fibre shown on the right.
\end{minipage}%
\smash{$\quad\quad$\raise -28pt\hbox{
\begin{minipage}{0.35\textwidth}
\vskip -15pt
\hskip -.5cm
  \includegraphics[height=2cm]{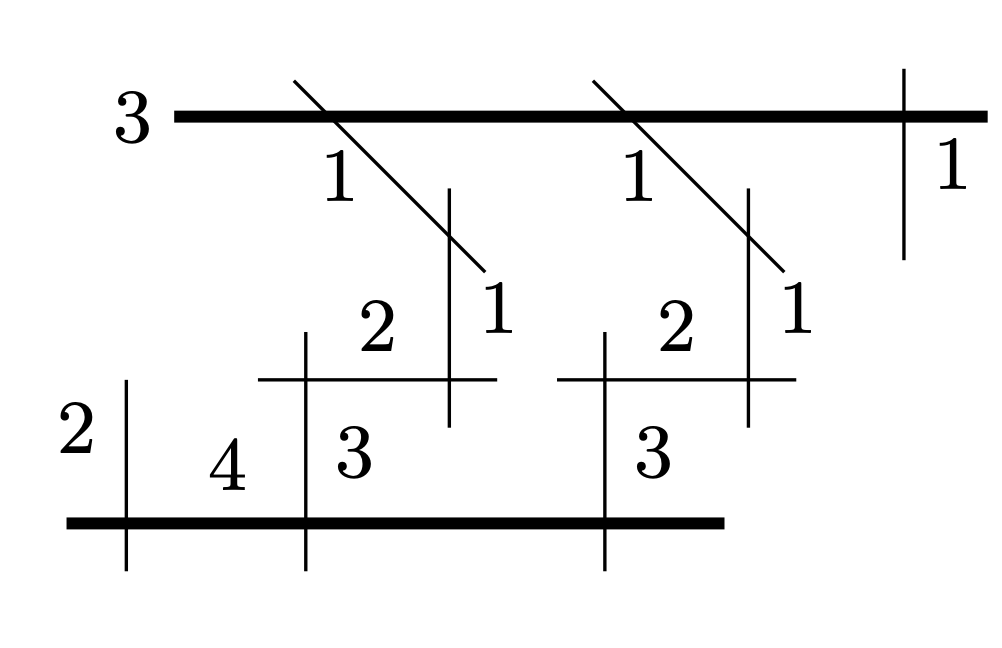}
\end{minipage}}
}

\newpage
\section{Tamagawa number (semistable case)}\label{se:tama}

Let $C/K : y^2 = f(x)$ be a \underline{semistable} hyperelliptic curve. The \emph{Tamagawa number} $c_{\Jac C}$ of the Jacobian of $C$ is the number of $k$-points of the component group-scheme of the special fibre of the N\'eron model of $\Jac C$. We explain how to read off $c_{\Jac C}$ from the cluster picture of $C$.

\begin{theorem}\label{thm:easyTamagawaNumbers}Suppose that $C$ has no \"ubereven clusters. For even clusters $\s\neq\cR$ write $c_\s=\leftchoice{2\delta_\s}{\text{if }\epsilon_\s(\Frob^{q_{\s}})=+1}{\gcd(2\delta_\s,2)}{\text{if }\epsilon_\s(\Frob^{q_\s})=-1}$, where $q_\s$ is the size of the $\Frob$-orbit of $\s$. 
The Tamagawa number of $\Jac C$ is given by $$c_{\Jac C} = \prod\nolimits_{\s} c_\s,$$ the product 
taken over representatives of $\Frob$-orbits of even clusters $\s\neq\cR$.
\end{theorem}

In general, when $C$ has \"ubereven clusters, the formula becomes significantly more complicated, and is best phrased in the language of BY trees.

\begin{notation}
	Let $T=T_C$ be the BY tree associated to $C$ (Definition~\ref{StoTC}), with edge-length function $\delta$.  
Let $B$ be the subgraph of $T$ consisting of blue vertices and blue edges, and $(F,\epsilon)\in\mathrm{Aut}\,T$ (see Definition \ref{actionb}) the induced action of $\mathrm{Frob}$ on $T$.

	For a vertex $v \in T \setminus B$ we write $q_v$ for the size of the $F$-orbit containing~$v$. We write $\epsilon_v = \prod_{j=0}^{q_v - 1} \epsilon(F^jX)$, where $X$ is the connected component of $T\setminus B$ containing~$v$. If $e \in T \setminus B$ is an edge then we define $q_e$ and $\epsilon_e$ similarly. We write $\hat{B} \subseteq T$ for the subgraph consisting of $B$ together with all vertices $v$ with $\epsilon_v = -1$ and edges $e$ with $\epsilon_e = -1$. Finally, we write $B' \subseteq \hat{B}' \subseteq T'$ for the respective quotients of $B \subseteq \hat{B} \subseteq T$ by the action of $F$, with length function $\delta'(e')=\delta(e)$ and with $q_{e'}=q_e$ for any edge $e \in T$ mapping to $e'$.
\end{notation}

\begin{thm}
\label{thm:TamagawaNumbers}
The Tamagawa number of $\mathrm{Jac}\,C$ is given by 
\[c_{\Jac C}= Q \cdot \tilde{c} \cdot \sum_{\{e_1',\dots,e_{r}'\} \in R}~~\prod_{j=1}^{r} \frac{\delta'(e_j')}{q_{e_j'}}, \qquad \text{where:}\]
\begin{enumerate}[leftmargin=*]
	\item $Q$ is the product of the sizes of all $F$-orbits of connected components of $\hat{B}$;
	\item $\tilde{c} \!=\! \prod_X \tilde{c}(X)$ is a product over the connected components $X$ of $\hat{B}'\setminus B'$ with
	\begin{enumerate}[leftmargin=12pt]
		\item $\tilde{c}(X) \!=\! 2^{\alpha-1}$ if the closure of $X$ contains $\alpha > 0$ points of $B'$ lying an even distance from a vertex of $\hat B'$ of degree $\geq 3$;
		\item $\tilde{c}(X) \!=\! \gcd(l,2)$ if the closure of $X$ consists of $2$ points of $B'$ distance $l$ apart;
		\item $\tilde{c}(X) \!= \!\gcd(b,2)$ otherwise, where $b$ is the number of points of $B'$ in the closure of $X$;
	\end{enumerate}
	\item $r = \# \pi_0(\hat{B}') - 1$ is the number of connected components of $\hat{B}'$, minus 1;
	\item $R$ is the set of unordered $r$-tuples
	of edges of $T'\setminus \hat{B}'$ whose removal disconnects the $r+1$ components of $\hat{B}'$
	from one another.
\end{enumerate}

\end{thm}
\begin{remark} Theorem \ref{thm:easyTamagawaNumbers} follows from Theorem \ref{thm:TamagawaNumbers}.
Since there are no \ub\ clusters there are no yellow vertices, and hence all yellow edges are disjoint and in bijection with even clusters. This implies that the closures of connected components of $\hat{B}' \setminus B'$ will always consist of two vertices, distance $2\delta_\s$ apart, and the $\tilde{c(X)}$ all fall in situation (2)(b). This is the contribution of orbits of even clusters with $\epsilon_\s(\Frob^{q_\s}) = -1$ in the formula of \ref{thm:TamagawaNumbers}. Furthermore, $R$ has size 1, the $r$-tuple of edges of $T'\setminus \hat{B}' $ which correspond to orbits of even clusters with $\epsilon_\s(\Frob^{q_\s})=1$, and so  $Q \prod \delta'(e_j')/q_{e_j'}$ is the contribution from clusters with $\epsilon_\s(\Frob^{q_\s})=1$. 
\end{remark}
\newpage

\begin{ex}
	\label{eg:concreteTamagawa}
	Consider $$C:y^2=(x^2-5)(x-1)(x-2)(x+1),$$ over $\Q_5$. Its cluster picture is 
	\clusterpicture            
  		\Root[A] {2} {first} {r1};
  		\Root[A] {} {r1} {r2};
  		\ClusterLDName c1[][1/2][\mathfrak{s}] = (r1)(r2);
  		\Root[A] {} {c1} {r3};
  		\Root[A] {} {r3} {r4};
  		\Root[A] {1} {r4} {r5};
  		\ClusterLD c2[][0] = (c1)(r5);
	\endclusterpicture
	with $\cR=\{\sqrt5,-\sqrt5,1,2,-1\}$ and $\s=\{\sqrt5,-\sqrt5\}$. 
	Then $\theta_\s=\sqrt{2}$, and $\epsilon_\mathfrak{s}(\mathrm{Frob})=-1$, 
	as $\sqrt{2}\notin\Q_5$. According to Theorem~\ref{thm:easyTamagawaNumbers}, the Tamagawa number of $\Jac C$ is $\gcd(1,2)=1$. The same value can be read off from the more general Theorem~\ref{thm:TamagawaNumbers}, using that the BY tree of $C$ is
	\smash{\raise-10pt\hbox{\begin{tikzpicture}[scale=\GraphScale]
		\InfVertices
		\VertexLN[x=0.75,y=0,L=\relax]{3}{}
  		\BlueVertices
  		\VertexLS[x=1.50,y=0.000,L=\relax]{1}{};
  		\VertexLN[x=0.000,y=0.000,L=1]{2}{};
  		\YellowEdges
  		\Edge(1)(2)
	        \TreeEdgeSignS(1)(2){0.5}{1}
		\TreeEdgeSignN(1)(2){0.5}{-}
	\end{tikzpicture}}} with trivial $F$-action.
\end{ex}

\begin{example}
Suppose that the cluster picture of $C$ is \clusterpicture            
  		\Root[A] {2} {first} {r1};
  		\Root[A] {} {r1} {r2};
  		\ClusterLD c1[][a/2] = (r1)(r2);
  		\Root[A] {} {c1} {r3};
  		\Root[A] {} {r3} {r4};
  		\ClusterLD c2[][b/2] = (r3)(r4);
  		\Root[A] {1} {c2} {r5};
  		\Root[A] {1} {r5} {r6};
  		\ClusterLD c3[][c/2] = (r5)(r6);
  		\ClusterLD c4[][0] = (c1)(c2)(c3);
	\endclusterpicture, with $\Frob$ acting trivially on clusters and $\epsilon_\s(\Frob)=+1$ for all clusters. In particular, $\hat{B} = B$ and the quotients $B', T'$ can be identified respectively with $B, T$. Using Theorem~\ref{thm:TamagawaNumbers}, we find that the Tamagawa number of $\Jac C$ is $ab+bc+ca$ (we leave details to the reader).
\end{example}

\begin{example}
	\label{eg:abstractposeps}
	Let $C/\QQ_p$ be a curve with BY tree as below. A concrete example would be $C/\QQ_p$ with $p \equiv 3 \mod 4$ and $$C: y^2 = ((x^2+1)^2 - p^{u})((x-1)^2 - p^{z})(x-p^{w/2})(x-p^{w/2+2})((x^2+p^{w+4})^2 - p^{2(w + 4) + a}),$$ with $w \equiv 2 \mod 4$ and $a > w+4$.
	\vskip -8pt
	\begin{figure}[ht]
	\begin{tikzpicture}[scale=\GraphScale]
		\InfVertices
		\VertexLN[x=0.25,y=-0.43,L=\relax]{100}{};
		\VertexLN[x=2.75,y=-0.43,L=\relax]{101}{};
		\BlueVertices
		\VertexLNW[x=3.00,y=0.800,L=\relax]{2}{};
		\VertexLNW[x=4, y=.2, L=\relax]{6}{};
		\VertexLNE[x=0,y=0.80,L=\relax]{4}{};
		\VertexLNE[x=0.000,y=-0.400,L=\relax]{3}{};
		\VertexLNW[x=3, y=-0.4, L=\relax]{5}{};
		\VertexLW[x=6, y=0.8, L=\relax]{7}{};
		\VertexLW[x=6,y=-0.4,L=\relax]{8}{};
		\YellowVertices
		\Vertex[x=1.50,y=0.150,L=\relax]{1};
		\YellowEdges
		\Edge(1)(3)
		\Edge(1)(2)
		\Edge(1)(4)
		\Edge(1)(5)
		\Edge(6)(7)
		\Edge(6)(8)
		 \TreeEdgeSignS(1)(3){0.7}{z}
		 \TreeEdgeSignS(1)(2){0.7}{u}
		\TreeEdgeSignS(1)(4){0.7}{u}
		\TreeEdgeSignS(1)(5){0.7}{w}
		\TreeEdgeSignS(1)(7){.9}{a}
		\TreeEdgeSignS(1)(8){.9}{a}
		\TreeEdgeSignN(1)(7){.9}{-}
		\TreeEdgeSignN(1)(8){.9}{-}
		\TreeSignAt(1)(0,.3){+}
		\BlueEdges
		\Edge(5)(6)
		\ESwap1214{in=60, out=120}
		\ESwap6768{in=90, out=-90}
	\end{tikzpicture}
	\end{figure}
	\vskip -10pt
	
	Label the edges $e_u^{\pm}, e_w, e_z, e_a^{\pm}$ where $e_w$ has length $w$ and so on. Since $\epsilon_v = \epsilon_e = 1$ for all vertices $v$ and edges $e$, $\hat{B} = B$, and $T'$ and $B'$ are given by the following picture, with $\hat{B}' = B'$:
	
	\vskip -14pt
	\begin{figure}[ht]
		\begin{tikzpicture}[scale=\GraphScale]
		\BlueVertices
		\VertexLNW[x=4, y=0, L=\relax]{6}{};
		\VertexLE[x=1.5,y=0.80,L=\relax]{4}{};
		\VertexLNE[x=6, y=0, L=\relax]{7}{};
		\VertexLNE[x=0.000,y=-0.400,L=\relax]{3}{};
		\VertexLNW[x=3, y=-0.4, L=\relax]{5}{};
		\YellowVertices
		\Vertex[x=1.50,y=0.000,L=\relax]{1};
		\YellowEdges
		\Edge(1)(3)
		\Edge(1)(4)
		\Edge(1)(5)
		\Edge(6)(7)
		\BlueEdges
		\Edge(5)(6)
		\TreeEdgeSignS(1)(5){.5}{w}
		\TreeEdgeSignS(1)(3){.5}{z}
		\TreeEdgeSignW(1)(4){.5}{u}
		\TreeSignAt(1)(0,-.35){+}
		\TreeEdgeSignS(6)(7){.5}{a}
		\TreeEdgeSignN(6)(7){.5}{+}
		\end{tikzpicture}
	\end{figure}
	\vskip -10pt
	
 There are four $F$-orbits in $\hat{B}$, two of size $1$ and two of size $2$. Therefore $Q=4$. The set $\hat{B}'\setminus B'$ is empty and so $\tilde{c} = 1$. Finally $r=3$, and the set $R = \left\{\{e_a', e_u', e_z'\}, \{e_a', e_z', e_w'\}, \{e_a', e_w', e_u'\}\right\}$ where $e_a'$ is the image of $e_a^{\pm}$ and so on. Putting this together we see $$ c_{\Jac C} = 2 \cdot 1 \cdot \left(\frac{a}{2}\cdot \frac{u}{2} \cdot z + \frac{a}{2}\cdot z \cdot w + \frac{a}{2}\cdot\frac{u}{2}\cdot w\right) = a(uz + 2zw + uw).$$
\end{example}

\begin{example}
	\label{eg:abstractnegeps}
	
	Consider the BY tree as in Example \ref{eg:abstractposeps}, but where $\epsilon = -1$ for each component instead of $1$. The edges $e_a^{\pm}$ and $e_u^{\pm}$ lie in $F$-orbits of size $2$ so $\epsilon_{e_a^{\pm}} = \epsilon_{e_b^{\pm}} = 1$, and $e_z$ and $e_w$ lie in an orbit of size $1$ so $\epsilon_{e_w} = \epsilon_{e_z} = -1$. The graphs $B'$ and $T'$ are as above, and $\hat{B}'$ is given by
	
	\vskip -10pt
	\begin{figure}[ht]
		\begin{tikzpicture}[scale=\GraphScale]
		\BlueVertices
		\VertexLNW[x=4, y=0, L=\relax]{6}{};
		\VertexLNE[x=1.5,y=0.80,L=\relax]{4}{};
		\VertexLNE[x=0.000,y=-0.400,L=\relax]{3}{};
		\VertexLNW[x=3, y=-0.4, L=\relax]{5}{};
		\VertexLNW[x=6, y=0, L=\relax]{7}{};
		\YellowVertices
		\Vertex[x=1.50,y=0.000,L=\relax]{1};
		\YellowEdges
		\Edge(1)(3)
		\Edge(1)(5)
		\BlueEdges
		\Edge(5)(6)
		\TreeEdgeSignS(1)(5){.5}{w}
		\TreeEdgeSignS(1)(3){.5}{z}
		\end{tikzpicture}
	\end{figure}
	\vskip -10pt
	
	There are three $F$-orbits of components in $\hat{B}$, one of size $1$ and two of size $2$, and hence $Q = 4$. There is one connected component $X' \in \hat{B}' \setminus B'$ and so $\tilde{c} = \tilde{c}(X')$ is non-trivial. We have assumed that $w$ is even, and so $\tilde{c} = \gcd(z+w, 2) = \gcd(z,2)$ as $X'$ consists of 2 points of $B'$ a distance $z+w$ apart. Finally, $r=2$ and $R = \{\{e_a, e_u\}\}$. Putting this all together $$c_{\Jac C} = 4 \cdot \frac{a}{2} \cdot \frac{u}{2} \cdot \gcd(z,2) = a u \gcd(z,2).$$

\end{example}

\newpage

\section{Galois representation}
\label{se:GalRep}

In this section, we will describe the Galois action on the $\ell$-adic \'{e}tale cohomology of the curve (equivalently its Jacobian) when $\ell \neq p$. For an arbitrary curve (or abelian variety), there always exists a decomposition of $\ell$-adic Galois representations $$\H(C/\bar{K},\Ql) = \H(\Jac C/\bar{K},\Ql) = H^1_{ab} \oplus \left( H^1_t \otimes \sp \right)$$ into ``abelian'' and ``toric'' parts, where for $\sigma \in I_K$, $\sp(\Frob^n\sigma)=\begin{pmatrix} 1 & t_{\ell}(\sigma) \\ 0 & q^{-n}\end{pmatrix}$ for a choice of tame $\ell$-adic character $t_{\ell}$, and with $q = |k|$. We will describe the abelian and toric parts in terms of the cluster picture.

\begin{theorem}
\label{GalRepthm}
Let $C/K$ be a hyperelliptic curve and let $\ell \neq p$ be prime. Let
\begin{eqnarray*}
	X &=& \{ \text{proper,  non-\"{u}bereven clusters } \s \}, \\
	Y &=& \{ \text{principal,  non-\"{u}bereven clusters } \s \} \subseteq X.
\end{eqnarray*}
Write $\bar{\epsilon}_\s$ for the restriction of $\epsilon_\s$ to $G_\s$.
\begin{enumerate} 
\item For all $\s \in Y$, there exists a continuous $\ell$-adic representation, $V_{\s}$, with finite image of inertia such that:
\begin{eqnarray*}
	H^1_{ab} &=& \bigoplus\limits_{\s \in Y/G_K} \Ind_{G_{\s}}^{G_K} V_{\s}, \\
	H^1_t &=& \bigoplus\limits_{\s \in X/G_K} \Ind_{G_{\s}}^{G_K} \bar{\epsilon}_{\s} \qquad \ominus \bar{\epsilon}_{\cR},	
\end{eqnarray*}
where $G_{\s}=\Stab_{G_K}(\s)$ is the Galois stabiliser of $\s$, and $\ominus$ is the inverse of the direct sum $\oplus$.

\item Let  $I_{\s}=\Stab_{I_K}(\s)$ be the inertia stabiliser of $\s$ and let $\gamma_{\s}: I_{\s} \rightarrow \bar\Q_{\ell}^{\times}$ be any character whose order is the prime-to-$p$ part of the denominator of $[I_K:I_{\s}] \tilde\lambda_{\s}$. Then for all $\s \in Y$, there is an isomorphism $$V_{\s} \cong \gamma_{\s} \otimes (\Ql[\tilde\s] \ominus \triv) \,\,\,\, \ominus \bar{\epsilon}_{\s} \qquad \text{ as $I_{\s}$-representations,}$$
where $\tilde{\s}$ is the set of odd children of $\s$ with $I_{\s}$-action.
\end{enumerate}
\end{theorem}

\begin{remark}\label{ssrepp}
The full Galois module structure of $V_{\s}$ cannot be determined by the cluster picture alone; indeed two curves with good reduction can have the same cluster picture but different Galois representations. It is however computable via a form of point-counting over finite fields; see \cite[Theorem 1.5 and Example 1.9]{skew} for the statement and a worked example.

On the other hand, Theorem \ref{GalRepthm}(2) gives an explicit description of the inertia representation. For tame curves, it is completely determined by the underlying abstract cluster picture (in the sense of Section \ref{se:iso}) without needing to know the inertia action on the roots a priori.
\end{remark}

\begin{remark}
\label{ssrep}
The Jacobian $\Jac C$ (equivalently $C$) is semistable if and only if both $H^1_{ab}$ and $H^1_t$ are unramified. If moreover $H^1_t$ is the zero representation, then this is equivalent to $\Jac C$ having good reduction. Recall from Section \ref{se:redtype} that these conditions are easy to read off from the cluster picture of $C$. 
\end{remark}

\newpage

\begin{notation} 
For a cluster $\s$, we let $I_{\s}$ denote the inertia stabiliser. If $n$ is coprime to $p$, we further write $\bar\Q_{\ell}[C_{n,\s}]$ to mean the $I_{\s}$-representation $\bar\Q_{\ell}[I_{\s}/I_n]$ where $[I_{\s}:I_n]=n$, and let $\chi_{n,\s}$ be a fixed faithful character of $I_{\s}/I_n$. We shall omit the cluster subscript when $\s=\cR$.

\end{notation}

\begin{example}
Let $\zeta_3$ be a primitive cube root of unity and consider the curve $C/\Q_7: y^2=x((x-7^{1/2})^3-7^{5/2})((x+7^{1/2})^3+7^{5/2})$, with cluster picture
$$\clusterpicture
  \Root[A] {} {first} {r1};
  \Root[A] {} {r1} {r2};
  \Root[A] {} {r2} {r3};
  \Root[A] {4} {r3} {r4};
  \Root[A] {} {r4} {r5};
  \Root[A] {} {r5}{r6};
  \Root[A] {3} {r6}{r7};
  \ClusterLDName c1[][\frac{1}{3}][\s_1] = (r1)(r2)(r3);
  \ClusterLDName c2[][\frac{1}{3}][\s_2] = (r4)(r5)(r6);
  \ClusterLDName c3[][\frac{1}{2}][\cR] = (c1)(c2)(r7);
\endclusterpicture, \s_1=\{7^{1/2}{+}\zeta_3^j 7^{5/6} \, | \, j=0,1,2 \}, \s_2=\{-7^{1/2}{-}\zeta_3^j 7^{5/6} \, | \, j=0,1,2\},$$
and $\cR =\s_1 \cup \s_2 \cup \{0\}$. In this case inertia acts on $\cR$ through a $C_6$-quotient and permutes $\s_1$ and $\s_2$. We shall compute the inertia action on $\H(C/\bar\Q_7,\bar\Q_{\ell})$. Note that every cluster is odd: this implies that there is no contribution from the toric part, i.e. $H^1_t=0$, and also that $V_{\s} \cong \gamma_{\s} \otimes (\bar\Q_{\ell}[\tilde\s] \ominus \triv)$ by definition of $\epsilon_{\s}$. Moreover, every proper cluster is principal and hence we choose representatives for $Y/I_{\Q_7}$ to be $\s_1$ and $\cR$.

First consider $V_{\cR}$. We compute that $\tilde\lambda_{\cR}=3/4$ hence $\gamma_{\cR}$ is an order $4$ character $\chi_4$. Therefore $\bar\Q_{\ell}[\tilde\cR] \cong \triv \oplus \bar\Q_{\ell}[C_2]$ and hence $V_{\cR} \cong \chi_4 \otimes \bar\Q_{\ell}[C_2] \cong \chi_4 \oplus \chi_4^{-1}$.

Next we compute $V_{\s_1}$. In this case, $\tilde{\lambda}_{\s_1}=9/4$ and hence $\gamma_{\s_1}$ is an order $2$ character $\chi_{2,\s_1}$ since $[I_{\Q_7}:I_{\s_1}]=2$. Now $I_{\s_1}$ cyclically permutes the children of $\s_1$ so $\bar\Q_{\ell}[\tilde{\s}_1] \cong \bar\Q_{\ell}[C_{3,\s_1}] \cong \triv \oplus \chi_{3,\s_1} \oplus \chi_{3,\s_1}^{-1}$, hence $V_{\s_1} \cong \chi_{6,\s_1} \oplus \chi_{6,\s_1}^{-1}$.

We must now induce this to $I_{\Q_7}$; since $I_{\Q_7}/I_{\s_1} \cong C_2$ we find that $\Ind_{I_{\s_1}}^{I_{\Q_7}} V_{\s_1} \cong \chi_{12} \oplus \chi_{12}^5 \oplus \chi_{12}^7 \oplus \chi_{12}^{11}$. One therefore has that for all $\ell \neq 7$, $$\H(C/\bar\Q_7,\bar\Q_{\ell}) = \chi_4 \oplus \chi_4^{-1} \oplus \chi_{12} \oplus \chi_{12}^5 \oplus \chi_{12}^7 \oplus \chi_{12}^{11}, \quad \text{ as $I_{\Q_7}$-representations.}$$
\end{example}

\begin{example}
Let $C/\Q_3$ be the curve $y^2=(x-1)((x-3)^2+81)((x+3)^2+81)$, whose cluster picture is:
$$
\clusterpicture
  \Root[A] {} {first} {r1};
  \Root[A] {} {r1} {r2};
  \Root[A] {4} {r2} {r3};
  \Root[A] {} {r3} {r4};
  \Root[A] {4} {r4} {r5};
  \ClusterLDName c1[][1][\mathfrak{t}_1] = (r1)(r2);
  \ClusterLDName c2[][1][\mathfrak{t}_2] = (r3)(r4);
  \ClusterLDName c3[][1][\s] = (c1)(c2);
  \ClusterLDName c3[][0][\cR] = (c3)(r5);
\endclusterpicture \quad \text{ where } \mathfrak{t}_1=\{3\pm 9i\}, \mathfrak{t}_2=\{-3 \pm 9i \}, \s=\mathfrak{t}_1 \cup \mathfrak{t}_2, \cR = \mathfrak{s} \cup \{1\},$$ and $i$ is a fixed square root of $-1$. One can check that $C$ is semistable (see Theorem \ref{semistability criterion subsection}) and we shall confirm this on Galois representation side via Remark \ref{ssrep}.

Note that Galois acts trivially on the proper clusters and the only \"{u}bereven cluster is $\s$ so $X/G_{\Q_3}=\{\mathfrak{t}_1, \mathfrak{t}_2, \cR \}$. Moreover, none of these are principal ($\cR$ is a cotwin) hence the abelian part is 0; this is expected as the Jacobian has totally toric reduction (cf. Theorem \ref{potential toric rank}), and so $\H(C/\bar\Q_3,\Ql) = (\epsilon_{\mathfrak{t}_1} \oplus \epsilon_{\mathfrak{t}_2}) \otimes \sp$.

Using $z_{\mathfrak{t}_1}=3$ as a centre of $\mathfrak{t}_1$, one computes that $\theta_{\mathfrak{t}_1}^2=234 = 2 \cdot 3^2 \cdot 13$. This implies that $\theta_{\mathfrak{t}_1}$ is fixed by inertia and negated by Frobenius and therefore $\epsilon_{\mathfrak{t}_1}$ is the unramified quadratic character $\eta$. Similarly, we find that $\theta_{\mathfrak{t}_2}^2=-468$ (using the centre $-3$) and hence $\epsilon_{\mathfrak{t}_2}=\epsilon_{\mathfrak{t}_1}=\eta$. Since $H^1_{ab}$ and $H^1_t$ are both unramified, the curve is semistable (as expected) and we have that for all $\ell \neq 3$, $$\H(C/\bar\Q_3,\Ql) = \eta^{\oplus 2} \otimes \sp \qquad \text{as $G_{\Q_3}$-representations.}$$
\end{example}

\newpage
\section{Conductor}
\label{se:Cond}

In this section, we describe the conductor exponent of $\Jac C$, which we shall denote by $n_C$.
\begin{theorem}
\label{sscond}
Suppose $C/K$ is semistable. Then $$n_C=\begin{cases} \#A -1 & \text{ if $\cR$ is \"{u}bereven,} \\ \#A & \text{ else,} \end{cases}$$ where $A=\{\text{even clusters } \s  \neq \cR \, \mid \, \s \text{ is not \"{u}bereven} \}.$
\end{theorem}
\medskip
For general $C$ the formula for the conductor is more involved. 
\begin{notation} For a proper cluster $\s$ we define $\xi_{\s}(a)$ to be the $2$-adic valuation of the denominator of $[I_K:I_{\s}]a$, where $I_{\s}$ is the stabiliser of $\s$ under $I_K$, with the convention that $\xi_{\s}(0)=0$. More formally, it is $\xi_{\s}(a) = \max \{ -\ord_2([I_K:I_{\s}]a) ,0\}$.

For a cluster picture associated to a curve $C/K$, we further define
\begin{eqnarray*}
	U &=& \{ \text{odd clusters } \s \neq \cR \mid \xi_{P(\s)}(\tilde\lambda_{P(\s)}) \leq \xi_{P(\s)}(d_{P(\s)}) \}, \\
	V &=& \{ \text{proper non-\"{u}bereven clusters } \s \mid \xi_{\s}(\tilde\lambda_{\s})=0 \}.
\end{eqnarray*}

\end{notation}

\medskip
\begin{theorem}
\label{allcond}
Let $C/K$ be a hyperelliptic curve. Decompose the conductor exponent $n_C$ of $\operatorname{Jac} C$ into tame and wild parts as $n_C=n_{\mathrm{tame}}+n_{\mathrm{wild}}$. Then:\medskip
\begin{enumerate}[leftmargin=*]
\item $n_{\mathrm{tame}} = 2g - \#(U/I_K) + \#(V/I_K) + \begin{cases} 1  &\text{ if $|\cR|$ and $v(c)$ are even,} \\ 0 &\text{ else;} \end{cases}$
\medskip
\item $n_{\mathrm{wild}} = \sum\limits_{r \in \cR/G_K} \left( v(\Delta_{K(r)/K}) - [K(r):K] + f_{K(r)/K} \right),$\\ where $\Delta_{K(r)/K}$ and $f_{K(r)/K}$ are the discriminant and residue degree of $K(r)/K$ respectively.
\end{enumerate}
\end{theorem}
\bigskip
\begin{remark}
\label{tamecond}
	If $p>2g+1$, then $C$ is tame so that $n_{\mathrm{wild}}=0$ and $n_C=n_{\mathrm{tame}}$. Moreover, in this case $n_C$ is completely determined by the underlying abstract cluster picture (in the sense of Section \ref{se:iso}) without needing to know the Galois action on the roots a priori.
\end{remark}

\newpage

\begin{example}
Let $C/\Q_p: y^2=(x^2-p^2)((x-1)^2-p^2)((x-2)^2-p^2)$ with cluster picture
$$
\clusterpicture
  \Root[A] {} {first} {r1};
  \Root[A] {} {r1} {r2};
  \Root[A] {3} {r2} {r3};
  \Root[A] {} {r3} {r4};
  \Root[A] {3} {r4} {r5};
  \Root[A] {} {r5} {r6};
  \ClusterLDName c1[][1][\s_1] = (r1)(r2);
  \ClusterLDName c2[][1][\s_2] = (r3)(r4);
  \ClusterLDName c3[][1][\s_3] = (r5)(r6);
  \ClusterLDName c4[][0][\cR] = (c1)(c2)(c3);
\endclusterpicture \qquad \text{and  } \cR=\{p,-p,1+p,1-p,2+p,2-p\}.
$$

One can check that the curve is semistable (Theorem \ref{semistability criterion subsection}) so we can apply Theorem \ref{sscond}. Observe that $A=\{\s_1, \s_2, \s_3 \}$ from which we obtain that $n_C=2$ since $\cR$ is \"{u}bereven. 
\end{example}

\begin{example}
Let $C/\Q_5: y^2=x^5+256$ and let $\zeta_5$ be a primitive 5th root of unity. Then the cluster picture is 
$$
\clusterpicture
  \Root[A] {} {first} {r1};
  \Root[A] {} {r1} {r2};
  \Root[A] {} {r2} {r3};
  \Root[A] {} {r3} {r4};
  \Root[A] {} {r4} {r5};
  \ClusterLDName c4[][\frac{1}{4}][\cR] = (r1)(r2)(r3)(r4)(r5);
\endclusterpicture \qquad \text{with  } \cR=\{\zeta_5^j \sqrt[5]{-256} \, \mid \, j=0,\cdots, 4\}.
$$

We begin with $n_{\mathrm{wild}}$ and observe that the roots form a single orbit under inertia. For all $r \in \cR$, we have that $\Q_5(r)/\Q_5$ has discriminant $50000$, degree $5$, residue degree $1$ hence $n_{\mathrm{wild}}=1$.

It remains to compute $n_{\mathrm{tame}}$. Now $\tilde\lambda_{\cR}=\frac{5}{8}$ so $\xi_{\cR}(\tilde\lambda_{\cR})=3$ and $\xi_{\cR}(d_{\cR})=2$ hence $U$ and $V$ are both empty. Therefore $n_{\mathrm{tame}}=2g=4$ and $n_C=4+1=5$.
\end{example}

\begin{example}
In this example we compute the conductor directly from the cluster picture without reference to a curve. Let $p \geq 7$ and let $C/K$ be a genus two hyperelliptic curve with $c=1$, with cluster picture
$$
\clusterpicture
  \Root[A] {} {first} {r1};
  \Root[A] {} {r1} {r2};
  \Root[A] {} {r2} {r3};
  \Root[A] {} {r3} {r4};
  \Root[A] {3} {r4} {r5};
  \ClusterLDName c4[][\frac{1}{4}][\s] = (r1)(r2)(r3)(r4);
  \ClusterLDName c5[][0][\cR] = (c4)(r5);
\endclusterpicture \qquad \text{where } \s=\{r_1,r_2,r_3,r_4\},\quad \cR=\s \cup \{r_5\}$$ and inertia acts by cyclically permuting the roots\footnote{This is actually the only possible action due to the depths. An example of such a curve is $C/\Q_7: y^2=x^5+5x^4+40x^3+80x^2+256x$.} in $\s$.

Now we compute that $\tilde\lambda_{\cR}=0$ and $\tilde\lambda_{\s}=\frac{1}{2}$ hence $\xi_{\cR}(\tilde\lambda_{\cR})=\xi_{\cR}(d_{\cR})=0$, $\xi_{\cR}(\tilde\lambda_{\s})=1$ and $\xi_{\cR}(d_{\s})=2$. This means that $U/I_{\Q_5}=\{\{r_1\},\{r_5\}\}$ (since $r_2,r_3,r_4$ are conjugate to $r_1$) and $V/I_{\Q_5}=\{\cR \}$. Thus $n_C=4-2+1=3$ by Remark~\ref{tamecond}.
\end{example}

\begin{example}
	Let $C/\Q_{97}: y^2=(x^3-97)(x-1)(x-2)(x-3)$ with cluster picture
$$
\clusterpicture
  \Root[A] {} {first} {r1};
  \Root[A] {} {r1} {r2};
  \Root[A] {} {r2} {r3};
  \Root[A] {3} {r3} {r4};
  \Root[A] {} {r4} {r5};
  \Root[A] {} {r5} {r6};
  \ClusterLDName c4[][\frac{1}{3}][\s] = (r1)(r2)(r3);
  \ClusterLDName c5[][0][\cR] = (c4)(r4)(r5)(r6);
\endclusterpicture	\quad \text{where } \s=\{\zeta_3^j\sqrt[3]{97} \, \mid \, j=0,1,2 \}, \text{ and } \zeta_3^3=1\neq \zeta_3.$$

Again by Remark \ref{tamecond}, $n_C=n_{\mathrm{tame}}$. Now $\tilde\lambda_{\cR}=0$ and $\tilde\lambda_{\s}=\frac{1}{2}$ from which we see that $U/I_{\Q_{97}}=\{\s,\{1\},\{2\},\{3\} \}$ and $V/I_{\Q_{97}}=\{\cR\}$. Therefore $n_C=4-4+1+1=2$ since $|\cR|$ and $v(c)$ are both even.
\end{example}

\begin{remark}
	As the first and last examples show, the conductor does not determine whether a curve is semistable, in contrast to the elliptic curve setting.
\end{remark}

\newpage
\section{Root number (tame case)}\label{se:rootnumbers}

We give a description of the local root number $W(A/K)$ of an abelian variety $A$ (eg. $A = \Jac C$), first in the case of semistable reduction, then in the case of tame reduction. For $\Jac C$, we give this description in terms of the cluster picture.

\begin{notation}
	Throughout, $\chi_n$ will denote a fixed character of $I_K$ of order $n$, and for an abelian variety $A/K$ we shall decompose $\H(A/\bar{K},\Ql) = \rho_{ab} \oplus \left( \rho_t \otimes \mathrm{sp}(2) \right)$ as in Section \ref{se:GalRep}. 
	For a cluster $\mathfrak{s}$, let $G_{\mathfrak{s}}=\Stab_{G_K}(\s)$ and $I_{\mathfrak{s}} = \Stab_{I_K}(\s)$ be its Galois and inertia stabilisers respectively, and let $n_{\mathfrak{s}} =  [I_K:I_{\mathfrak{s}}]$. We write $X$ for the set of all cotwins and all even, non-\"ubereven clusters of $C$.
\end{notation}

\begin{theorem} \label{thm:rootnumberss}
	Let $A/K$ be an abelian variety with semistable reduction. Then $$W\left(A/K\right) = \left(-1\right)^{\langle \triv, \rho_t \rangle}.$$
\end{theorem}

When $C/K$ is semistable, $W(\Jac C/K)$ may then be computed from the cluster picture as follows.
\begin{prop}
	\label{prop:onerhot}
	Let $C/K$ be a (not necessarily semistable) hyperelliptic curve. The toric part $\rho_t$ of the representation of $\Jac C$ satisfies:
	\begin{align*}
		\langle \triv, \rho_t \rangle & = \#\{\mathfrak{s} \in (X \setminus \mathcal{R}) / G_K\colon \Res_{G_\s}\epsilon_{\s} = \triv\} - \left\{ \begin{array}{ll} 1 & \mathcal{R} \textrm{ \"ubereven and $c \in K^{\times 2}$},\\ 0  & \textrm{else.} \end{array}\right. 
	\end{align*}
\end{prop}

\begin{theorem} \label{thm:rootnumber}
	Let $A/K$ be an abelian variety with tame reduction. Let 
	\begin{align*}
		m_t = \langle \rho_t \big|_{I_K}, \chi_2 \rangle, \qquad \qquad
		m_e = \langle \rho_{ab} \big|_{I_K}, \chi_e \rangle \textrm{ for } e \geq 2.
	\end{align*}
	Then 
	\begin{align*}
		W\left(A/K\right) = \smash{\Big(\prod_{e\geq 3}W_{q,e}^{m_e}\Big)\left(-1\right)^{\langle \triv, \rho_t \rangle} W_{q,2}^{m_t + \frac{1}{2} m_2},}
	\end{align*}
	where $q=|k|$ and 
	\begin{align*}
		W_{q,e} = \left\{ \begin{array}{cl}
			\left(\frac{q}{l}\right) & \textrm{if $e=l^k$ for some odd prime $l$ and integer $k \geq 1$;}\\
			\left(\frac{-1}{q}\right) & \textrm{if $e=2l^k$ for some prime $l\equiv3 \,\mathrm{ mod }\, 4$, $k\geq 1$, or if $e=2$;} \\
			\left(\frac{-2}{q}\right) & \textrm{if $e=4$;} \\
			\left(\frac{2}{q}\right) & \textrm{if $e=2^k$ for some integer $k\geq3$;} \\
			1 & \textrm{else}.
		\end{array} \right.
	\end{align*}
\end{theorem}

In the case of Jacobians of hyperelliptic curves with tame reduction, $\langle \triv, \rho_t \rangle$ can be calculated as in Proposition \ref{prop:onerhot} and $m_t$ can be read off the cluster picture.

\begin{prop} \label{thm:toricrn}
	Let $C/K$ be a hyperelliptic curve with tame reduction and $\deg(f)$ even. Then
	\begin{align*}
		m_t & \equiv  
			v(c) + \# \{\mathfrak{s} \in X / I_K \,:\, n_{\mathfrak{s}}\left( \nu_{\mathfrak{s}} - |\mathfrak{s}|d_{\mathfrak{s}}\right)\,\, \mathrm{ even }\,\} + \sum_{\mathfrak{s} \in X / I_K} n_{\mathfrak{s}} \quad\mathrm{mod}\, 2.
	\end{align*}
\end{prop}

The final quantities which are required are the $m_e$, the multiplicities of the eigenvalues of the abelian part of the representation. These are straightforward to calculate by hand in terms of the Galois representation from Section \ref{se:GalRep}. An explicit but rather messy description in terms of the cluster pictures exists; see \cite[Theorem~4.5, Section~4.1.2]{BisattCl}.

\begin{remark}
	The root number of curves $C$ with potentially totally toric reduction (possibly wild) can be calculated via \cite[Lemma 3.8, 4.1]{BisattRN}.
\end{remark}

\newpage

\begin{example}
	First we give three cases where the curve is semistable (this can be checked using Theorem \ref{semistability criterion subsection}), and the calculation simplifies by Theorem \ref{thm:rootnumberss}. \vskip 4pt
	
	\hskip -12pt \textbf{(i)} Let $C_1/\QQ_{23}$ be given by\footnote{we picked $p=23$ as it is the smallest prime $p$ such that $\left(\frac{2}{p} \right) = \left(\frac{3}{p} \right) = 1$ and $\left(\frac{-1}{p} \right) = -1$.} $$\clusterpicture            
	\Root[A] {} {first} {r1};
	\Root[A] {} {r1} {r2};
	\ClusterLDName c1[][2][\mathfrak{s}_1] = (r1)(r2);
	\Root[A] {} {c1} {r3};
	\Root[A] {} {r3} {r4};
	\ClusterLDName c2[][2][\mathfrak{s}_2] = (r3)(r4);
	\Root[A] {} {c2} {r5};
	\Root[A] {} {r5} {r6};
	\ClusterLD c2[][0] = (c1)(r5)(r6);
	\endclusterpicture \qquad y^2 = (x^2 - 23^4)((x-1)^2 - 23^4)(x-2)(x-3),$$ with $\mathcal{R} = \{23^2, -23^2, 1+23^2, 
	1-23^2, 2, 3\}$. Note that $\mathfrak{s}_1$ and $\mathfrak{s}_2$ lie in their own orbits, and $\epsilon_{\s_1} = \epsilon_{\s_2} = \triv$. Therefore by Proposition \ref{prop:onerhot}, $\langle \triv, \rho_t \rangle = 2$ and $$W(\Jac C_1/\QQ_{23}) = (-1)^{\langle \triv, \rho_t \rangle} = (-1)^2 = 1.$$
	
	\hskip-12pt \textbf{(ii)} Now let $C_2/\QQ_{23}$ be given by $$y^2 = ((x-i)^2 - 23^4)((x+i)^2 - 23^4)(x-2)(x-3),$$ with $\mathcal{R} = \{i + 23^2,i -23^2, -i + 23^2, -i - 23^2, 2, 3\}$. The cluster picture of $C_2$ is the same as $C_1$, but now Frobenius swaps $\s_1$ and $\s_2$. One checks that $\Res_{G_{\s}}\epsilon_{\mathfrak{s}_1} = \triv$. Therefore $\langle \triv, \rho_t \rangle = 1$ and hence $$W(\Jac C_2)/\QQ_{23}) = (-1)^{\langle \triv, \rho_t \rangle} = (-1)^1 = -1.$$ 
	
	\hskip-12pt \textbf{(iii)} Now let $C_3/\QQ_{23}$ be given by $$y^2 = -(x^2 - 23^4)((x-1)^2 - 23^4)(x-2)(x-3),$$ so the roots and cluster picture are the same as (i), but now the leading coefficient is $-1$. Now $\mathfrak{s}_1$ and $\mathfrak{s}_2$ are in their own orbits again but $\epsilon_{\s_1}$ and $\epsilon_{\s_2}$ have order $2$. Therefore $\langle \triv, \rho_t \rangle = 0$ and hence $$W(\Jac C_3 /\QQ_{23}) = (-1)^{\langle \triv, \rho_t \rangle} = (-1)^0 = 1.$$ 
\end{example}

\begin{example}
	Let $C/\QQ_7$ be given by $$\clusterpicture            
	\Root[A] {} {first} {r1};
	\Root[A] {} {r1} {r2};
	\ClusterLDName c1[][\frac 52][\mathfrak{s}_1] = (r1)(r2);
	\Root[A] {} {c1} {r3};
	\Root[A] {} {r3} {r4};
	\ClusterLDName c2[][\frac 52][\mathfrak{s}_2] = (r3)(r4);
	\Root[A] {} {c2} {r5};
	\Root[A] {} {r5} {r6};
	\ClusterLDName c3[][\frac 52][\mathfrak{s}_3] = (r5)(r6);
	\ClusterLD c3[][0] = (c1)(c2)(c3);
	\endclusterpicture \qquad y^2 = 7\left(x^2 - 7^5\right)\left((x-1)^2 - 7^5\right)\left((x-2)^2 - 7^5\right). $$
	Since there are no principal, non-\"ubereven clusters, the abelian part of the representation is trivial. We calculate $\s_1^* = \mathcal{R}$, $\theta_{\mathfrak{s}_1^*}^2 = 7$ and hence $\epsilon_{\mathfrak{s}_1}$ is a character of order $2$. Similarly, $\epsilon_{\mathfrak{s}_2}$ and $\epsilon_{\mathfrak{s}_3}$ are characters of order $2$. Therefore, since $c \not\in \QQ_7^{\times 2}$, $\langle \triv, \rho_t \rangle = 0$ by Proposition \ref{prop:onerhot}. Furthermore, $\mu = 1$ and for $i=1,2,3$, $n_{\mathfrak{s}_i} = 1$ and $n_{\mathfrak{s}_i}\left(\nu_{\mathfrak{s}_i} - |\mathfrak{s}_i|d_{\mathfrak{s}_i}\right) = 1$. By Proposition \ref{thm:toricrn} $m_t \equiv 4 \mod 2$, so by Theorem \ref{thm:rootnumber} $$W(\Jac C/\QQ_7) = W_{7,2}^4 = 1. $$
\end{example}

\begin{example}
	Let $C/\QQ_7$ be given by $$\clusterpicture 
	\Root[A] {} {first} {r1};
	\Root[A] {} {r1} {r2};
	\Root[A] {} {r2} {r3};
	\ClusterLDName c1[][\frac 83][] = (r1)(r2)(r3);
	\Root[A] {} {c1} {r4};
	\Root[A] {} {r4} {r5};
	\Root[A] {} {r5} {r6};
	\ClusterLDName c2[][\frac 72][] = (r4)(r5)(r6);
	\ClusterLD c3[][0] = (c1)(c2);
	\endclusterpicture \qquad y^2 = (x^3 - 7^8)(x-1)((x-1)^2-7^7).$$
	Since the only even cluster is $\cR$, the toric part of the representation is trivial (Theorem \ref{th:ptr}) and hence only the abelian part contributes to the root number. On inertia, $ \rho_{ab} = \chi_3 \oplus \chi_3^{-1} \oplus \chi_4 \oplus \chi_4^{-1}$ and so $m_3 = 1$, $m_4=1$ and $m_e = 0$ for all other $e \in \mathbb{N}$. We calculate $W_{7,3} = \left(\frac{7}{3}\right) = 1$, $W_{7,4} = \left(\frac{-2}{7}\right) = -1$ and $$W(\Jac C/\QQ_7) = W_{7,3}W_{7,4} = -1.$$
\end{example}

\newpage
\section{Differentials (semistable case)}
\label{se:diff}
Let $\Omega_{C/K}^1(C)$ be the $g$-dimensional $K$-vector space of regular differentials of $C$. It is spanned by $\omega_0,\dots,\omega_{g-1}$, where $\omega_i=x^i\frac{dx}{y}$.

Fix a regular model $\mathcal{C}/\cO_K$ of $C$ (see Section \ref{se:MinRegMod}), and consider the global sections of the relative dualising sheaf $\omega_{\mathcal{C}/\cO_K}$.

\begin{rmrk}
\label{rmrk:DualisingSheaf}
		(i) $\omega_{\mathcal{C}/\cO_K}(\mathcal{C})$ can be thought of as the space of those differentials that are regular 
			not only along $C$ (the generic fibre of $\mathcal{C}$) but also along every irreducible component 
			of the special fibre of $\mathcal{C}$;
			
		(ii) $\omega_{\mathcal{C}/\cO_K}(\mathcal{C})$ can be viewed as an $\cO_K$-lattice in $\Omega_{C/K}^1(C)$;
		
		(iii) $\omega_{\mathcal{C}/\cO_K}(\mathcal{C})$ is independent of the choice of the regular model $\mathcal{C}$.
\end{rmrk}

\begin{defn}
\label{defn:IntegralDifferentials}
A \textit{basis of integral differentials of} $C$, denoted $\omega_0^\circ,\dots,\omega_{g-1}^\circ$, is an $\cO_K$-basis of $\omega_{\mathcal{C}/\cO_K}(\mathcal{C})$ as an $\cO_K$-lattice in $\Omega_{C/K}^1(C)$.
\end{defn}

\begin{thm}
\label{thm:IntegralDifferentials}
	Suppose $C/K$ is semistable. 
	For $i=0,\dots,g-1$ inductively
	\begin{itemize}
	\item compute 
	$e_{\t,i}=\frac{\nu_{\t}}{2}-d_\t-\sum_{j=0}^{i-1}d_{\s_j\wedge\t}$ 
	for every proper cluster $\t$;
	\item choose a proper cluster $\s_i$ so that $e_{\s_i,i}=\max_{\t}\{e_{\t,i}\}$.\footnote{Suppose the maximal value is obtained by two different clusters $\s$ and $\s'$. If
	$\s'\subseteq\s$, choose $\s_i=\s$, if $\s\subseteq\s'$, choose $\s_i=\s'$, otherwise choose freely any of the two.} Denote $e_{\s_i,i}$ by $e_i$.
	\end{itemize}
	Fix a centre $z_\s\in K^{nr}$ for every proper cluster $\s$,\footnote{This is always possible by Theorem \ref{semistability criterion subsection} and \cite[Lemma B.1]{m2d2} since $C$ is semistable.} then choose a finite unramified extension $F/K$ such that $z_\s \in F$ for all $\s$.
	Let $\beta\in \cO_F^\times$ be any element such that $\mathrm{Tr}_{F/K}(\beta)\in \cO_K^\times$.
	Then the differentials
	\[\omega^\circ_i=\pi^{e_i}\cdot\mathrm{Tr}_{F/K}\bigg(\beta\prod_{j=0}^{i-1}(x\!-\!z_{\s_j})\bigg)\frac{dx}{y},\qquad i=0,\dots,g-1,\]
	form a basis of integral differentials of $C$.
\end{thm}

\begin{rmrk}
\label{rmrk:DifferentialsAlgClosed}
If $F=K$ then in Theorem \ref{thm:IntegralDifferentials} we can choose $\beta=1$ and the trace is just the identity. One can take $F=K$ if and only if $\Frob$ does not permute clusters.
\end{rmrk}

Consider $\omega=\omega_0\wedge\dots\wedge\omega_{g-1},\omega^\circ=\omega_0^\circ\wedge\dots\wedge\omega_{g-1}^\circ\in\det\Omega_{C/K}^1(C)=\bigwedge^g\Omega_{C/K}^1(C)$. As $\det\Omega_{C/K}^1(C)$ is a $1$-dimensional $K$-vector space, there exists $\lambda\in K$ such that $\omega^\circ=\lambda\cdot\omega$. We will denote this element by $\frac{\omega^\circ}{\omega}$. 

\begin{rmrk}
\label{rmrk:FudgeFactor}
Note that $\frac{\omega^\circ}{\omega}$ is only well-defined up to a unit. Moreover, it depends on the choice of Weierstrass equation for $C$.
\end{rmrk}

\begin{thm}
\label{thm:FudgeFactor}
Suppose $C/K$ is semistable. With the notation above,
\[
8\cdot v\big(\tfrac{\omega^\circ}{\omega}\big)= 4g\cdot v(c)+\sum_{\s\text{ even}}\delta_\s(|\s|\!-\!2)|\s|+\sum_{\s\text{ odd}}\delta_\s(|\s|\!-\!1)^2,\qquad\text{where $\delta_\cR=d_\cR$.}
\]
\end{thm}

\begin{prop}
\label{prop:HyperellipticDiscriminant}
Let $\Delta_C$ be the discriminant of $C$ (see Section \ref{sec:minimaldiscriminant}). Then \[g\cdot v(\Delta_C)-(8g+4)\cdot v\big(\tfrac{\omega^\circ}{\omega}\big)\] is independent of the choice of equation for $C$. If $C/K$ is semistable, it is given by
\[\sum_{\substack{\s\text{ even}\\1<|\s|<2g+1}}\frac{\delta_\s}{2}|\s|(2g\!+\!2\!-\!|\s|)+\sum_{\substack{\s\text{ odd}\\1<|\s|<2g+1}}\frac{\delta_\s}{2}(|\s|\!-\!1)(2g\!+\!1\!-\!|\s|).\]

\end{prop}

\begin{ex}
\label{eg:diff1}
Consider the semistable genus $3$ curve  \[\vspace{-4pt}C:y^2=((x-7^2)^2+1)((x-2\cdot 7^2)^2+1)((x-3\cdot 7^2)^2+1)(x^2-1)\text{ over }\Q_7.\] Its cluster picture is
\clusterpicture            
  \Root[A] {2} {first} {r1};
  \Root[A] {} {r1} {r2};
  \Root[A] {} {r2} {r3};
  \ClusterLDName c1[][2][\mathfrak{t}_1] = (r1)(r2)(r3);
  \Root[A] {1} {c1} {r4};
  \Root[A] {} {r4} {r5};
  \Root[A] {} {r5} {r6};
  \ClusterLDName c2[][2][\mathfrak{t}_2] = (r4)(r5)(r6);
  \Root[A] {1} {c2} {r7};
  \Root[A] {} {r7} {r8};
  \ClusterLD c3[][0] = (c1)(c2)(r7)(r8);
\endclusterpicture
with $\t_1=\{i\!+\!7^2,~i\!+\!2\!\cdot\! 7^2,~i\!+\!3\!\cdot\!7^2\}$, $\t_2=\{-i\!+\!7^2,~-i\!+\!2\!\cdot\!7^2,~-i\!+\!3\!\cdot\! 7^2\}$ and $\cR=\t_1\cup\t_2\cup\{\pm 1\}$, where $i^2=-1$. We want to find a basis of integral differentials of $C$ using Theorem \ref{thm:IntegralDifferentials}. First compute $e_{\t,0}$ for $\t_1,\t_2,\cR$ and note that $e_{\t_1,0}=e_{\t_2,0}=\max_\t \{e_{\t,0}\}$ (see table below). Since neither $\t_1\subset\t_2$ nor $\t_2\subset\t_1$, we are free to choose any of the two as $\s_0$. Set $\s_0=\t_2$. We repeat this procedure for $e_{\t,1}$ and $e_{\t,2}$ as shown in the following table.

\begin{center}
    \small\begin{tabular}{c|c|c|c!{\vline width 1pt}c|c|c}
     & $z_\t$ &  $d_\t$& $\nu_\t/2$ & $e_{\t,0}$ ($=\nu_\t/2-d_\t$)&$e_{\t,1}$ ($=e_{\t,0}-d_{\s_0\wedge\t}$)&$e_{\t,2}$ ($=e_{\t,1}-d_{\s_1\wedge\t}$)\\
     \Xhline{1pt}
     $\t_1$& $i$ & $2$ & $3$ & $1$ & $\color{red}1$ & $-1$ \\
     \hline
     $\t_2$& $-i$ & $2$ & $3$ & $\color{red}1$ & $-1$ & $-1$ \\
     \hline
    $\cR$& $0$ & $0$ & $0$ & $0$ & $0$ & $\color{red}0$ \\ 
\end{tabular}
\end{center}
The numbers coloured in red are the $e_i$'s. Choosing $F=\Q_7(i),\beta=1$, we have 
\begin{center}
$\begin{array}{lr}
\begin{array}{l}
\omega_0^\circ=7^1\mathrm{Tr}_{\Q_7(i)/\Q_7}(1)\tfrac{dx}{y}=14\tfrac{dx}{y},\\
\omega_1^\circ=7^1\mathrm{Tr}_{\Q_7(i)/\Q_7}(x+i)\tfrac{dx}{y}=14x\tfrac{dx}{y},\\
\omega_2^\circ=7^0\mathrm{Tr}_{\Q_7(i)/\Q_7}((x+i)(x-i))\tfrac{dx}{y}=2(x^2+1)\tfrac{dx}{y},
\end{array}& 
	\text{\scriptsize{$\left(\begin{matrix}
    \vspace{1pt}\omega_0^\circ\\
    \vspace{1pt}\omega_1^\circ\\
    \omega_2^\circ
    \end{matrix}\right)=
    \left(\begin{matrix}
    14 \!\!\!\!& 0 \!\!\!\!& 0\cr
    0 \!\!\!\!&  14 \!\!\!\!& 0\cr
    2 \!\!\!\!&  0 \!\!\!\!& 2
    \end{matrix}\right)
    \left(\begin{matrix}
    \omega_0\cr
    \omega_1\cr
    \omega_2
    \end{matrix}\right),$}}
    \end{array}$
\end{center}
form a basis of integral differentials. In particular, $\omega^\circ=8\cdot 7^2\omega$ and so $v\big(\frac{\omega^\circ}{\omega}\big)=2$. Finally, we check this result agrees with what the formula in Theorem \ref{thm:FudgeFactor} predicts:
\begin{align*}
v\big(\tfrac{\omega^\circ}{\omega}\big)&=\tfrac{1}{8}\big(4g\cdot v(c)+d_\cR(|\cR|-2)|\cR|+\delta_{\t_1}(|\t_1|-1)^2+\delta_{\t_2}(|\t_2|-1)^2\big)\\
&=\tfrac{1}{8}\big(12\cdot 0+0(8-2)8+2(3-1)^2+2(3-1)^2\big)=2.
\end{align*}
\end{ex}

\begin{ex}
\label{eg:diff2}
Let $f(x)=7^2(x^6-1)\in\Q_7[x]$ and $C^n:y^2=7^{6n}f(x/7^n)$, $n\in\Z$, a family of isomorphic semistable hyperelliptic curves of genus $2$. The cluster picture of $C^n$ is 
\clusterpicture            
  \Root[A] {2} {first} {r1};
  \Root[A] {} {r1} {r2};
  \Root[A] {} {r2} {r3};
  \Root[A] {} {r3} {r4};
  \Root[A] {} {r4} {r5};
  \Root[A] {} {r5} {r6};
  \ClusterLD c3[][n] = (r1)(r2)(r3)(r4)(r5)(r6);
\endclusterpicture
with $\cR=\{7^n,\zeta_3 7^n, \zeta_3^2 7^n, -7^n,-\zeta_3 7^n, -\zeta_3^2 7^n\}$. Since we have only one cluster, $\s_0=\s_1=\cR$. Then $e_0=1+2n$ and $e_1=1+n$. As $0$ is a centre of $\cR$, we are in the situation of Remark \ref{rmrk:DifferentialsAlgClosed}, and so $\omega_0^\circ=7^{1+2n}\frac{dx}{y}$, $\omega_1^\circ=7^{1+n}x\frac{dx}{y}$. This shows that $v(\omega^\circ/\omega)=2+3n$ does depend on $n$, i.e.\,on the choice of equation.

On the other hand, from the formula in Proposition \ref{prop:HyperellipticDiscriminant} we immediately see that $g\cdot v(\Delta_C)-(8g+4)\cdot v(\omega^\circ/\omega)=0$, which is independent of $n$.
\end{ex}

\begin{ex}
\label{eg:diff3}
Let $C\!:\!y^2\!=\!f(x)$ be a semistable curve with $f(x)$ monic. Suppose
\vspace{-12pt}\[\begin{array}{lr}\clusterpicture            
  \Root[A] {2} {first} {r1};
  \Root[A] {} {r1} {r2};
  \ClusterLDName c1[][][\mathfrak{t}_1] = (r1)(r2);
  \Root[A] {1} {c1} {r3};
  \Root[Dot] {} {r3} {r4};
  \Root[Dot] {} {r4} {r5};
  \Root[Dot] {} {r5} {r6};
  \Root[A] {} {r6} {r7};
  \ClusterLDName c2[][][\mathfrak{t}_2] = (c1)(r3)(r4)(r5)(r6)(r7);
  \Root[A] {1} {c2} {r8};
  \Root[Dot] {} {r8} {r9};
  \Root[Dot] {} {r9} {r10};
  \Root[Dot] {} {r10} {r11};
  \Root[A] {} {r11} {r12};
  \ClusterLDName c3[][][\cR] = (c2)(r8)(r9)(r10)(r11)(r12);
\endclusterpicture \quad\begin{array}{l}\text{is its cluster picture, with }d_{\t_1}=u/2, d_{\t_2}=a, d_{\cR}=b,\\\text{for some }u,a,b\in\Z, u/2>a>b.\end{array}
\end{array}\]
As in Example \ref{eg:diff1}, to compute $e_i$ for $i=0,\dots g-1$, we draw the following table

\begin{center}
    \small\begin{tabular}{c|c!{\vline width 1pt}c|ccc|cc}
     &  $d_\t$& $e_{\t,0}$&$e_{\t,1}$ & $\dots$ & $e_{\t,m-1}$ & $e_{\t,m}$ & $\dots$ \\
     \Xhline{1pt}
     $\t_1$& $u/2$ & $\frac{|\t_2|-|\t_1|}{2}a+\frac{|\cR|-|\t_2|}{2}b$ & $e_{\t_1,0}-a$ & $\dots$ & $e_{\t_1,0}-(m-1)a$ & $e_{\t_1,0}-ma$ & $\dots$\\
     \hline
     $\t_2$&  $a$ & $\color{red}\frac{|\t_2|-|\t_1|}{2}a+\frac{|\cR|-|\t_2|}{2}b$ & $\color{red}e_{\t_2,0}-a$ & $\color{red}\dots$ & $\color{red}e_{\t_2,0}-(m-1)a$ & $e_{\t_2,0}-ma$ & $\dots$\\
     \hline
    $\cR$& $b$ & $\frac{|\cR|-|\t_1|}{2}b$ & $e_{\cR,0}-b$ & $\dots$ & $e_{\cR,0}-(m-1)b$ & $\color{red}e_{\cR,0}-mb$ & $\color{red}\dots$ \\ 
\end{tabular}
\end{center}
where $m$ is the least positive integer such that $e_{\cR,0}-mb\geq e_{\t_2,0}-ma$. Then $m=\big\lfloor\frac{|\t_2|-1}{2}\big\rfloor$ and $\s_0=\dots=\s_{m-1}=\t_2, \s_m=\dots=\s_{g-1}=\cR$. Note that he twin $\t_1$ is never selected, and $\omega_0^\circ,\dots,\omega_{m-1}^\circ$ form a basis of integral differentials of $C_{\t_2}\colon y^2=\prod_{r\in\t_2}(x-r)$. These are general phenomena (see \cite[Lemma 4.2]{Kunzweiler}).
\end{ex}

\newpage
\section{Minimal discriminant (semistable case)} \label{sec:minimaldiscriminant}
	
	The {\em discriminant} $\Delta_C$ of $C$ is given by
	\[
	\Delta_C = 16^g \, c^{4g+2}\, \textrm{disc}\Big(\frac{1}{c} f(x)\Big).
	\]
	
	The following theorem provides a formula to compute the valuation of the discriminant in terms of cluster pictures. 
	
	\begin{theorem}
		\label{thm:discriminant}
		The valuation of the discriminant of $C$ is given by
		\begin{align*}
		v(\Delta_C)=&\, v(c)(4g+2) + \sum_{\s \textrm{ proper}} \delta_{\s} |\s|(|\s|-1),
		\end{align*}
		where $\delta_{\s} = d_{\cR}$ when $\s = \cR$.
	\end{theorem}
	
	Let $v(\Delta_C^{\min})$ denote the valuation of the minimal discriminant\footnote{The valuation of the minimal discriminant is the minimum of $v(\Delta)$ amongst all integral Weierstrass equations for $C$.} of the curve~$C$. If $C$ has semistable reduction, one may read off this quantity from the cluster picture  or from the centred BY tree associated to the equation.
	
	\begin{theorem}
		\label{thm:minimal_discriminant}
		If $C/K$ is semistable and $|k| > 2g+1$, then
		\[
		\frac{v(\Delta_C) - v(\Delta_C^{\min})}{4g+2} 
		= v(c) - E + \sum_{g+1<|\s|} \delta_{\s} (|\s|-g-1),
		\]
		where $\delta_{\s} = d_{\cR}$ when $\s = \cR$, and $E = 0$ unless there are two clusters of size $g+1$ that are permuted by Frobenius and $v(c)$ is odd, in which case $E = 1$.
	\end{theorem}
	
	
	\begin{definition}
		\label{def:centredBYtree_sec16}
		For a connected subgraph $T$ of a BY tree, we define a genus function by $g(T) = \#(\textrm{connected components of the blue part}) - 1 + \sum_{v \in V(T)} g(v)$. 
		
		If there is an edge $e \in E(T_C)$ such that both trees in $T_C\setminus\{e\}$ have equal genus (i.e. genus $\lfloor \frac{g}{2}\rfloor$), then we insert a vertex $z_T$ on the midpoint of $e$ and call it the {\em centre} of $T_C$. Otherwise, there exists a unique vertex $v \in V(T_C)$ such that all trees in $T_C\setminus \{v\}$ have genus smaller than $g/2$. In this case $z_T = v$ is the centre of $T_C$.
		In both cases, the {\em centred BY tree} $T^*_C$ is the tree with vertex set $V(T_C)\cup \{z_T\}$;   we denote by $\preceq$ the partial ordering on $V(T^*_C)$ with maximal element $z_T$.  
	\end{definition}
	
	\begin{notation}	
		Define a weight function on $V(T^*_C)$  by \[
		S(v) = \sum_{v'\preceq v \textrm{ blue}}  (2g(v') + 2 - \#\textrm{blue edges at } v'). 
		\]
		For each $v \neq z_T$, write $e_v$ for the edge connecting $v$ with its parent, i.e. the vertex connected to $v$ lying on the path to the centre of $T^*_C$. Let $\delta_v = \textrm{length}(e_v)$ if $e_v$ is blue, and $\delta_{v} = 1/2 \cdot \textrm{length}(e_v)$ if $e_v$ is  yellow.
	\end{notation}
	
	\begin{theorem}
		\label{thm:discriminant_by_sec16}
		Suppose that $C$ is semistable and $|k|>2g+1$. Let $T_C^*$ be the centred BY tree associated to $C$. Then the valuation of the minimal discriminant of $C$ is given by
		\[
		v(\Delta_C^{\min}) 
		=  E\cdot (4g+2) + \sum_{v \neq z_T} \delta_v S(v)(S(v)-1),
		\]
		where $E = 0$ unless $z_T$ has exactly two children $v_1,v_2$ with $S(v_1)=S(v_2)=g+1$ that are permuted by Frobenius and $(g+1)\delta_{v_1}$, $(g+1)\delta_{v_2}$ are odd, in which case $E=1$.
	\end{theorem}
	
	\newpage

	\begin{example}
		Consider \(C: y^2 = p(x^2-p^5)(x^3-p^3)((x-1)^3-p^9)\) over $\QQ_p$ for $p > 7$. This is a genus $3$ hyperelliptic curve with cluster picture
		$$
		\clusterpicture           
		\Root[A] {} {first} {r1};
		\Root[A] {} {r1} {r2};
		\Root[A] {4} {r2} {r3};
		\Root[A] {} {r3} {r4};
		\Root[A] {} {r4} {r5};
		\Root[A] {5} {r5} {r6};
		\Root[A] {} {r6} {r7};
		\Root[A] {} {r7} {r8};
		\ClusterLDName c1[][{\frac 32}][\s_1] = (r1)(r2);
		\ClusterLDName c2[][{1}][\s_2] = (c1)(r3)(r4)(r5);
		\ClusterLDName c3[][{3}][\s_3] = (r6)(r7)(r8);
		\ClusterLD c4[][{0}] = (c2)(c3);
		\endclusterpicture.
		$$
		Using the formula from Theorem \ref{thm:discriminant}, we get that the valuation of the discriminant of the equation is
		\[
		v(\Delta_C) = 1 \cdot (4\cdot 3 + 2) +  3/2 \cdot 2 \cdot 1 + 1 \cdot 5 \cdot 4 + 3 \cdot 3 \cdot 2 = 55.
		\]
		
		Since $C$ has semistable reduction and $|\F_p| > 7$, we may now apply Theorem \ref{thm:minimal_discriminant} in order to find the valuation of the minimal discriminant.
		The right hand side of the equation in that theorem is
		\(
		v(c) - E + \sum_{g+1<|\s|} \delta_{\s} (|\s|-g-1) = 2
		\),
		hence
		\(
		v(\Delta_C^{\min}) = v(\Delta_C) - 2 \cdot (4g+2) = 27.
		\)
		
		Alternatively, we could have used the associated BY tree $T_C$:\vspace{-.2cm}
		$$	{\begin{tikzpicture}[scale=\GraphScale]
			\BlueVertices
			\VertexLN[x=3.00,y=0.000,L=1]{2}{$v_{\s_3}$};
			\VertexLN[x=0.000,y=0.000,L=\relax]{3}{$v_{\s_1}$};
			\VertexLN[x=1.50,y=0.000,L=1]{1}{$v_{\s_2}$};
			\BlueEdges
			\Edge(1)(2)
			\YellowEdges
			\Edge(1)(3)
			\TreeEdgeSignS(1)(2){0.5}{4}
			\TreeEdgeSignS(1)(3){0.5}{3}
			\end{tikzpicture}}$$
			\vskip -.2cm
		In this example $V(T^*_C) = V(T_C)$ and $v_{\s_2}$ is the centre of $T^*_C$. Then $S(v_{\s_1}) \!=\! 2$, $S(v_{\s_3})\!=\! 3$, $\delta_{v_{\s_1}}\!=\! 3/2$ and $\delta_{v_{\s_3}}\! =\! 4$. It follows from Theorem \ref{thm:discriminant_by_sec16} that 
		$v(\Delta^{\min}_C) = 3/2 \cdot 2\cdot 1 + 4 \cdot 3 \cdot 2 = 27$.
	\end{example}

	\begin{example}
		\label{eg:mindisc2}
		Consider the curve \(C: y^2 = 7(x^2+1)(x^2+36)(x^2+64)\) defined over $\QQ_7$. This is a genus $2$ hyperelliptic curve with cluster picture 
		\(
		\clusterpicture           
		\Root[A] {} {first} {r1};
		\Root[A] {} {r1} {r2};
		\Root[A] {} {r2} {r3};
		\Root[A] {4} {r3} {r4};
		\Root[A] {} {r4} {r5};
		\Root[A] {} {r5} {r6};
		\ClusterLDName c1[][{1}][\s_1] = (r1)(r2)(r3);
		\ClusterLDName c2[][{1}][\s_2] = (r4)(r5)(r6);
		\ClusterLD c4[][{0}] = (c1)(c2);
		\endclusterpicture \,.
		\)
		Using one of the formulas from Theorem \ref{thm:discriminant}, we get 
		\(
		v(\Delta_C) = 22.
		\)
		
		Since $C$ has semistable reduction, we can apply Theorem \ref{thm:minimal_discriminant}. Note that the two clusters $\s_1 = \{i,i\pm 7i\}, \, \s_2 = \{-i,-i \pm 7i\}$ are permuted by Frobenius. Therefore $E=1$ here and the right hand side of the formula vanishes. In particular, we find that $v(\Delta_C^{\min}) = v(\Delta_C) = 22$. The minimality of the equation is also implied by Theorem  \ref{thm:minwmodelsscase}, since Condition (1) of that theorem is satisfied.
		
		The minimal discriminant is not invariant under unramified extensions. Let $C_{K}$ denote the base change of $C$ to $K = \QQ_7(i)$. Since the extension is unramified, the cluster picture does not change. However the two clusters $\s_1$ and $\s_2$ are no longer swapped by Frobenius, hence $E=0$ and by Theorem \ref{thm:discriminant}, 
		$v(\Delta_{C_K}^{\min}) = v(\Delta_{C_K}) -(4g+2) = 12$. 
		A minimal Weierstrass equation over $K$ can be attained by the change of variables $x = i({x'+6})/({x'-1})$ and $y = 49{y'}/{(x'-1)^3}$:
		\[
		y'^2 = - x'(x'-2)(2x'+5)(5x'-12)(9x'-2).
		\]
		The cluster picture corresponding to this equation is 
		\(
		\clusterpicture           
		\Root[A] {} {first} {r1};
		\Root[A] {} {r1} {r2};
		\Root[A] {3} {r2} {r3};
		\Root[A] {} {r3} {r4};
		\Root[A] {} {r4} {r5};
		\ClusterLD c1[][{2}] = (r3)(r4)(r5);
		\ClusterLD c2[][{0}] = (r1)(r2)(c1);
		\endclusterpicture \,.
		\)
		
		In both of the above cases the associated BY trees consist of two blue vertices joined by a blue edge of length $2$: 
		\smash{\raise -7pt\hbox{\begin{tikzpicture}[scale=\GraphScale]
				\BlueVertices
				\VertexLE[x=1.500,y=0.000,L=1]{2}{$v_{\s_2}$};
				\VertexLW[x=0.000,y=0.000,L=1]{1}{$v_{\s_1}$};
				\BlueEdges
				\Edge(1)(2)
				\TreeEdgeSignS(1)(2){0.5}{2}
				\end{tikzpicture}}}. The centred BY trees are obtained by adding an additional vertex in the midpoint of the edge joining $v_{\s_1}$ and $v_{\s_2}$: 
		\smash{\raise -7pt\hbox{\begin{tikzpicture}[scale=\GraphScale]
				\BlueVertices
				\VertexLE[x=2.00,y=0.000,L=1]{2}{$v_{\s_2}$};
				\VertexLW[x=0.000,y=0.000,L=1]{1}{$v_{\s_1}$};
				\VertexLN[x=1.0,y=0.000,L=\relax]{3}{};
				\BlueEdges
				\Edge(3)(2)
				\Edge(3)(1)
				\TreeEdgeSignS(3)(1){0.5}{1}
				\TreeEdgeSignS(3)(2){0.5}{1}
				\TreeSignAt(3)(0,0.3){z_T}
				\end{tikzpicture}}}. 
		From the formula in Theorem \ref{thm:discriminant_by_sec16}, we see that the valuation of the minimal discriminant is given by $12 + 10 \cdot E$. The only difference between the (centred) BY trees corresponding to $C$ and $C_K$ is the action of Frobenius, and we have $E=1$ for $C$ and $E=0$ for $C_K$. As before, we find $v(\Delta^{\min}_C) = 22$ and $v(\Delta^{\min}_{C_K}) = 12$. 
	\end{example}

	\newpage
\section{Minimal Weierstrass equation}
\label{se:integweier}

Here we explain how one can tell if a Weierstrass equation is minimal. Recall that a Weierstrass equation of a curve $C/K : y^2=f(x)$ is \emph{integral} if $f(x)\in\cO_K[x]$. It is \emph{minimal} if the valuation of its  discriminant is minimal amongst all integral Weierstrass equations.  

We first characterise when the equation is integral in terms of the cluster picture. Note that the cluster picture of hyperelliptic curve is unchanged by a substitution $x\mapsto x-t$. As a result, for a hyperelliptic curve $C/K:y^2=f(x)$ it is not possible to check whether $f(x)\in\cO_K[x]$ from the cluster picture of $C$, but up to these shifts in the $x$-coordinate this is possible. 
\bigskip
\begin{thm}\label{thm:integralcriterion}
	Let $C/K:y^2=f(x)$ be a hyperelliptic curve and suppose that $G_K$ acts tamely on $\cR$. Then $f(x-z)\in\cO_K[x]$ for some $z\in K$ if and only if 
	either
	\begin{itemize}
		\item $v(c)\geq 0$ and $d_{\cR}\geq0,$ or 
		\item there is a $G_K$-stable proper cluster $\c$ with $d_{\c}\leq 0$ and 
		$$v(c)+(|\c|-|\c'|)d_{\c}+\sum_{r\notin\c}d_{\{r\}\wedge\c}\geq0,$$ 
		for some $\c'$ that is either empty or a $G_K$-stable child $\c'<\c$ with either $|\c'|=1$ or $d_{\c'}\geq 0$.
	\end{itemize} 
\end{thm}
\bigskip

We are further able to give a criterion for checking whether a given Weierstrass equation is in fact minimal.
\bigskip
\begin{thm}\label{thm:genralminweqcriterion}
	Let $C:y^2=f(x)$ be a hyperelliptic curve over $K$ with $f(x)\in\cO_K[x]$. If $d_{\cR}=v(c)=0$ and the cluster picture of $C$ has no cluster $\c\neq\cR$ with $|\c|>g+1$, then $C$ is a minimal Weierstrass equation.
\end{thm}
\bigskip
For semistable hyperelliptic curves, we can give a full characterisation of minimal Weierstrass equations in terms of cluster pictures:
\bigskip
\begin{thm}\label{thm:minwmodelsscase}
	Suppose $C:y^2= f(x)$ is a semistable hyperelliptic curve over $K$ with $f(x)\in\cO_K[x]$, and that $|k|>2g+1$. Then $C$ defines a minimal
	Weierstrass equation if and only if one of the following conditions hold:
	\begin{enumerate}
		\item there are two clusters of size $g+1$ that are swapped by Frobenius, $d_{\cR}=0$ and $v(c)\in\{0,1\}$,
		\item there is no cluster of size $>g+1$ with depth $>0$, but there is some $G_K$-stable cluster $\c$ with $|\c|\geq g+1$, $d_{\c}\geq 0$ and $v(c)=-\sum_{r\notin\c}d_{\{r\}\wedge\c}$.  
	\end{enumerate}
\end{thm}

\bigskip

Using examples we now illustrate how one can easily use cluster pictures and the results of this section to check whether a Weierstrass equation is integral and/or minimal.
\newpage
\begin{example}
	Consider $C:y^2=f(x)=p(x-\frac{1}{p^2})((x-\frac{1}{p^2})^3-p^9)(x-\frac{1}{p^2}-\frac{1}{p})$, a genus $2$ hyperelliptic curve over $\QQ_p$, for some prime $p>3$. Let us use the cluster picture of $C$ to test whether there exists some $z\in K$ such that $f(x-z)\in\cO_K[x]$. The cluster picture of $C$ is as follows: 
		$$\clusterpicture            
		\Root[A] {} {first} {r1};
		\Root[A] {2} {r1} {r2};
		\Root[A] {} {r2} {r3};
		\Root[A] {} {r3} {r4};
		\Root[A] {} {r4}{r5};
		\ClusterLDName c1[][4][\c] = (r2)(r3)(r4)(r5);
		\ClusterLDName c2[][-1][\cR] = (r1)(c1);
		\endclusterpicture 
		\textrm{ with }
		d_{\cR}=-1,\textrm{ and } d_{\c}=3.$$ 
		Note that $\cR$ and $\c$ are both proper and $G_{\QQ_p}$-stable, $\c<\cR$, $d_{\cR}\leq0$, and $d_{\c}\geq0$. A simple calculation gives that
		$$v(c)+(|\cR|-|\c|)d_{\cR}+\sum_{r\notin\cR}d_{\{r\}\wedge\cR}=0.$$ 
			Therefore, by Theorem \ref{thm:integralcriterion}, we conclude that there exists some $z\in K$ such that $f(x-z)\in\cO_K[x]$. Indeed, we can take $z=-\frac{1}{p^2}$.
\end{example}

\begin{example}
	Consider $C:y^2=(x^2-1)(x^3-p)((x-2)^3-p^7)$, a genus $3$ hyperelliptic curve over $\QQ_p$, for some prime $p>3$. The cluster picture of $C$ is as follows:
	$$\clusterpicture            
	\Root[A] {} {first} {r1};
	\Root[A] {} {r1} {r2};
	\Root[A] {2} {r2} {r3};
	\Root[A] {} {r3} {r4};
	\Root[A] {} {r4} {r5};
	\Root[A] {4} {r5}{r6};
	\Root[A] {} {r6}{r7};
	\Root[A] {} {r7}{r8};
	\ClusterLDName c1[][\frac{7}{3}][] = (r6)(r7)(r8);
	\ClusterLDName c2[][\frac{1}{3}][] = (r3)(r4)(r5);
	\ClusterLDName c3[][0][\cR] = (r1)(r2)(c1)(c2);
	\endclusterpicture. $$
	Note that $d_{\cR}=v(c)=0$ and every cluster $\c\neq\cR$ has size $<4$, so by Theorem \ref{thm:genralminweqcriterion} we can conclude that $C$ is a minimal Weierstrass equation.
\end{example}

\begin{example}
		Consider $C:y^2=p^2(x-\frac{1}{p^2})(x^5-1)$, a genus $2$ hyperelliptic curve over $\QQ_p$ for some prime $p>5$. The cluster picture of $C$ is as follows:
	$$\clusterpicture            
	\Root[A] {} {first} {r1};
	\Root[A] {2} {r1} {r2};
	\Root[A] {} {r2} {r3};
	\Root[A] {} {r3} {r4};
	\Root[A] {} {r4} {r5};
	\Root[A] {} {r5}{r6};
	\ClusterLDName c1[][2][\c] = (r2)(r3)(r4)(r5)(r6);
	\ClusterLDName c2[][-2][\cR] = (r1)(c1);
	\endclusterpicture \quad
	\textrm{with }
	d_{\cR}=-2,\textrm{ and }d_{\c}=0.$$
	Note that $d_{\cR},v(c)\neq 0$ and cluster $|\c|=5>3$, so we are unable to conclude by Theorem \ref{thm:genralminweqcriterion} whether or not $C$ is a minimal Weierstrass equation. However, one can easily check that
	the semistability criterion in Section \ref{se:redtype} is satisfied (see the examples in that section for further details of how to check this), so $C$ is semistable. Now, there is no cluster of size $>3$ with depth $>0$, but $\c$ is $G_{\QQ_p}$-stable with $|\c|=5\geq 3$, $d_{\c}=0$, and $2=v(c)=-\sum_{\{r\}\notin\c}d_{\{r\}\wedge\c}$. So, by Theorem \ref{thm:minwmodelsscase} we can conclude that $C$ defines a minimal Weierstrass equation.
\end{example}

\begin{example}
	Consider the hyperelliptic curve $C:y^2=(x^3-p^{15})(x^2-p^6)(x^3-p^3)$ over $\QQ_p$ for some prime $p>7$. We claim that the substitutions $x = p^3x'$ and $y = p^9y'$, result in a minimal Weierstrass equation
	$$ C':y'^2 = (x'^3-p^6)(x'^2-1)(p^6x'^3-1),$$ 
	whose cluster picture is as follows:
	$$\clusterpicture           
	\Root[A] {} {first} {r1};
	\Root[A] {} {r1} {r2};
	\Root[A] {} {r2} {r3};
	\Root[A] {2} {r3} {r4};
	\Root[A] {} {r4} {r5};
	\Root[A] {2} {r5} {r6};
	\Root[A] {} {r6} {r7};
	\Root[A] {} {r7} {r8};
	\ClusterLDName c1[][2][\c_1] = (r6)(r7)(r8);
	\ClusterLDName c2[][2][\c_2] = (c1)(r4)(r5);
	\ClusterLDName c4[][-2][\cR] = (c2)(r1)(r2)(r3);
	\endclusterpicture \quad\textrm{with }d_{\cR}=-2, d_{\c_2}=0, \textrm{ and } d_{\c_1}=2.$$
	We are able verify that $C'$ is indeed minimal. Note that its cluster picture has no cluster of size $>g+1$ with depth $>0$, but $\c_2$ is fixed by $G_{\QQ_p}$, $|\c_2|=5\geq 4$, $d_{\c_2}=0$, and $v(c)=-\sum_{r\notin \c_2}d_{\{r\}\wedge\c_2}=6$. So, since $C'$ is semistable, by Theorem \ref{thm:minwmodelsscase} (2) we have that $C'$ is minimal.
\end{example}

\newpage
\section{Isomorphisms of curves and canonical cluster pictures}
\label{se:iso}

\begin{defn}\label{def:abstractclusterpic}
	Let $X$ be a finite set, $\Sigma$ a collection of non-empty subsets of $X$ (called \emph{clusters}), and some $d_{\c}\in\QQ$ for every $\c\in\Sigma$ of size $>1$, called the \emph{depth} of $\c$. Then $\Sigma$ (or $(\Sigma,X,d)$) is a \emph{cluster picture} if: $X\in\Sigma$ and $\{x\}\in\Sigma$ for every $x\in X$; two clusters are either disjoint or one is contained in the other; for $\c,\c'\in\Sigma$, if $\c'\subsetneq\c$ then $d_{\c'}>d_{\c}$.
	For a hyperelliptic curve $C/K:y^2=f(x)$ denote the \textit{cluster picture} by $\Sigma_{C}=(\Sigma_C,\cR,d)$, the collection of all clusters of $\cR$ with depths. 
	
	Cluster pictures $(\Sigma^i,X^i,d^i)$, $i=1,2$, are \emph{isomorphic} ($\Sigma^1\cong\Sigma^2$) if there is a bijection $\phi:X^1\to X^2$ which induces a bijection from $\Sigma^1$ to $\Sigma^2$ and $d^1_{\c}=d^2_{\phi(\c)}.$ 
\end{defn}

\begin{defn}\label{def:equivclusterpics}
	 We say $\Sigma=(\Sigma,X,d)$ and $\Sigma'=(\Sigma',X',d')$ are \emph{equivalent} if $\Sigma'$ is isomorphic to a cluster picture obtained from $\Sigma$ in a finite number of the steps:
	\begin{enumerate}
		\item \emph{increase the depth of all clusters by $m\in\QQ$}: 
		$d_{\c}'=d_{\c}+m$ for all $\c\in\Sigma,$
		
		\item \emph{add a root $r$} if $X$ is odd:
		$X'=X\cup\{r\}, \Sigma'=(\Sigma\cup\{\{r\},X'\})\setminus\{X\},d'_{\c}=d_{\c}$ for all proper $\c\in\Sigma'\setminus\{X'\}$ and $d'_{X'}=d_{X}$,
		
		\item \emph{remove a root $r\in X$} if $X$ is even, $\{r\}<X$ and $X\setminus\{r\}\notin\Sigma$: 
		$X'\!=\!X\setminus \{r\}$, $\Sigma'\!=\!(\Sigma\cup\{X'\})\setminus\{X,\{r\}\},d'_{\c}\!=\!d_{\c}$ for $\c\in\Sigma'\setminus\{X'\}$ proper and $d'_{X'}\!=\!d_{X}$,
		
		\item \emph{redistribute the depth between child $\s<X$ and $\s^{\text{c}}=X\setminus\c$} when $X$ is even: 
		pick $m\in\Q$ with  $-\delta_\s\!\le\! m\le \delta_{\s^\text{c}}$ (if $|\s|\!=\!1$ there is no lower bound on $m$, and similarly for $\s^\text{c}$) and set 
		$X'\!=\!X$, $\Sigma'\!=\!\Sigma\cup\{\s,\s^{\text{c}}\}$, $d'_{X'}\!=\!d_{X}$, 
		$d'_{\t}\!=\!d_{\t}\!+\!m$ for proper clusters $\t\subseteq\s$,  
		$d'_{\t}\!=\!d_{\t}\!-\!m$ for proper clusters $\t\subseteq\s^{\text{c}}$. 
		Here we consider $\delta_{\s^{\text{c}}}\!=\!0$ if $\s^{\text{c}}\notin\Sigma,$ 
		and remove $\s^\text{c}$ from $\Sigma'$ if $\delta'_{\s^\text{c}}\!=\!0$.
		
	\end{enumerate}
For a pictorial description of these moves see Example \ref{eg:moves}.
\end{defn}

\begin{thm}\label{thm:isomhaveequiv}
	If $C_1$ and $C_2$ are isomorphic hyperelliptic curves over $K$, then their cluster pictures are equivalent. Furthermore, if a cluster picture $\Sigma'$ is equivalent to $\Sigma_{C_1}$, then there is a $\bar{K}$-isomorphic hyperelliptic curve $C'/\bar{K}$ with $\Sigma_{C'}\cong\Sigma'$.
\end{thm}

\begin{thm}\label{thm:equiviffisomcores}
	Let $C_1$ and $C_2$ be semistable hyperelliptic curves over $K$. Then $\Sigma_{C_1}$ and $\Sigma_{C_2}$ are equivalent if and only if the BY trees $T_{C_1}$ and $T_{C_2}$ are isomorphic. 
\end{thm}

It turns out that, provided $|k|>2g+1$, every equivalence class of cluster pictures of semistable hyperelliptic curves has an `almost canonical' representative.

\begin{thm}\label{thm:nearlywmodel}
	Let $C'/K$ be a semistable hyperelliptic curve and suppose that $|k|>2g+1$. Then there is a $K$-isomorphic curve $C:y^2=f(x)$ with $f(x)\in\cO_K[x]$, $\textrm{deg}(f)=2g+2$ such that
	\begin{enumerate}
		\item $d_\cR=0$, 
		\item the cluster picture of $C$ has no cluster of size $>g+1$ other than $\cR$ and
		\item either there is at most one cluster in $\Sigma_{C}$ of size $g+1$ and $v(c)=0$, or $\Frob$ swaps two clusters of size $g+1$ and $v(c)\in\{0,1\}$.
	\end{enumerate}
	Furthermore, if $C'$ has even genus, then we may replace (3) by the following:
	\begin{enumerate}
		\item[(3')] either $v(c)=0$ and there is no cluster of size $g+1$, or $v(c)\in\{0,1\}$ and there are two clusters of size $g+1$ with equal depths.
	\end{enumerate}
	In the even genus case, any other $K$-isomorphic curve satisfying (1), (2), and (3') has the same cluster picture and valuation of leading term as $C$.
\end{thm}

For a semistable hyperelliptic curve $C/K$, to practically use BY trees to find the canonical representative of the equivalence class of $\Sigma_C$, attach an open yellow edge to the ``\emph{centre}'' (\cite[Definition 5.13]{hyble}) of $T_C$. For a more detailed explanation of this see Remarks \ref{rem:canonical_clusterpic_for_BY} and \ref{rem:openBY_centredBY}.

\newpage

\begin{example}
	Consider the hyperelliptic curve $C\colon y^2=x^6-1$ over $\QQ_p$, for some prime $p\neq 3$, where $\Sigma_C=
	\clusterpicture            
	\Root[A] {} {first} {r1};
	\Root[A] {} {r1} {r2};
	\Root[A] {} {r2} {r3};
	\Root[A] {} {r3} {r4};
	\Root[A] {} {r4} {r5};
	\Root[A] {} {r5}{r6};
	\ClusterLDName c1[][0][] = (r1)(r2)(r3)(r4)(r5)(r6);
	\endclusterpicture$.
	By Definition \ref{def:equivclusterpics} (1) we may increase the depth of $\cR$ by $m=\frac{1}{3}$ to obtain an equivalent cluster picture. Theorem \ref{thm:isomhaveequiv} tells us there is some $\bar{\QQ}_p$-isomorphic curve $C'/\bar{\QQ}_p$ with this cluster picture. In particular, we find that under the transformations $x=x'/p^{1/3}$ and $y=y'/p$, $C$ is ${\QQ_p}(\sqrt[3]{p})$-isomorphic to $C'/{\QQ_p}(\sqrt[3]{p}):y'^2=x'^6-p^2$.
\end{example}
    
\begin{example}\label{eg:moves}
	Consider the hyperelliptic curve $C/\QQ_7 : y^2=(x^2-1)(x^4-7^8)$. It has cluster picture
	$\Sigma_C=
	\clusterpicture            
	\Root[A] {} {first} {r1};
	\Root[A] {} {r1} {r2};
	\Root[A] {2} {r2} {r3};
	\Root[A] {} {r3} {r4};
	\Root[A] {} {r4} {r5};
	\Root[A] {} {r5}{r6};
	\ClusterLDName c1[][2][] = (r3)(r4)(r5)(r6);
	\ClusterLDName c2[][0][] = (r1)(r2)(c1);
	\endclusterpicture 
	\textrm{ with }
	\cR = \{1, -1, 7^2, -7^2, 7^2i, -7^2i\}.$
	Definition \ref{def:equivclusterpics} gives us that the equivalence class of $\Sigma_C$ is as follows:
	\vspace{-11px}
\begin{figure}[h]
	\centering
	\begin{tikzpicture}
	\draw (0,0) node[left, font=\small]{$
		\scalebox{0.8}{\clusterpicture            
			\Root[A] {} {first} {r1};
			\Root[A] {} {r1} {r2};
			\Root[A] {3} {r2} {r3};
			\Root[A] {} {r3} {r4};
			\Root[A] {} {r4} {r5};
			\Root[A] {2} {r5}{r6};
			\ClusterLDName c1[][2][] = (r1)(r2);
			\ClusterLDName c2[][n][] = (r3)(r4)(r5)(c1);
			\ClusterLDName c3[][][] = (r6)(c2);
			\endclusterpicture 
		}$};
	
	\draw (2.5,0) node[left, font=\small]{$
		\scalebox{0.8}{\clusterpicture            
			\Root[A] {} {first} {r1};
			\Root[A] {} {r1} {r2};
			\Root[A] {3} {r2} {r3};
			\Root[A] {} {r3} {r4};
			\Root[A] {} {r4} {r5};
			\Root[A] {} {r5}{r6};
			\ClusterLDName c1[][2][] = (r1)(r2);
			\ClusterLDName c2[][][] = (r5)(r6)(r3)(r4)(c1);
			\endclusterpicture  
		}$};

	\draw (2.5,-1) node[left, xshift = -0.15cm, 
	font=\small]{$
	\scalebox{0.8}{\clusterpicture            
		\Root[A] {} {first} {r1};
		\Root[A] {} {r1} {r2};
		\Root[A] {3} {r2} {r3};
		\Root[A] {} {r3} {r4};
		\Root[A] {} {r4} {r5};
		\ClusterLDName c1[][2][] = (r1)(r2);
		\ClusterLDName c2[][][] = (r3)(r4)(r5)(c1);
		\endclusterpicture 
	}$};
	
	\draw (5.2,0) node[left, font=\small]{$
		\scalebox{0.8}{\clusterpicture            
			\Root[A] {} {first} {r1};
			\Root[A] {} {r1} {r2};
			\Root[A] {4} {r2} {r3};
			\Root[A] {} {r3} {r4};
			\Root[A] {} {r4} {r5};
			\Root[A] {} {r5}{r6};
			\ClusterLDName c1[][b][] = (r3)(r4)(r5)(r6);
			\ClusterLDName c2[][a][] = (r1)(r2);
			\ClusterLDName c3[][][] = (c1)(c2);
			\endclusterpicture  
		}$};

	\draw (7.7,0) node[left, font=\small]{$
	\scalebox{0.8}{\clusterpicture            
		\Root[A] {} {first} {r1};
		\Root[A] {} {r1} {r2};
		\Root[A] {2} {r2} {r3};
		\Root[A] {} {r3} {r4};
		\Root[A] {} {r4} {r5};
		\Root[A] {} {r5}{r6};
		\ClusterLDName c1[][2][] = (r3)(r4)(r5)(r6);
		\ClusterLDName c2[][][] = (r1)(r2)(c1);
		\endclusterpicture  
	}$};

	\draw (7.7,-1) node[left, xshift=-0.15cm, font=\small]{$
	\scalebox{0.8}{\clusterpicture            
		\Root[A] {} {first} {r1};
		\Root[A] {2} {r1} {r2};
		\Root[A] {} {r2} {r3};
		\Root[A] {} {r3} {r4};
		\Root[A] {} {r4} {r5};
		\ClusterLDName c1[][2][] = (r2)(r3)(r4)(r5);
		\ClusterLDName c2[][][] = (r1)(c1);
		\endclusterpicture 
	}$};

	\draw (10.4,0) node[left, font=\small]{$
	\scalebox{0.8}{ \clusterpicture            
		\Root[A] {} {first} {r1};
		\Root[A] {2} {r1} {r2};
		\Root[A] {2} {r2} {r3};
		\Root[A] {} {r3} {r4};
		\Root[A] {} {r4} {r5};
		\Root[A] {} {r5}{r6};
		\ClusterLDName c1[][2][] = (r3)(r4)(r5)(r6);
		\ClusterLDName c2[][n][] = (r2)(c1);
		\ClusterLDName c3[][][] = (r1)(c2);
		\endclusterpicture 
	}$};

	\draw[->](-0.05,0) -- node[below,font=\tiny]{(4)} ++ (0.6,0);
	\draw[->](2.45,0) -- node[below,font=\tiny]{(4)} ++ (0.6,0);
	\draw[->](5.15,0) -- node[below,font=\tiny]{(4)} ++ (0.6,0);
	\draw[->](7.65,0) -- node[below,font=\tiny]{(4)} ++ (0.6,0);
	\draw[->](1.5,-0.35) -- node[right,font=\tiny]{(3)} ++ (0,-0.4);
	\draw[->](6.7,-0.35) -- node[right,font=\tiny]{(3)} ++ (0,-0.4);
	\end{tikzpicture}
\end{figure}\vspace{-11px}\\
Here the top clusters' depths are not written as these can take any value, due to Definition \ref{def:equivclusterpics} (1), and $n,a,b \in \Q_{>0}$ with $a+b=2$. Vertical lines indicate that a root has been added or removed as in Definition \ref{def:equivclusterpics} (2) and (3). Horizontal lines indicate that the depth of a child $\c<\cR$ has been redistributed to $\cR\setminus\c$ as described in Definition \ref{def:equivclusterpics} (4).

Let $C_1/\QQ_7:y^2=(x^2-7^4)(x^4-1)$, this is isomorphic to $C$ over $\QQ_7$ and has
\vspace{-5px}
$$\Sigma_{C_1}=\clusterpicture            
\Root[A] {} {first} {r1};
\Root[A] {} {r1} {r2};
\Root[A] {3} {r2} {r3};
\Root[A] {} {r3} {r4};
\Root[A] {} {r4} {r5};
\Root[A] {} {r5}{r6};
\ClusterLDName c1[][2][] = (r1)(r2);
\ClusterLDName c2[][0][] = (r5)(r6)(r3)(r4)(c1);
\endclusterpicture.$$
So, $\Sigma_{C_1}$ is in the equivalence class of $\Sigma_C$,  verifying the first part of Theorem \ref{thm:isomhaveequiv}. 

Consider the transformation $x \to \frac{\sqrt[5]{7}}{x+\sqrt[5]{7}}$. It gives a model $C_2$ for $C/\Q_7(\sqrt[5]{7})$ with roots $ \frac{\sqrt[5]{7}}{1+\sqrt[5]{7}},  \frac{\sqrt[5]{7}}{-1+\sqrt[5]{7}},  \frac{1}{1+\sqrt[5]{7}^9},  \frac{1}{1-\sqrt[5]{7}^9},  \frac{1}{1+i\sqrt[5]{7}^9},  \frac{1}{1-i\sqrt[5]{7}^9}$, and cluster picture  
$$\clusterpicture            
\Root[A] {} {first} {r1};
\Root[A] {} {r1} {r2};
\Root[A] {4} {r2} {r3};
\Root[A] {} {r3} {r4};
\Root[A] {} {r4} {r5};
\Root[A] {} {r5}{r6};
\ClusterLDName c1[][\frac 15][] = (r1)(r2);
\ClusterLDName c2[][\frac 95][] = (r3)(r4)(r5)(r6);
\ClusterLDName c3[][0][] = (c1)(c2);
\endclusterpicture.$$
This illustrates how to obtain the middle picture with $a\!=\! \frac 15$ and $b\!=\! \frac 95$ over~$\bar{\Q}_7$. 

All of $C,C_1,$ and $C_2$ have the following BY tree:
\smash{\raise -6pt\hbox{\begin{tikzpicture}[scale=\GraphScale]
\BlueVertices
\Vertex[x=1.50,y=0.000,L=1]{1};
\Vertex[x=0.000,y=0.000,L=\relax]{2};
\YellowEdges
\Edge(1)(2)
\TreeEdgeSignS(1)(2){0.5}{4}\TreeEdgeSignN(1)(2){0.5}{}
\end{tikzpicture}}}.
Indeed, so does any other hyperelliptic curve with a cluster picture in the equivalence class of $\Sigma_C$. Conversely, any hyperelliptic curve $C'$ with BY tree $T_{C'}=T_C$ would need to have its cluster picture %
 in the equivalence class of $\Sigma_C$. 
\end{example}

\begin{remark}\label{rem:mobiusmaps}
	It is useful to note that the steps described in Definition \ref{def:equivclusterpics} can be made by applying the following M\"obius transformations to the roots in $\cR$:
	\begin{enumerate}
		\item $\phi(z)=\pi^m z$ (for $m\in\Q$),
		\item $\phi(z)=\frac{1}{z}$ (after first shifting by $z_{\cR}\in K$, i.e. applying $\phi'(z)=z-z_{\cR}$),
		\item $\phi(z)=\frac{1}{z}$ (first shifting by $r$ and using (1) to assume that $z_{\cR}=r=d_\cR=0$),
		\item $\phi(z)=\frac{\pi^a}{z}$ (first scaling so $d_\cR=0$, and shifting so $v(r)=a$  for $r \in \s$).	\end{enumerate}
\end{remark}

\begin{example}
	By Theorem \ref{thm:nearlywmodel} any semistable genus $2$ hyperelliptic curve, where $|k|\!>\!2g\!+\!1$, has a model with one of the following cluster pictures with $m,n,t \in \Z$:
	\vspace{-12px}
	\begin{figure}[h]
		\centering
		\begin{tikzpicture}
		\draw (0.2,0) node[left, font=\small]{$
			\scalebox{0.8}{\clusterpicture            
				\Root[A] {} {first} {r1};
				\Root[A] {} {r1} {r2};
				\Root[A] {} {r2} {r3};
				\Root[A] {} {r3} {r4};
				\Root[A] {} {r4} {r5};
				\Root[A] {} {r5}{r6};
				\ClusterLDName c1[][][] = (r1)(r2)(r3)(r4)(r5)(r6);
				\endclusterpicture 
			}$};
		
		\draw (2.8,0) node[left, font=\small]{$
			\scalebox{0.8}{\clusterpicture            
				\Root[A] {} {first} {r1};
				\Root[A] {} {r1} {r2};
				\Root[A] {} {r2} {r3};
				\Root[A] {4} {r3} {r4};
				\Root[A] {} {r4} {r5};
				\Root[A] {} {r5}{r6};
				\ClusterLDName c1[][t][] = (r4)(r5)(r6);
				\ClusterLDName c2[][t][] = (r1)(r2)(r3);
				\ClusterLDName c3[][][] = (c1)(c2);
				\endclusterpicture  
			}$};
		
		\draw (5.3,0) node[left, xshift = -0.15cm, 
		font=\small]{$
			\scalebox{0.8}{\clusterpicture            
				\Root[A] {} {first} {r1};
				\Root[A] {} {r1} {r2};
				\Root[A] {} {r2} {r3};
				\Root[A] {} {r3} {r4};
				\Root[A] {2} {r4} {r5};
				\Root[A] {} {r5}{r6};
				\ClusterLDName c1[][\frac{n}{2}][] = (r5)(r6);
				\ClusterLDName c2[][][] = (r1)(r2)(r3)(r4)(c1);
				\endclusterpicture  
			}$};
		
		\draw (8,0) node[left, font=\small]{$
			\scalebox{0.8}{\clusterpicture            
				\Root[A] {} {first} {r1};
				\Root[A] {} {r1} {r2};
				\Root[A] {} {r2} {r3};
				\Root[A] {4} {r3} {r4};
				\Root[A] {2} {r4} {r5};
				\Root[A] {} {r5}{r6};
				\ClusterLDName c1[][\frac{n}{2}][] = (r5)(r6);
				\ClusterLDName c2[][t][] = (r4)(c1);
				\ClusterLDName c3[][t][] = (r1)(r2)(r3);
				\ClusterLDName c4[][][] = (c2)(c3);
				\endclusterpicture  
			}$};

		\draw (1.35,-0.7) node[left, font=\small]{$
			\scalebox{0.8}{ \clusterpicture            
				\Root[A] {} {first} {r1};
				\Root[A] {} {r1} {r2};
				\Root[A] {2} {r2} {r3};
				\Root[A] {} {r3} {r4};
				\Root[A] {4} {r4} {r5};
				\Root[A] {} {r5}{r6};
				\ClusterLDName c1[][\frac{m}{2}][] = (r5)(r6);
m				\ClusterLDName c2[][\frac{n}{2}][] = (r3)(r4);
				\ClusterLDName c3[][][] = (r1)(r2)(c1)(c2);
				\endclusterpicture 
			}$};
		
		\draw (4.1,-0.7) node[left, font=\small]{$
			\scalebox{0.8}{ \clusterpicture            
				\Root[A] {} {first} {r1};
				\Root[A] {} {r1} {r2};
				\Root[A] {4} {r2} {r3};
				\Root[A] {} {r3} {r4};
				\Root[A] {4} {r4} {r5};
				\Root[A] {} {r5}{r6};
				\ClusterLDName c1[][\frac{t}{2}][] = (r5)(r6);
				\ClusterLDName c2[][\frac{m}{2}][] = (r3)(r4);
				\ClusterLDName c3[][\frac{n}{2}][] = (r1)(r2);
				\ClusterLDName c4[][][] = (c3)(c1)(c2);
				\endclusterpicture 
			}$};
		
		\draw (7.2,-0.7) node[left, font=\small]{$
			\scalebox{0.8}{\clusterpicture            
				\Root[A] {} {first} {r1};
				\Root[A] {2} {r1} {r2};
				\Root[A] {} {r2} {r3};
				\Root[A] {6} {r3} {r4};
				\Root[A] {2} {r4} {r5};
				\Root[A] {} {r5}{r6};
				\ClusterLDName c1[][\frac{m}{2}][] = (r5)(r6);
				\ClusterLDName c2[][t][] = (r4)(c1);
				\ClusterLDName c3[][\frac{n}{2}][] = (r2)(r3);
				\ClusterLDName c4[][t][] = (r1)(c3);
				\ClusterLDName c5[][][] = (c2)(c4);
				\endclusterpicture  
			}$};
		
		\end{tikzpicture}
	\end{figure}
\end{example}

\newpage
\section{Appendix: Minimal discriminant and BY trees (semistable case)}

Throughout this section, it is assumed that $C$ is \underline{semistable}. We give a proof for how to read off $v(\Delta_C^{\min})$ from the BY tree $T_C$ associated to $C$. 

\begin{definition}
	\label{def:genustree}
	For a connected subgraph $T$ of a BY tree, we define a genus function by 
	\(g(T) = \#(\textrm{connected components of the blue part}) - 1 + \sum_{v \in V(T)} g(v).
	\)
\end{definition}
Note that $g(T_C) = g$ as per Lemma \ref{genuscurve}.

\begin{definition}	
	\label{def:centredBYtree}
	If there is an edge $e \in E(T_C)$ such that both trees in $T_C\setminus\{e\}$ have equal genus (i.e. genus $\lfloor \frac{g}{2}\rfloor$), then we insert a genus-0 vertex $z_T$ on the midpoint of $e$, colour it the same as $e$, and call it the {\em centre} of $T_C$. Otherwise, choose $z_T \in V(T_C)$  such that all trees in $T_C\setminus \{z_T\}$\footnote{$T_C\setminus\{z_T\}$ is obtained from $T_C$ by removing $z_T$ together with the incident edges.} have genus smaller than $g/2$.
	In both cases, the {\em centred BY tree} $T^*_C$ is the tree with vertex set $V(T^*_C) = V(T_C) \cup \{z_T\}$; we denote by $\preceq$ the partial order on $V(T^*_C)$ with maximal element $z_T$.  For a vertex $v \in V(T^*_C)$, we say that the vertex connected to $v$ lying on the path to the centre of $T^*_C$ is its {\em parent}. All other vertices connected to $v$ are called {\em children} of $v$. The centre itself does not have a parent.
\end{definition}

\begin{definition}
	Define a weight function on the vertex set $V(T_C)$ by
	\[
	s(v) = \begin{cases}
	2g(v) + 2 - \#\textrm{blue edges at } v & \textrm{if }v \textrm{ is blue,}\\
	0 & \textrm{if }v \textrm{ is yellow.}\\
	\end{cases}
	\]
	For a connected subgraph $T$ of $T_C$, we set
	\(
	s(T) = \sum_{v\in T} s(v).
	\)
\end{definition}

\begin{remark}
	\label{rem:centre_equiv_def}
	Observing that $s(T_C) = 2g+2$, it follows from \cite[Lemma 5.12]{hyble} that exactly one of the following is true.
	\begin{itemize}
		\item There is a unique vertex $v \in V(T_C)$ with the property that  $s(T)<g+1$ for all trees in $T_C \setminus \{v\}$. 
		\item There is a unique edge $e \in E(T_C)$ with the property that $s(T) = g+1$ for both trees in $T_C \setminus \{e\}$. 
	\end{itemize}
	Further $g(T) = \lfloor \frac{s(T) - 1}{2} \rfloor$ for any connected subgraph $T$ of a BY tree (see \cite[Remark 5.14]{hyble}). 
	This shows that the centre of a BY tree is indeed well defined.
\end{remark}

\begin{definition}	
	Define a weight function on $V(T^*_C)$  by 
	\(
	S(v) = \sum_{v'\preceq v}  s(v'). 
	\)

	For each $v \neq z_T$, write $e_v$ for the edge connecting $v$ with its parent and let 
	\[
	\delta_v = 
	\begin{cases}
	\textrm{length}(e_v) &\textrm{if } e_v \textrm{ is blue},\\
	1/2 \cdot \textrm{length}(e_v) & \textrm{if }e_v \textrm{ is yellow}. \end{cases}\]
\end{definition}

\begin{theorem}
	\label{thm:discriminant_by}
	Let $T^*_C$ be the centred BY tree associated to $C$. Suppose $|k|\!>\!2g\!+\!1$. Then the valuation of the minimal discriminant of $C$ is given by
	\[
	v(\Delta_C^{\min}) 
	= E\cdot (4g+2) + \sum_{v \neq z_T} \delta_v S(v)(S(v)-1),
	\]
	where $E = 0$ unless $z_T$ has exactly two children $v_1,v_2$ with $S(v_1)=S(v_2)=g+1$ that are permuted by Frobenius and $\delta_{v_i}(g+1)$ is odd for $i \in \{1,2\}$. In this case $E=1$.
\end{theorem}

\begin{proof}
	Let $\Sigma = \Sigma_C$ be the cluster picture associated to $C$, see Definition \ref{def:abstractclusterpic} for the definition of abstract cluster pictures. We associate a cluster picture $\Sigma_1 = (\Sigma_1, X_1,d_1)$ to the centred tree $T^*_C$ in the following way.
	
	For every vertex $v \in T^*_C$, define
	\[
	\s_v = \bigcup_{v' \prec v \textrm{ maximal}} \s_{v'} \cup \bigcup_{i=1}^{s(v)} \{r_{v,i}\},
	\]
	where $\{r_{v,i}\}$ are singletons. For $v \neq z_T$, the relative depth of the cluster $\s_v$ is given by  $\delta_{\s_v} = \delta_{v}$.  We have $\s_{z_{T}} = X_1$ and assign to it depth $d_{X_1} = 0$.
	
	The construction of the cluster picture coincides with Construction 4.15 in \cite{hyble}, although phrased in a slightly different language (cf. Remark \ref{rem:openBY_centredBY}). Therefore the BY tree associated to this cluster picture is $T_C$. 
	Moreover, it is clear from the construction that for every vertex $v \in V(T^*_C)$, we have $S(v) = |\s_v|$ and that every cluster $\s \neq \cR$ has size $\leq g+1$. 
	
	From Theorems \ref{thm:isomhaveequiv} and \ref{thm:equiviffisomcores} it follows that there is a hyperelliptic curve $C_1: y^2 = f_1(x)$ which is $\bar{K}$-isomorphic to $C$ and has cluster picture $\Sigma_1$. Applying the formula of Theorem \ref{thm:discriminant}, we find that 
	\begin{equation}
	\label{eqn:disc_by_sigma1}
	v(\Delta_{C_1}) = v(c_1) (4g+2)+ \sum_{v \neq  z_T} \delta_v S(v)(S(v)-1),
	\end{equation}
	where $c_1$ denotes the leading coefficient of $f_1$.
	 We will now modify the cluster picture $\Sigma_1$ in order to find a curve $C_2$ which is isomorphic to $C$ over $K$. 
	
	Let us first consider the case where $z_T \in V(T_C)$. In that case we moreover have that $|\s_v|<g+1$ for all clusters $\s_v \neq \cR$. It follows from Theorem \ref{thm:nearlywmodel} and the uniqueness of the centre $z_T$ that there is a $K$-isomorphic curve $C_2: y^2=f_2(x)$ with cluster picture $\Sigma_{C_2} = \Sigma_1$ and $v(c_2) = 0$, where $c_2$ is the leading coefficient of $f_2$. This completes the first case.
	
	Now consider the case $z_T\notin V(T_C)$. Then $\cR = \s_1 \sqcup \s_2$, where $|\s_1|=|\s_2|=g+1$. In this case, it might be necessary to redistribute depth between the clusters $\s_1$ and $\s_2$, see Definition \ref{def:equivclusterpics}. However this does not change the valuation of the discriminant since the two clusters have equal size. Hence we may still use Equation \ref{eqn:disc_by_sigma1}.
	If the two clusters $\s_1$ and $\s_2$ are not permuted by Frobenius, let $\Sigma_2$ be the cluster picture obtained by redistributing all depth from $\s_1$ to $\s_2$ (or vice versa). It follows from \ref{thm:nearlywmodel} that there is a $K$-isomorphic curve $C_2$ with this cluster picture and $v(c_2)=0$.
	
	In the other case, where the two clusters $\s_1,\s_2$ are permuted by Frobenius, we know that there exists a curve $C_2$ which is isomorphic to $C$ with $v(c_2)\in\{0,1\}$ and $\Sigma_{C_2} = \Sigma_2$, where $\Sigma_2$ is obtained from $\Sigma_1$ by shifting depth $m \in \QQ$ from $\s_1$ to $\s_2$. It remains to compute $v(c_2)$. For that purpose denote by $\delta_1 = \delta_{\s_1}-m$ and $\delta_2 = \delta_{\s_2}+m$ the new relative depths of the clusters $\s_1$ and $\s_2$. It follows from the semistability criterion (Theorem~\ref{semistability criterion subsection}) that $v(c_2) \equiv \delta_1(g+1) \equiv \delta_2(g+1)\pmod 2$. If $g$ is odd, this implies $v(c_2)=0$. On the other hand if $g$ is even, we may assume that $\delta_1 =\delta_2$ (see Theorem \ref{thm:nearlywmodel}). Hence $v(c_2)=1$ if and only if $\delta_{\s_i}(g+1)$ is odd. 
	
	In all cases, we have seen that there is a $K$-isomorphic curve for which $v(c_2)=E$ and the valuation of the discriminant is given by the formula in the theorem. By Theorem \ref{thm:minimal_discriminant}, this is indeed the valuation of the minimal discriminant.
\end{proof}

\begin{remark} \label{rem:canonical_clusterpic_for_BY}
	The cluster picture $\Sigma_1$ constructed in the proof presents a canonical representative for the equivalence class of the cluster picture associated to $C$ (see Definition \ref{def:equivclusterpics}).
\end{remark}

\begin{remark} \label{rem:openBY_centredBY}
	Instead of working with the centred BY tree $T^*_C$, one could also consider the open BY tree (\cite[Definition 3.21]{hyble}) obtained by gluing an open yellow edge to the centre of $T_C$. The order on the vertices of this tree and the construction of the cluster picture $\Sigma_1$ described in the proof of the theorem then coincide exactly with the definitions in Construction 4.15 in \cite{hyble}.
\end{remark}



\end{document}